\documentclass[12pt, reqno, letterpaper]{amsart}
\usepackage[margin=1in]{geometry}
\usepackage{lmodern}
\usepackage[T1]{fontenc}
\usepackage{amssymb, amsmath, amsfonts, amsthm, mathrsfs, tikz-cd, hyperref}
\usepackage{cite}
\usepackage{dutchcal, bm}
\usepackage{fancyhdr}
\usepackage{color}

\usepackage{subfiles}
\usepackage{xr}

\definecolor{linkcol}{rgb}{0,0,0.6}
\definecolor{hrefcol}{rgb}{0,0.5,0.5}
\definecolor{citecol}{rgb}{0.6,0,0.6}
\definecolor{amber}{rgb}{0.5,0.4,0.0}
\hypersetup{colorlinks=true,
            linkcolor=linkcol,
            filecolor=magenta,      
            urlcolor=hrefcol,
            citecolor=citecol}
\tikzcdset{arrow style=tikz, diagrams={>=stealth}}

\makeatletter
\def\l@subsection{\@tocline{2}{0pt}{4pc}{5pc}{}}
\makeatother



\makeatletter 
\renewcommand\th@plain{\slshape}
\makeatother

\newtheoremstyle{result}
  {1.7ex}
  {1.7ex}
  {\slshape}
  {0pt}
  {\scshape}
  {.}
  { }
  {\thmname{#1}\thmnumber{ #2}\thmnote{\if&#3& #3\else { }$-$ #3\fi}}

\newtheoremstyle{unique}
  {1.7ex}
  {1.7ex}
  {}
  {0pt}
  {\scshape}
  {.}
  { }
  {\thmnote{#3}}

\newtheoremstyle{other}
  {1.7ex}
  {1.7ex}
  {}
  {0pt}
  {\scshape}
  {.}
  { }
  {\thmname{#1}\thmnumber{ #2}\thmnote{\if&#3& #3\else { }$-$ #3\fi}}

\theoremstyle{result}
\makeatletter
  \ifdefined\preamble@file 
    \newtheorem{THM}{Theorem}
  \else 
    \newtheorem{THM}{Theorem}[section]
  \fi
\makeatother
\newtheorem{PROP}[THM]{Proposition}
\newtheorem{LEM}[THM]{Lemma}
\newtheorem{COR}[THM]{Corollary}

\theoremstyle{unique}

\theoremstyle{other}
\newtheorem{DEF}[THM]{Definition}

\newtheorem{CONSTR}[THM]{Construction}
\newtheorem{RMK}[THM]{Remark}
\newtheorem{RMK!}[THM]{Important Remark}

\newtheorem{EXP}[THM]{Example}

\newtheorem{CONJ}[THM]{Conjecture}

\def\bb{\mathbb}
\def\fk{\mathfrak}
\def\0{\varnothing}
\def\scr{\mathcal}
\def\<{\langle}
\def\>{\rangle}

\def\spinc{\scr S\hspace{-0.16em}{\it pin^c}}
\def\rspinc{{\it r}\hspace{-0.1em}\scr {S}\hspace{-0.16em}{\it pin^c}}
\def\rsym{{\it r}\hspace{-0.1em}\scr S}
\def\rdiv{{\it r}\hspace{0.05em}\scr D}
\def\hatP{\widehat{\scr P}}
\def\rhatP{{\it r}\hspace{-0.05em}\widehat{\scr P}}

\newcommand{\keywd}[1]{{\fontseries{b}\selectfont{\boldmath#1}}}

\newcommand{\nocontentsline}[3]{}
\let\origcontentsline\addcontentsline
\newcommand\stoptoc{\let\addcontentsline\nocontentsline}
\newcommand\resumetoc{\let\addcontentsline\origcontentsline}

\makeatletter
\ifdefined\preamble@file 
  \labelformat{subsection}{\thesubsection$^{\texttt{\jobname\@}}$}
  \labelformat{equation}{\theequation$^{\texttt{\jobname\@}}$}
  
  \labelformat{THM}{\theTHM$^{\texttt{\jobname\@}}$}
  \newcommand{\makebib}[1]{\bibliographystyle{abbrvurl}\bibliography{#1}{}}
\else 
  \numberwithin{equation}{section}
  
  \newcommand{\makebib}[1]{}
\fi
\makeatother

\title[$\widehat{\it HFR}$ with $\mathbb Z$ coefficients]{Real Heegaard Floer homology \\ with integral coefficients}
\author{Ciprian M. Bonciocat}
\address{Department of Mathematics, Stanford University, 450 Jane Stanford Way, Building 380,
Stanford, CA 94305-2125, USA.}
\email{ciprianb@stanford.edu}

\begin{document}


\begin{abstract}
  We treat the question of upgrading Guth-Manolescu \cite{GM25}'s real Heegaard Floer theory to a $\mathbb Z/2$-graded invariant defined over $\mathbb Z$, specifically for the hat version with one basepoint on each component of the branching link. We provide if-and-only-if obstructions to the existence of relative $\mathbb Z/2$-gradings and $\mathbb Z$-lifts, in terms of purely homological information associated to the 3-manifold with involution. When the corresponding obstruction vanishes, we study the set of such $\mathbb Z$-lifts, which a priori is only an invariant up to a torsor action. The obstruction theory could be of independent interest in various Floer theories more broadly, particularly Lagrangian Floer theory.
\end{abstract}
\maketitle

\setcounter{tocdepth}{2}
\tableofcontents
\newpage

\section{Introduction}\label{sec:intro}

		In 2025, Guth and Manolescu \cite{GM25} defined a real version of the Heegaard Floer homology of Ozsv\'ath and Szab\'o \cite{OS-I,OS-II,OS-III}, where the word ``real'' signifies the presence of an orientation-preserving involution $\tau$ on the 3-manifold $Y$, with non-empty fixed locus of codimension 2, and which extends to a conjugate-linear automorphism of the ${\rm Spin}^c$ structure.  This construction is motivated by earlier gauge-theoretic analogues, such as Li's real monopole Floer homology \cite{Li22}, shown later by Xiao \cite{Xi26-I} to be equivalent to Konno-Miyazawa-Taniguchi's construction \cite{KMT24}, as well as the real Seiberg-Witten invariants of Tian and Wang \cite{TW09}, later generalized by Kato \cite{Ka22}. These theories were used by Miyazawa \cite{Miy23} in order to produce an infinite family of exotic $\mathbb{RP}^2 \subset S^4$, and consequently exotic involutions on the double branched cover $\mathbb {CP}^2$, due to further work by Hughes, Kim, and Miller \cite{HKM24}, showing that the smooth structure on these $\mathbb{CP}^2$'s is standard. Other relevant work in the development of real Seiberg-Witten or monopole theory, as well as interesting applications can be found in \cite{Nak13,Nak20,KMT21,KMT24,BH24,Kuh25,Kuh26,KPT24}. Further work has been done in building the foundations for real Heegaard Floer theory as well \cite{He26, Sr26, LO26, Xi26-II, BGX26}, including Guth-Manolescu \cite{GM26}, very recently establishing naturality of the $\bb Z/2$-coefficients homology in various flavors.

		Miyazawa's original example of an exotic $\mathbb{RP}^2$ in $S^4$ comes from ambient connect-summing the roll-spin of the Fintushel-Stern knot $P(2, 3, 7)$ with a standard $\mathbb{RP}^2$. A key fact used in his argument is the trace-like formula
		\begin{equation}\label{eq:degeq}
			|{\rm deg}(\rho K \subset S^4)| = |{\rm deg}(K \subset S^3)|,
		\end{equation}
		relating an invariant of a knot in $S^3$ to an invariant of the associated roll-spun 2-knot in $S^4$, both called ``degree.'' The 3-dimensional invariant on the right equals the absolute value of the Euler characteristic of Li's $\widetilde{\it HMR}$ associated to the double branched cover of $S^3$ over $K$. The 4-dimensional invariant on the left is defined as the real Seiberg-Witten invariant of the double branched cover of $S^4$ along the roll-spun 2-knot $\rho K$. Miyazawa's formula (\ref{eq:degeq}) is an equality of numerical invariants over $\mathbb Z$ and not just $\mathbb Z/2$, and this fact ends up being crucial in his proof of exoticness, since all the relevant invariants over $\bb Z/2$ end up determined by elementary algebraic topology of the branched cover. 
		
		Guth-Manolescu's $\widehat{\it HFR}$, which is a direct analogue of Li's $\widetilde{\it HMR}$ (and conjecturally isomorphic to it) has so far only been defined over $\mathbb Z/2$, with a relative $\mathbb Z/2$-grading in certain cases. Srivastava \cite{Sr26} gives a sufficient condition for the existence of \emph{absolute} $\mathbb Z/2$-gradings in the hat version of the theory, which holds true in the case of $S^3$, and further identifies the Euler characteristic as a suitable specialization of the Alexander polynomial. The question of $\mathbb Z$-coefficients remains to be addressed, in order to hope for a Heegaard-theoretic analogue of Miyazawa's formula (\ref{eq:degeq}).
		
		In this paper, we provide the obstructions to the existence of \emph{relative} $\mathbb Z/2$-gradings and $\mathbb Z$-lifts, in the case of $\widehat{\it HFR}$ theory with exactly one basepoint on every component of the branching link $C := Y^\tau$. The obstructions are if-and-only-if in a technical sense to be made precise later (cf. {\sc Defs.~\ref{def:z2gr},~\ref{def:cohsys}}). In short, we prove that
		\begin{list}{$\cdot$}{}
			\item There exists a free abelian group $A := A(Y, \tau)$, endowed with a bilinear form $F : {\rm Sym}^2 A \to \mathbb Z/2$, such that the obstructions to $\mathbb Z/2$-gradings and $\mathbb Z$-coefficients for any particular real Heegaard diagram are expressed in terms of vanishing conditions on the form $F$;
			\item The set of choices of $\mathbb Z$-lifts of the theory forms a torsor over $A^\vee := {\rm Hom}(A, \mathbb Z/2)$ for any particular diagram, and Heegaard moves induce torsor isomorphisms between the sets of choices of $\mathbb Z/2$-gradings and $\mathbb Z$-lifts. In particular, the set of choices is only well-defined as relatively graded over $\mathbb Z/2$ or $A^\vee$, respectively;
			\item In many cases, e.g.~ when $Y' := Y / \tau$ is a $\mathbb Z/2$-homology sphere and the link $C$ is connected, the group $A$ is automatically zero, and hence the $\mathbb Z$-lift is unique.
		\end{list}
	For more precise and complete statements, see {\sc Thm.~\ref{thm:main}} listed at the bottom of this section, together with the vanishing properties listed in {\sc Thm.~\ref{thm:vanish}}, which show that it is quite difficult for the obstruction to $\mathbb Z$-coefficients to vanish in practice, though such cases do exist.

	Our main results boil down to an explicit computation of the homotopy type of a suitable stabilization of the Lagrangian Floer pathspace associated to the given Lagrangian intersection problem, as well as an explicit formula for Stiefel-Whitney classes of the difference of the two Lagrangians tangent bundles. The homotopy type ends up being a disjoint union of tori, indexed over real ${\rm Spin}^c$-structures (cf.~{\sc Prop.~\ref{prop:spcalc},~Thm.~\ref{thm:comp}}), with $\pi_1$ of each component being canonically identified with the free abelian group $A(Y, \tau)$ mentioned previously. The space of choices of $\mathbb Z/2$-gradings and $\mathbb Z$-lifts are torsors over $H^0$ and $H^1$ of the toroidal components with $\bb Z/2$-coefficients, i.e.~ $\mathbb Z/2$ and $A^\vee = {\rm Hom}(A, \bb Z/2)$, respectively.

	In this paper, we do not rigorously address any naturality claims of the choices involved, but content ourselves to proving that Heegaard moves produce \emph{some} natural identification between the sets of choices. This latter correspondence is given by adding some element in $\mathbb Z/2$ or $A^\vee$ under the torsor action. For a complete understanding of this phenomenon, one must analyze the effect of loops in the ``space'' of real Heegaard diagrams on the said choices. In particular, our result only recovers Srivastava's result \cite{Sr26} up to an ambiguity of the choice of $\mathbb Z/2$-grading. We provide reasons to believe that these torsors of choices can be made canonical, both via higher multiplications in the Fukaya category, as well as via pushforwards in $K$-theory, in the last section {\sc\S\ref{sec:abs}}, although more details would have to be provided for a rigorous mathematical proof.

	We have only addressed the hat version with one basepoint in each connected component of the branching link $C$, because the homotopy type of the components of the stable pathspace would become more complicated. Depending on which complications one decides to allow, $\bb{RP}^\infty$ and $\bb{CP}^\infty$ factors may appear, and tracking them becomes non-trivial in the computation of Stiefel-Whitney classes, although we expect this to be feasible with more work.

	\makebib{main}

\subsection{Precise formulation of the main results}\label{ssec:conv}

	We now give more rigorous statements of the definitions and theorems promised at the beginning of the Introduction, and set forth more notation to be used throughout the rest of the paper.
	
	\hypertarget{setup}{}{\sc Setup.} The whole theory will be developed with respect to a \emph{double branched cover}, which we encode by the following geometric data called $\scr Y$:
	\begin{list}{$\cdot$}{}
	\item A commutative diagram $\tau \circlearrowleft Y \overset\pi\to Y'$ of closed connected oriented 3-manifolds, where $\tau$ is a smooth involution that preserves orientation, and $\pi$ is the projection that identifies $Y'$ with $Y/\tau$, canonically inheriting the orientation from $Y$.
	\item The fixed-point locus $C \subset Y$ is required to be a closed \emph{non-empty} 1-dimensional manifold, i.e.~ a link inside $Y$. It naturally embeds as a link in the quotient $Y'$ as well, and we still use the same symbol $C$ to denote its image in $Y'$.
	\item Away from the link, we obtain an honest $2:1$ cover $Y \setminus C \twoheadrightarrow Y' \setminus C$, which is uniquely determined as a $C_2 \cong \{\pm 1\}$-principal bundle by a class $v \in H^1(Y' \setminus C; \{\pm 1\})$ which evaluates to $-1$ on a closed loop $\gamma \subset Y$ precisely when the monodromy around $\gamma$ is non-trivial.
	\item We also use $N$ and $N'$ to denote small tubular neighborhoods of $C$ in $Y$ and $Y'$ respectively, requiring $N$ to be $\tau$-invariant, and $N' = N / \tau$. The class $v$ above has the property that when restricted to any toroidal component of $\partial \bar N'$, it counts the mod-2 winding number around the meridian $\mu$. 
	\end{list}

	Now, we must explain to what extent the promised invariant $\widehat{\it CFR}$ is a well-defined invariant, whether we are talking about $\mathbb Z/2$-gradings or $\mathbb Z$-coefficient lifts. It is necessary to introduce the following notions: 

	\begin{DEF}
		If $G$ is a group, and $\scr C$ is a category, a \keywd{relatively $G$-graded collection} of objects in $\scr C$ is defined to be a $G$-torsor $X$, together with a map $T : X \to {\rm Ob}(\scr C)$. A morphism $(X_1, T_1) \to (X_2, T_2)$ is given by an isomorphism of $G$-torsors $\phi : X_1 \to X_2$, together with morphisms $T_1(x) \to T_2(\phi(x))$ in $\scr C$, for all $x \in X_1$; composition and identity morphisms are defined in the obvious way, hence producing a new category \keywd{$\scr C \uparrow G$}.
	\end{DEF} 

	\begin{DEF}
		Let $R$ be a ring, $G$ be a group, and $g \in G$ be a privileged central element. Define a \keywd{relatively $(G, g)$-graded chain complex of $R$-modules} to be a $G$-graded collection of $R$-modules $C_* \in R\textit{--Mod} \uparrow G$, equipped with a degree-$g$ differential $d : C_* \to C_{g \cdot *}$ satisfying $d^2 = 0$. Define a chain map $f$ to be a homomorphism in $R\textit{--Mod} \uparrow G$ that satisfies $d \circ f = f \circ d$, and define a chain homotopy $H$ between equigraded $f_0$ and $f_1$ to be a $g^{-1}$-graded map that satisfies $H \circ d + d \circ H = f_1 - f_0$. Define \keywd{${\bf K}^{\rm rel}(R, G, g)$} as the homotopy category of such relatively $(G, g)$-graded chain complexes.
	\end{DEF}

	The relevance of this definition is that isomorphism in ${\bf K}^{\rm rel}(R, G, g)$ is the adequate notion of equivalence up to which our invariant will be defined. The group $G$ will turn out to be abelian, and takes the form $\mathbb Z/2$ in the case of $\mathbb Z/2$-gradings, $A^\vee := {\rm Hom}(A, \bb Z/2)$ in the case of $\mathbb Z$-coefficients, and $\mathbb Z/2 \times A^\vee$ in the case of both, where:

	\begin{DEF}\label{def:A}
		We define the group \underline{$A = A(\scr Y)$} as the twisted coefficients cohomology
		\[ A := H^1(Y' \setminus N', \partial N'; \underline{\bb Z}^v),\]
		where $\underline{\mathbb Z}^v$ is the local coefficient system with fiber $\mathbb Z$ and monodromy given by the element $v \in H^1(Y' \setminus N'; \pm 1)$ introduced in the setup, viewed as a homomorphism $\pi_1(Y' \setminus N') \to \{\pm 1\}$.
	\end{DEF}

	This group ends up being isomorphic to $H^1(Y, C; \mathbb Z)^{-\tau^*}$ {\sc(Prop.~\ref{prop:bigA})}, as well as to the real version of the group of periodic domains  associated to any given real Heegaard diagram {\sc(Cor.~\ref{cor:AisrPi})}, although in the formulation above it is independent of any auxiliary choices. Several important algebraic properties of $A$ are worked out in results {\sc\ref{prop:bigA}-\ref{prop:Apr2}}, such as its freeness, as well as its vanishing for the case that $Y'$ is an $\mathbb F_2$-homology sphere and $C$ a knot. Associated to the group $A$ are the following bilinear forms, which contain the information needed to encode the obstructions to $\mathbb Z/2$-gradings and $\mathbb Z$-lifts:

	\begin{DEF}
		We define a symmetric bilinear form \underline{$F : {\rm Sym}^2 A \to \mathbb Z/2$} via the triple cup product
		\[ F(x, y) := \int_{Y' \setminus N'} x \smile y \smile v, \]
		for all $x, y \in A = H^1(Y' \setminus N', \partial N'; \underline{\bb Z}^v)$. From here, we can naturally change $F$ to obtain an alternating bilinear form \underline{$Q : \Lambda^2 A \to \bb Z/2$} on the same space:
		\[ Q(x, y) := F(x, y) - F(x, x) \cdot F(y, y), \]
	\end{DEF}

	\begin{RMK}\label{rmk:form}
		Although we only defined the form $F$ on $A$, its defining formula can be extended to make sense for any two elements $x, y \in H^1(Y' \setminus N', \partial N'; \bb F_2)$, even if they don't lie in the image of $A$ via mod 2 reduction. Another interesting observation is that the formula defining $F$ can be rewritten in a way that has no reference to $v$, as we prove later in {\sc Cor.~\ref{cor:F2F2summ}}, namely:
		\[ F(x, y) = \int x^{2} \smile y = \int x \smile y^{2}. \]
		Even though it is not apparently obvious why the second equality in this display holds, it follows from the fact that ${\rm Sq}^1(x \smile y) = x^{2} \smile y + x \smile y^{2}$, and by the vanishing of ${\rm Sq}^1 : H^2(Y'; \bb F_2) \to H^3(Y'; \bb F_2)$, thanks to the orientability of $Y'$. This bilinear form indeed agrees with $F$ when pulled back to $H^1(Y' \setminus N', \partial N'; \underline{\bb Z}^v)$, because as we shall prove later (cf. {\sc Prop.~\ref{prop:Apr1}}), all the elements in the latter satisfy
		\[ x \smile x = x \smile v ~\text{mod}~ 2. \]
		Thus, $F(x, y)$ actually does not depend on $v$ at all, but its domain of definition does.

		We also mention that the second term in the definition of $Q$ \emph{becomes zero} (i.e.~ $F = Q$ is already alternating) when $v$ admits an integral lift, or at least a lift to $\mathbb Z/4$-coefficients. Indeed, $x \smile x \in H^2(Y' \setminus N', \partial N'; \bb Z/4)$ is a 2-torsion element by the anti-symmetry of the cup product; by lifting $v$ to an element $V \in H^1(Y' \setminus N'; \bb Z/4)$, this guarantees that the triple cup product lands in $2\mathbb Z / 4 \mathbb Z$ inside $\mathbb Z/4 \cong H^3(Y' \setminus N', \partial N'; \mathbb Z/4)$, and hence dies under reduction mod 2. 
	\end{RMK}

	With all these preliminaries, we may state our main result, together with some vanishing results concerning $A$, $F$ and $Q$.

	\begin{THM}[Main result]\label{thm:main}
		Let $\scr Y$ be a branched double cover, as described in the introductory {\sc Setup} at the beginning of this subsection. The ``hat'' version $\widehat{\it CFR}(\scr Y, \mathfrak s)$ of Guth-Manolescu's invariant with one basepoint on each connected component of $C$, at a particular real ${\rm Spin}^c$-structure $\mathfrak s$, is subject to the following upgrades:
		\begin{list}{$\cdot$}{}
			\item If the diagonal terms \keywd{$F(x, x)$ are zero} for all $x \in A(\scr Y)$, then there exists a relatively $(\mathbb Z/2,1)$-graded lift $\widehat{\it CFR}_*(\scr Y, \mathfrak s; \mathbb F_2)$, distinguished notationally by the asterisk, whose underlying ungraded chain complex recovers $\widehat{\it CFR}(\scr Y, \mathfrak s; \mathbb F_2)$. Its isomorphism type in ${\bf K}^{\rm rel}(\mathbb F_2, \mathbb Z/2, 1)$ is a well-defined invariant.
			\item If the bilinear form \keywd{$Q(x, y)$ is zero} identically, then there exists a relatively $(A^\vee, 0)$-graded chain complex $\widehat{\it CFR}(\scr Y, \fk s; \bb Z[\scr O])$ of $\mathbb Z$-modules distinguished notationally by the $\mathbb Z[\scr O]$, such that each homogeneous piece lifts the Guth-Manolescu invariant $\widehat{\it CFR}(\scr Y, \fk s; \bb F_2)$. Its isomorphism type in ${\bf K}^{\rm rel}(\mathbb Z, A^\vee, 0)$ is a well-defined invariant.
		\end{list}
		In particular, the joint condition for both $\mathbb Z/2$-gradings and $\mathbb Z$-coefficient upgrades to exist is that \keywd{$F(x, y)$ is zero} identically. If so, there is a relatively $(\mathbb Z/2 \times A^\vee, (1, 0))$-graded chain complex $\widehat{\it CFR}_*(\scr Y, \fk s; \mathbb Z[\scr O])$ of $\mathbb Z$-modules, distinguished notationally by incorporating both decorations. Its obvious downgrades recover the expected invariants, and moreover its isomorphism type in ${\bf K}^{\rm rel}(\mathbb Z, \mathbb Z/2 \times A^\vee, (1, 0))$ is a well-defined invariant.
	\end{THM}

	\begin{RMK}\label{rmk:fkO}
		The notation $\mathbb Z[\scr O]$ is used in order to suggest that the invariant breaks up
		\[ \widehat{\it CFR}(\scr Y, \mathfrak s; \mathbb Z[\scr O]) \cong \bigoplus_{\mathfrak o \in \scr O} \widehat{\it CFR}(\scr Y, \mathfrak s; \mathbb Z \cdot \mathfrak o) \]
		according to ``the set'' of $\mathbb Z$-lift choices $\mathfrak o \in \scr O$, each lifting the Guth-Manolescu invariant. However, this ``set'' $\scr O$ of choices is not a well-defined invariant as such, but only up to the torsor action of $A^\vee$, so that if one were to study the effect of a loop of Heegaard moves, one could potentially encounter a non-identity identification of the set of choices with itself. Of course, if one can calculate the individual summands and see that they are all non-isomorphic, then $\scr O$ does become well-defined, cf. {\sc \S\ref{ssec:s1s2},~\S\ref{ssec:z4}} for examples. This is also trivially the case when $A = 0$, since the set $\scr O$ is a singleton. We conjecture that non-identity automorphisms can never appear, hence making $\scr O$ always well-defined, cf. last section {\sc\S\ref{sec:abs}}.
	\end{RMK}

	\begin{RMK}
		Since both the spaces in which the obstructions live, as well as the obstructions themselves, do not depend on the real ${\rm Spin}^c$-structure $\mathfrak s$, it is natural to wonder if the sets of choices in various ${\rm Spin}^c$-structures for the same real Heegaard diagram are in canonical bijection. Regarding $\mathbb Z/2$-gradings, the absoluteness of the gradings constructed by \cite{Sr26} shows that at least in that case this claim holds true, though we see no reason why this should be true in general. Likewise, it is unclear to us whether the claim holds for $\mathbb Z$-coefficients. We note that the independence of the obstructions with respect to $\mathfrak s$ is itself somewhat accidental, and we do not expect it to be true for instance in the case of $\mathbb Z$-gradings, due to private communication of Jiakai Li concerning the monopole-theoretic analogue. 
	\end{RMK}

	Some algebraic vanishing criteria for the obstructions are as follows:

	\begin{THM}[Some vanishing results]\label{thm:vanish}
		The abelian group $A$ vanishes if any of the following holds:
		\begin{list}{$\cdot$}{}
			\item $H^1(Y'; \bb F_2) = 0$ and $C$ is a knot. {\sc(Cor.~\ref{cor:YpCF2})}
			\item The involution $\tau^*$ on $H^1(Y; \bb Q)$ is trivial. {\sc(Cor.~\ref{cor:trivtauact})}
		\end{list}
		The latter turns out to be an if-and-only-if characterization, because in fact the free abelian group $A$ has the same rank as $H^1(Y; \bb Q)^{-\tau}$ {\sc(Prop.~\ref{prop:bigA})}.
		Furthermore, even when potentially $A \neq 0$, we can guarantee vanishing of some of the obstructions, in the following cases:
		\begin{list}{$\cdot$}{}
			\item The diagonal terms $F(x, x)$ vanish when the class $v \in H^1(Y' \setminus C; \bb Z/2)$ admits a lift to $\mathbb Z$, or at least $\mathbb Z/4$ coefficients. {\sc({\sc Rmk.~\ref{rmk:form}}, {\rm second paragraph})}
			\item $F$ vanishes identically when $C$ is a local link in $Y'$, or more generally contained in an $\bb F_2$-homology ball, and $v$ is supported inside this homology ball. {\sc(Cor.~\ref{cor:loclink})}
			\item $F$ vanishes identically when $H_1(Y'; \bb Z)$ has no $\mathbb Z/2$-summand. {\sc(Cor.~\ref{cor:F2F2summ})}
			\item $Q$ vanishes identically when $H_1(Y'; \bb Z)$ has at most one $\bb Z/2$-summand. {\sc(Cor.~\ref{cor:F2F2summ})}
		\end{list}
		In the last two criteria, the group $H_1(Y' \setminus C; \bb Z)$ may be used instead {\sc(Prop.~\ref{prop:2torcompl})}.
	\end{THM}

	\begin{RMK}
		Restating the last criterion, i.e.~ the one for the vanishing of $Q$, it follows that $H_1(Y'; \bb Z)$ must have at least a $\bb Z/2 \oplus \bb Z/2$ summand in order for the obstruction to $\mathbb Z$-lifts to have a chance at being nonzero. One may naturally wonder if there exists a double branched cover $\scr Y$ where $Q \not\equiv 0$ and $H_1(Y'; \bb Z) \cong \bb Z/2 \oplus \bb Z/2$. We show that such a double branched cover does indeed exist, with $Y' \cong S^3 / Q_8$, cf. {\sc Ex.~\ref{exp:heegquat}}.
	\end{RMK}

	More algebraic properties about $A$, $Q$, $F$ may be found in {\sc\S\ref{sec:alg}}, including a reformulation for the vanishing of $Q$ in terms of the vanishing of $F$ on a certain subgroup, cf. {\sc Prop.~\ref{prop:checkQ}}.

	\begin{RMK}
		Our work on $\bb Z/2$-gradings generalizes both Guth-Manolescu \cite{GM25} and Srivastava \cite{Sr26}. The former imposes the condition that $c_1(\fk s) \in H^2(Y; \bb Z)$ is divisible by 4 to get relative $\bb Z/2$-gradings; our condition requires only that $c_1(\fk s)$ cupped with $(-\tau^*)$-invariant classes in $H^1(Y,C; \bb Z)$ is divisible by 4; see {\sc Prop.~\ref{prop:obsc1}}. The latter constructs absolute $\bb Z/2$-gradings by requiring that $v \in H^1(Y' \setminus C; \bb Z/2)$ lift to a primitive integral element; we are able to produce relative gradings even when $v$ only lifts to $\bb Z/4$, cf. last paragraph of {\sc Rmk.~\ref{rmk:form}}, which can happen as we show in {\S\ref{ssec:z4}}.
	\end{RMK}

	\subsection{Outline of the paper} In {\sc\S\ref{sec:sym}}, we recall the closely related notions of (real) symmetric products, and (real) divisor spaces, establish homotopical stability relating $\pi_*, H_*, H^*$ of the former to the latter for large enough degree {\sc(Prop.~\ref{prop:sdconn})}, and compute the Stiefel-Whitney classes of the tangent bundle of the former in the stable range {\sc(Thm.~\ref{thm:wi})}. We move on to {\sc\S\ref{sec:heeg}}, where we recall the basic notions concerning real Heegaard diagrams, introduce the unstable and stable pathspaces associated to such a diagram {\sc(Def.~\ref{def:path})}, and further deduce a similar kind of homotopical stability for them {\sc(Prop.~\ref{prop:connsd})}. We show that the stable pathspace has the homotopy type of a disjoint union of tori with $\pi_1 = A$ defined above {\sc(Prop.~\ref{prop:spcalc})}. Section {\sc\S\ref{sec:rs}} recounts the various equivalent reformulations of the notion of a real ${\rm Spin}^c$-structure, and finishes the analysis of the homotopy type of the stable pathspace, by showing that it can be interpreted as the space of (real, relative) ${\rm Spin}^c$-structures on $Y$ {\sc(Thm.~\ref{thm:comp})}, so in particular $\pi_0$ of this space is given by isomorphism classes of such structures.

	In {\sc\S\ref{sec:gror}}, we give a largely self-contained treatise of $\bb Z/2$-gradings and $\bb Z$-coefficients in the context of Lagrangian Floer theory, though at least part of the machinery would apply equally well to any other Floer theory, such as monopole and instanton theories. The obstruction theory we develop is if-and-only-if, in the sense that gradings and orientations are required to respect homotopies of Whitney disks, even when they contribute no holomorphic representatives; the obstructions end up being $w_1$ and $w_2$ of a certain $K$-theory class on the connected component of interest of the pathspace {\sc(Thms.~\ref{thm:obs},~\ref{thm:invar})}. The work appears to be new in its formulation and in its degree of generality, although it expands on other earlier ideas, such as coupled Spin \cite{AM25} and relative Spin structures \cite{Fuk00}, and is related to the modern approach to Floer homotopy theory. The theory has a distinct combinatorial flavor, reminiscent of \cite{OS-II}, and is amenable to a diagram-level understanding.

	The proof of the main result {\sc Thm.~\ref{thm:main}} is presented in {\sc\S\ref{sec:prf}}. Sections {\sc\S\ref{sec:alg}} and {\sc\S\ref{sec:ex}} treat algebraic properties of $A, F, Q$, and then various examples, respectively, and can be read almost independently of the earlier proof sections. The final section {\sc\S\ref{sec:abs}} is somewhat speculative in character, but establishes good reasons to believe that the set of choices of either $\bb Z/2$-gradings or $\bb Z$-coefficients could be made canonical by means of higher Fukaya operations, and likely has an algebro-topological interpretation via Atiyah-Bott-Shapiro pushforwards in $\it KR$-theory, due to the analogy with (real) monopole Floer homology. A precise formulation of these ideas can be found in {\sc Conj.~\ref{conj:abs}}, which treats the more advanced notion of polarization class, from which the obstructions to $\bb Z/2$-gradings and $\bb Z$-coefficients could be extracted, as well as $\bb Z$-gradings or ring spectrum coefficients.

\subsection{Acknowledgements}
  I would like to thank my doctoral advisor Ciprian Manolescu for his guidance and close involvement in this project, as well as David Baraglia, Gary Guth, Judson Kuhrman, Jiakai Li, and Eha Srivastava, for many useful conversations and comparisons between the results established in this paper and analogous developments in both real Heegaard Floer and monopole settings. I am also greatly indebted to William R. and Sara Hart Kimball for sponsoring my Stanford Graduate Fellowship, as well as to the Simons Foundation.

\section{Symmetric products and divisor spaces}\label{sec:sym}

	In this section, we review the notions of symmetric products and divisor spaces for surfaces, including the appropriate real versions, as well as homotopical stability results that will be needed later.

	\begin{DEF}\label{def:symdiv}
		Let $X$ be a finite CW complex. We define the \keywd{$d$-fold symmetric power} ${\scr S}^d X$ as the quotient
		\( {\scr S}^d X := X^{\times d} / \mathfrak S_d \)
		of the natural action of the symmetric group $\mathfrak S_d$ on the $d$-fold Cartesian product, equipped with the quotient topology. 
		
		Likewise, we define the \keywd{divisor group} ${\scr D} X$ set-theoretically as the free abelian group on $X$, topologized with the final topology induced by the collection of maps $X^{\times a} \times X^{\times b} \to \scr D X$ taking $(x_1, \ldots, x_a; y_1, \ldots y_b)$ to $x_1 + \cdots + x_a - y_1 - \cdots - y_b \in \scr D X$ (this agrees with the construction of \cite{Mc69} and \cite{Dr12} for based spaces after adding a disjointed basepoint.) Define the \keywd{degree map} ${\rm deg} : \scr D X \to \mathbb Z$ as the induced map from the terminal morphism $X \to *$, and use the notation $\scr D^d X$ for ${\rm deg}^{-1} (d)$.

		When $X$ is equipped with an involution $\tau$, we use the notation $\rsym^d X$ and $\rdiv^d X$ to indicate the \keywd{real symmetric power}, resp. \keywd{real divisor group}, defined as the fixed-point-set of the induced involution $\tau$ on ${\scr S}^d X$ and ${\scr D}^d X$, respectively.
	\end{DEF}

	\begin{RMK}
		The nomenclature ``divisor group'' for the free abelian group is non-standard, though we use it in this paper since $X$ will almost always be a Riemann surface (with complex-conjugate involution), and moreover the divisor-theoretic interpretation will be significant.
	\end{RMK}

	\begin{EXP}
		There is a natural homeomorphism $\scr S^d \mathbb C \cong \mathbb C^d$ given by sending $[z_1, \ldots, z_d]$ to the coefficients of the monic polynomial $\prod_{i=1}^d (z-z_i) \in \mathbb C[z]$. When $\mathbb C$ is endowed with the conjugation involution, we have $\rsym^d \mathbb C \cong \mathbb R^d$, corresponding to monic polynomials of degree $d$ in $\mathbb R[z]$.
	\end{EXP}

	\begin{EXP}
		In the case that $X$ is a Riemann surface, the symmetric power $\scr S^d X$ naturally becomes a complex manifold, because near every point $k_1 p_1 + \cdots + k_\ell p_\ell$, the manifold locally looks like $\scr S^{k_1} \mathbb C \times \cdots \scr S^{k_\ell} \mathbb C$, which is a naturally identified with $\mathbb C^{k_1 + \cdots + k_\ell} \cong \mathbb C^d$ in virtue of the previous example. If $X$ is further endowed with a conjugate-linear biholomorphism, the space $\rsym^d X$ is a totally real submanifold of half the dimension.
	\end{EXP}

	There are natural inclusions $\scr S^dX \hookrightarrow \scr D^d X$ and $\rsym^d X \hookrightarrow \rdiv^d X$. In what follows, we investigate the connectivity properties of these maps in the case that $X$ is an open Riemann surface. A slightly lesser known version of the Dold-Thom theorem states that $\scr D X$ has the homotopy type of a product of Eilenberg-MacLane spaces with $\pi_* \scr D X \cong H_* (X; \bb Z)$, cf. \cite{Mc69, Dr12} after adding a disjointed basepoint to $X$. There are equivariant versions of the Dold-Thom theorem \cite{Gui} in terms of Bredon cohomology, although we will offer more straightforward interpretations in terms of transfer maps in our particular cases of interest.

	In what follows, let $\Sigma$ be a closed connected genus $g$ Riemann surface, equipped with a complex-conjugate biholomorphic involution $\tau : \Sigma \to \Sigma$, such that $C := \Sigma^\tau$ is a disjoint union of $k > 0$ many circles $C_1, \ldots, C_k$. Pick a collection $\vec w = \{w_1, \ldots, w_k\}$ of basepoints $w_i$ in each circle $C_i$, and consider the punctured surface $\Sigma^\circ := \Sigma \setminus \vec w$ with the induced involution still denoted $\tau$.

	\begin{DEF}\label{def:rdeg}
		Since $C$ is potentially disconnected, it becomes relevant to introduce the \keywd{real degrees} ${\rm rdeg}_i : \rdiv\Sigma^\circ \to \mathbb Z/2$ by counting the parity of the degree of the divisor when restricted to $C_i^\circ := C_i \setminus w_i$. The only relation between these real degrees and the non-equivariant degree is given by
		\[ {\rm deg} \equiv \sum_{i=1}^k {\rm rdeg}_i \quad {\rm mod~2}. \]
		As a result, $\pi_0 (\rdiv\Sigma^\circ)$ is given by the quotient of $\mathbb Z \times (\mathbb Z/2)^k$ by the relation above, which is non-canonically isomorphic to $\mathbb Z \times (\mathbb Z/2)^{k-1}$. All connected components are homeomorphic, since the whole space is a topological abelian group. The analogous statement for $\pi_0 (\rsym^d \Sigma^\circ)$ is that it is indexed by tuples of real degrees $d_i \in {0, 1}$ that sum up to $d$ mod 2, and such that their sum over $\mathbb Z$ does not exceed $d$.

		In either case, we use the notation \keywd{$\rsym^d_{d_1, \ldots, d_k} \Sigma^\circ$} and \keywd{$\rdiv^d_{d_1, \ldots, d_k} \Sigma^\circ$} to indicate the connected component with the real degrees $d_i \in \{0, 1\}$.
	\end{DEF}

	\begin{DEF}
		Denote by $\Sigma'$ the quotient $\Sigma / \tau$, and likewise $\Sigma'^\circ := \Sigma^\circ / \tau$, which are the same up to homotopy. The \keywd{transfer maps}
		\[ p^{-1} : \scr{S}^d \Sigma'^\circ \to \rsym^{2d} \Sigma^\circ, \qquad p^{-1} : \scr{D} \Sigma'^\circ \to \rdiv\Sigma^\circ \]
		are defined by taking, for every point $x \in \Sigma'^\circ$, the sum of the two preimages in $p^{-1}(x)$, where $p$ is the quotient map, with the convention that when there is only one preimage, it should be counted with multiplicity two.

		There are also \keywd{stabilization maps}
		\[ s_i : \rsym^d_{d_1, \ldots, d_k} \Sigma'^\circ \to \rsym^{d+1}_{d_1, \ldots, d_i + 1, \ldots, d_k} \Sigma'^\circ, \quad s_i : \rdiv_{d_1, \ldots, d_k} \Sigma'^\circ \to \rdiv_{d_1, \ldots, d_i + 1, \ldots, d_k} \Sigma'^\circ \]
		well-defined up to contractible choice, given by adding one point in $C_i^\circ$, which can be used to change the real degrees of a divisor.
	\end{DEF}

	\begin{PROP}\label{prop:treq}
		There is a homotopy equivalence from $\scr S^d \Sigma'^\circ$ to $\rsym^{2d + \ell}_{d_1, \ldots, d_k} \Sigma^\circ$, where $\ell$ is the number of nonzero $d_i$, given by taking the transfer map followed by stabilizations for each index $i$ with nonzero $d_i$. Analogously, $\scr D^d \Sigma'^\circ$ is homotopy equivalent to $\rdiv^{2d + \ell}_{d_1, \ldots, d_n} \Sigma^\circ$.
	\end{PROP}

	\begin{proof}
		Since $\scr S^d(-)$ is a homotopy-functor, we may replace $\Sigma^\circ$ with the quotient $\Sigma^\circ / \sim$ that collapses each $C_i^\circ$ to a point, and likewise $\Sigma'^\circ$ with the similarly defined quotient, without changing the homotopy type of the symmetric power. After making this modification, the composition of the transfer map and the stabilizations yields an honest homeomorphism, hence proving the claim. The proof is the same for $\scr D(-)$.
	\end{proof}

	\begin{PROP}\label{prop:sdconn}
		The connectivity of the canonical inclusion maps $\scr S^d\Sigma^\circ \hookrightarrow \scr D^d\Sigma^\circ$ and $\rsym^d\Sigma^\circ \hookrightarrow \rdiv^d\Sigma^\circ$ goes to $\infty$ as $d$ increases.
	\end{PROP}

	\begin{proof}
		First, let us reduce the real statement to a non-equivariant statement on $\Sigma'^\circ$. On the level of $\pi_0$, the statement follows from the discussion at the end of {\sc Def.~\ref{def:rdeg}}, because the only discrepancy between the two is the extra condition that the sum of the real degrees does not exceed $d$ for $\scr S^d$ in the case of $\pi_0 \rsym^d(\Sigma)$, which will become automatic once $d \ge k$.
		
		Thus, it suffices to show the desired connectivity result in any given connected component, which as we may recall is determined by a sequence of real degrees $d_i \in \mathbb Z/2$ under the constraint that they sum to $d$ mod 2. By {\sc Prop.~\ref{prop:treq}}, we have a commutative square
		\[ \begin{tikzcd}
            \scr S^{\tfrac 12(d - \ell)} \Sigma'^\circ \rar[hookrightarrow]\dar["\sim"] & \scr D^{\tfrac 12(d - \ell) }\Sigma'^\circ \dar["\sim"] \\
			\rsym^{d}_{d_1, \ldots, d_k} \Sigma^\circ \rar[hookrightarrow] & \rdiv^{d}_{d_1, \ldots, d_k} \Sigma^\circ 
		\end{tikzcd} \]
		where $\ell$ is the number of nonzero real degrees $d_i$; the vertical maps are homotopy-equivalences.

		Hence it remains to establish the general connectivity result for the inclusions $\scr S^d F \hookrightarrow \scr D^d F$  associated to any fixed 2-dimensional connected open manifold $F$. This is probably very standard, but we sketch a proof here: first of all, we may replace $\scr D^d$ with the infinite symmetric power $\scr S^\infty F$, defined as the direct limit of all the $\scr S^d F$ along stabilization by an arbitrarily chosen basepoint $p$ inside $F$. Indeed, there is a map $\scr S^\infty F \to \scr D^d F$ given by preserving all the points in the formal sum with the same multiplicities, except for the basepoint $p$, whose multiplicity in $\scr D^d F$ is uniquely determined by the total degree condition. We claim that this map is a homotopy equivalence. To this end, let us further replace $\scr D^d F$ up to homeomorphism with $\scr D F / \scr Dp$ via the obvious projection map; so it suffices to show that the composite $\scr S^\infty F \to \scr D F / \scr Dp$ is a homotopy equivalence. To prove this latter claim, let us recall that the Dold-Thom theorem asserts that $\scr S^\infty(-)$ and $\scr D(-)/\scr D*$ as functors of based spaces turn cofiber sequence into quasi-fiber sequences, hence producing cohomology theories once we apply $\pi_*$. The map $\scr S^\infty(-) \to \scr D(-)/\scr D*$ is therefore a natural transformation of cohomology theories, which must be an isomorphism by the usual inductive 5-lemma argument.

		Thanks to this natural isomorphism, we may reduce the proof of the main claim to showing that $\scr S^d F \hookrightarrow \scr S^\infty F$ has the desired connectivity property as $d \to \infty$. We do this by explicit computation in the following manner: since $F$ is non-closed, it is homotopy-equivalent to a bouquet of circles, which is further homotopy-equivalent to multiply punctured plane $\mathbb C$. We have $\scr S^d \bb C \cong \mathbb C^d$ via sending a tuple $(z_1, \ldots, z_d)$ to the coefficients of the polynomial $(x-z_1) \cdots (x-z_d)$; thus, we may identify the $d$-fold symmetric product of $\mathbb C - \{q_1, \ldots, q_\ell\}$ with the set of polynomials that do not vanish at $q_1, \ldots, q_\ell$. Each vanishing condition cuts out a hyperplane inside $\mathbb C^d$, so $\scr S^d F$ is the complement of $\ell$ many hyperplanes. Once $d \ge \ell$, the functionals cutting out the hyperplanes become linearly independent (e.g.~ by Vandermonde determinants), so the collection of hyperplanes can be arranged to be standard. Now, it is clear that the homotopy type is that of an $\ell$-torus, given by $\ell$ many copies of $\mathbb C^\times$, and $d - \ell$ copies of $\mathbb C$.
	\end{proof}

	By the argument in the last paragraph of the proof above, we have also established:

	\begin{PROP}\label{prop:divhom}
		The spaces $\scr D^d \Sigma^\circ$ and $\rdiv^d \Sigma^\circ$ have $\pi_{* \ge 2} \equiv 0$, and $\pi_1$ given by $H_1(\Sigma^\circ)$ and $H_1(\Sigma')$, respectively. The induced map on $\pi_1$ given by the natural inclusion $\rdiv^d \Sigma^\circ \hookrightarrow \scr D^d \Sigma^\circ$ is represented by the transfer homomorphism $H_1(\Sigma') \to H_1(\Sigma^\circ)$.
	\end{PROP}

	\begin{COR}\label{cor:cohsym}
		Given a fixed $m \ge 0$, we have
		\[ \begin{aligned} H^m(\scr{S}^d\Sigma^\circ; R) &\cong {\rm Hom}(\Lambda^m H_1(\Sigma^\circ), R) \\
		 H^m(\rsym^d_{d_1, \ldots, d_k}\Sigma^\circ; R) &\cong {\rm Hom}(\Lambda^m H_1(\Sigma'), R) \end{aligned} \]
		for all $\Sigma$ and all $d$ bounded below by a constant depending only on $m$.
	\end{COR}

	\begin{proof}
		By the proposition, we find that $\rsym^d_{d_1, \ldots, d_k}\Sigma^\circ$ admits a map to a $K(\pi, 1)$ with $\pi$ a free abelian group, which becomes at least $(m+1)$-connected for $d$ sufficiently large in terms of $m$. Since such $K(\pi, 1)$ are homotopy-equivalent to a torus of the appropriate rank, the conclusion then follows by the K\"unneth formula.
	\end{proof}

	\begin{CONSTR}\label{constr:probe}
		There is a very concrete way to understand these cohomology classes as follows: given parameterized curves $\gamma_1, \ldots, \gamma_k : S^1 \to \Sigma'$ or $\Sigma$, the fundamental class of the torus $\gamma_1 \times \cdots \times \gamma_k$ mapping into $\scr S^d \Sigma^\circ$ or $\rsym^d_{d_1, \ldots, d_k} \Sigma^\circ$ via taking the formal sum of the $k$-many points in each $\gamma_i$, and appropriately stabilizing to get to the correct degrees, gives an element of $H_k(\scr S^d \Sigma^\circ; R)$ or $H_k(\rsym^d_{d_1, \ldots, d_k} \Sigma^\circ; R)$. We can pair any given cohomology class with this homology class to get an element of $R$, for any tuple $(\gamma_1, \ldots, \gamma_k)$ of curves inside $\Sigma'$ or $\Sigma^\circ$. It is not difficult to check that these numbers are equal to the value of the associated multilinear map $\Lambda^k H_1(\Sigma^\circ) \to R$ or $\Lambda^k H_1(\Sigma') \to R$ under the isomorphism of the proposition above, evaluated at $[\gamma_1] \wedge \cdots \wedge [\gamma_k]$. Indeed, by naturality of the Dold-Thom theorem, this reduces to checking that the fundamental class of the torus in $\gamma_1 \times \cdots \times \gamma_k \subset \scr D(\gamma_1 \sqcup \cdots \sqcup \gamma_k)$ represents the element $[\gamma_1] \wedge \cdots \wedge [\gamma_k]$. This inclusion map breaks up on the nose as a product of the inclusions $\gamma_i \hookrightarrow \scr D \gamma_i$, so in turn by the multiplicativity of the fundamental classes with respect to the K\"unneth product, it remains to check the result for $k = 1$. In this case, $S^1 \hookrightarrow \scr D^d S^1$ is a homotopy equivalence, e.g.~ by the same argument used in the very last part of the proof of {\sc Prop.~\ref{prop:sdconn}}.
	\end{CONSTR}

	Now, if we pick enough curves to form a basis for $H_1(\Sigma')$ or $H_1(\Sigma)$, the evaluation maps we have just introduced in {\sc Constr.~\ref{constr:probe}}, with respect to all possible size-$k$ subsets, uniquely identify any cohomology class on the symmetric or real symmetric power. In this setting, we investigate the characteristic classes of the symmetric and real symmetric powers in the stable range (cf. {\sc Prop.~\ref{prop:sdconn}}). The Chern classes over $\mathbb Q$ have been studied in the non-equivariant case by \cite{Mc62}, which in particular give the Stiefel-Whitney classes over $\mathbb Z/2$. In our paper, we are interested in the Stiefel-Whitney classes for the real scenario. We use a different technique, which allows us to go through with the computation even in the case that $\Sigma'$ is non-orientable, which does not seem amenable to the algebro-geometric techniques of \cite{Mc62}. First, we introduce some combinatorics and linear algebra associated to it:

    \begin{DEF}
        Let $S$ be a set. A \emph{matching} $M$ in $S$ is a subset $M \subseteq 2^S$ such that $|e| = 2$ for all $e \in M$, and $e \cap e' = \0$ for all $e \neq e'$ in $M$. We use the notation $\scr M(S)$ for the set of matchings in $S$, and the notation $V(M) \subseteq S$ for $\bigcup_{e \in M} e$, given any $M \in \scr M(S)$.
    \end{DEF}
    
    \begin{DEF}\label{def:symtoalt}
        Let $A$ be an abelian group, and $I : {\rm Sym}^2 A \to \mathbb Z/2$ be a symmetric bilinear form. We define an associated multilinear form $W^I_n : A^{\otimes n} \to \mathbb Z/2$ via the formula
        \[ W^I_n(x_1, \ldots, x_n) := \sum_{M \in \scr M(S)} \left(\prod_{\{i, j\} \in M} I(x_i, x_j)\right) \left(\prod_{k \in S \setminus V(M)} I(x_k, x_k)\right), \]
        where $S = \{1, \ldots, n\}$, or more succinctly, but with some redundancy
        \[ W^I_n(x_1, \ldots, x_n) := \sum_{\substack{\sigma \in \mathfrak S_n \\ \sigma^2 = {\rm id}}} \prod_{i=1}^n I(x_i, x_{\sigma(i)}). \]
    \end{DEF}

    \begin{PROP}
        The form $W^I_n$ is alternating, i.e.~ it vanishes whenever $x_i = x_j$ for $i \neq j$.
    \end{PROP}

    \begin{proof}
        Under the assumption that $x_a = x_b$ for some $a \neq b$, we argue that terms cancel in pairs. Namely, the terms corresponding to
        \begin{list}{$\cdot$}{}
            \item Matchings $M$ containing the edge $\{a, b\}$ as an element, and matchings $M'$ for which $a \notin V(M)$ and $b \notin V(M)$, via the bijective correspondence $M' = M \cup \{\{a, b\}\}$. This is due to $I(x_a, x_b) = I(x_a, x_a) I(x_b, x_b)$.
            \item Matchings $M$ satisfying $a \in V(M)$ and $b \notin V(M)$, and matchings $M'$ satisfying $a \notin V(M)$ and $b \in V(M)$ via the bijective correspondence $M' = M \setminus \{\{a, c\}\} \cup \{\{b, c\}\}$ where $c \in V(M)$ is the unique element satisfying $\{a, c\} \in M$. This is due to $I(x_a, x_c) I(x_b, x_b) = I(x_b, x_c) I(x_a, x_a)$.
            \item Matchings $M$ satisfying $a, b \in V(M)$ but $\{a, b\} \notin M$ can be paired up according to the rule
            \[ \{ \{a, c\}, \{b, d\} \} \sqcup N = M \leftrightarrow M' = N \sqcup \{ \{a, d\}, \{b, c\} \}, \]
            where $\sqcup$ indicates that the two operands are disjoint, and $c, d \in V(M), N \in \scr M(S \setminus \{a, b\})$ are the unique elements that make the equality on the left true. This is due to $I(x_a, x_c) I(x_b, x_d) = I(x_a, x_d) I(x_b, x_c)$. \qedhere
        \end{list}
    \end{proof}

	Now, we have the necessary language to state the formula of the Stiefel-Whitney classes:

    \begin{THM}\label{thm:wi}
        The tangent bundle to $\rsym^d_{d_1, \ldots, d_k}\Sigma^\circ$ has Stiefel-Whitney classes satisfying
        \[ w_d([\gamma_1] \wedge \ldots \wedge [\gamma_d]) = W^I_d([\gamma_1], \cdots, [\gamma_d]) \]
        under the isomorphism of {\sc Cor.~\ref{cor:cohsym}}, and the conventions of {\sc Constr.~\ref{constr:probe}}, where $I(-, -)$ is the intersection pairing on $H_1(\Sigma')$
    \end{THM}

    \begin{proof}
		We will need to refine the homotopy equivalence from {\sc Prop.~\ref{prop:treq}} in order to respect the tangent bundles, potentially up to copies of the trivial bundle. Recall that this homotopy equivalence is built by taking the transfer map, and then applying several stabilizations in order to get the appropriate degrees. However, this map is not a diffeomorphism (not even homeomorphism), so we cannot directly conclude that the tangent bundles are respected. Instead, let us find open subsets in both spaces that are actually homeomorphic up to taking Cartesian product with several copies of $\mathbb R$, such that their inclusion into the whole space is a homotopy equivalence; this suffices in order to prove the claim.

		To this end, consider the complement $\tilde \Sigma'$ of a collar neighborhood of the boundary in $\Sigma'$, which indeed satisfies the property that the inclusion $\tilde \Sigma' \hookrightarrow \Sigma'$ is a homotopy-equivalence, and hence induces a homotopy-equivalence on symmetric powers. The image thereof under the homotopy-equivalence from {\sc Prop.~\ref{prop:treq}} is given by the subset of $\rsym^d\Sigma^\circ$ of tuples that avoid a tubular neighborhood of $C := \Sigma^\tau$, except for the points that were artificially added by stabilization. We may enlarge this subset without changing its homotopy type by allowing these artificially added points to roam freely in the fixed-point set; the result on the level of diffeomorphism types is that of taking Cartesian product with $\mathbb R^\ell$, and the resulting subset is now open inside $\rsym^d\Sigma^\circ$. This proves the claim.

        This reduces the computation of Stiefel-Whitney classes of $\rsym^d_{d_1, \ldots, d_k} \Sigma^\circ$ to those of the non-equivariant $\scr{S}^{\tfrac 12(d - \ell)} \Sigma'{}^\circ$. First, note that by adding extra non-orientable handles to $\Sigma'$ induces an inclusion on $H_1$, and the Stiefel-Whitney classes, when viewed as alternating multilinear forms on $H_1$, respect this inclusion.  Therefore, we may assume that the surface $\Sigma'$ is non-orientable without loss of generality. Any such surface can be represented by attaching $n$-many (pairwise unlinked) non-orientable handles. The great advantage of this description (which does not exist for orientable surfaces) is that we can find a basis for $H_1$ consisting of \emph{disjointly embedded} $\gamma$-curves (cf. {\sc Constr.~\ref{constr:probe}}), in the manner presented in {\sc Fig.~\ref{fig:surfgens}} 
		\begin{figure}
			\centering\includegraphics{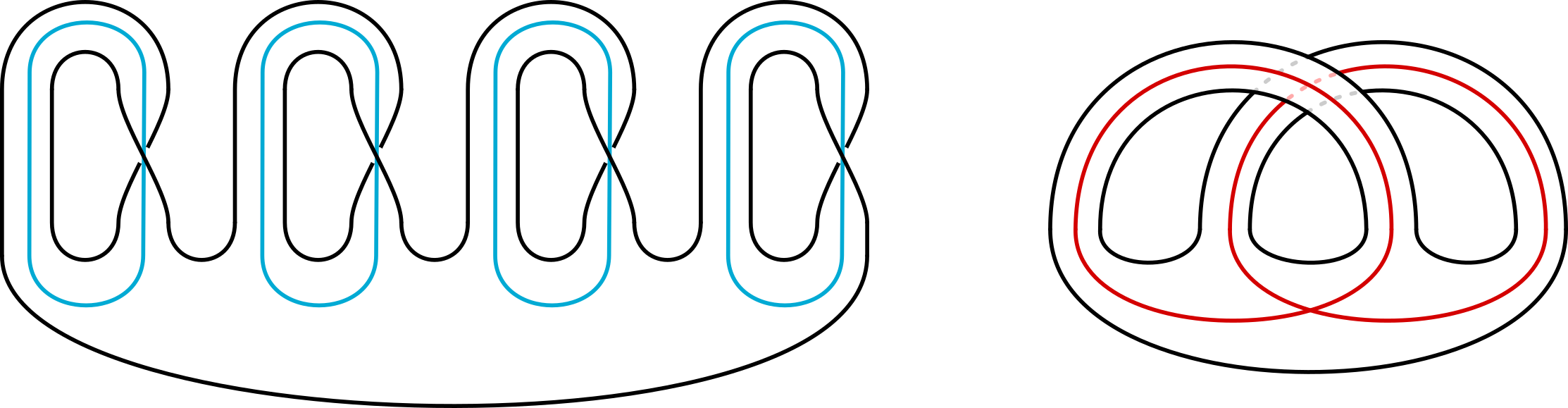}
			\caption{On the left side, we show that a set of disjoint embedded curves representing a basis for $H_1$ can be found in the non-orientable case; on the right we show an example of the contrary in the case of a punctured torus.}\label{fig:surfgens}
		\end{figure}
		such that the restriction of the tangent bundle to any torus $\gamma_1 \times \cdots \times \gamma_k$ simply breaks up as a disjoint union of $k$ line bundles, in addition to several copies of the trivial bundle:
        \[ \mathbb R^{2k - 2d} \oplus \bigoplus_{i = 1}^d T\Sigma'|_{\gamma_i}, \]
        where $\gamma_i$ are the disjointly embedded simple closed curves mentioned previously. By the Cartan formula, we find that $w_d$ is equal to 1 on each simple tensor. Since the simple tensors span the exterior power, all we have to do to deduce the desired conclusion is check that the alternating form $W^I_d$ defined in {\sc Def.~\ref{def:symtoalt}} agrees with $w_d$ on the simple tensors. Indeed, in this case because the simple closed curves are pairwise disjoint, the intersection numbers between various different $\gamma_i$ are zero. Thus, in the big sum defining $W^I_d$, all terms are zero, except potentially the one corresponding to the empty matching, namely $\prod_i I(\gamma_i, \gamma_i)$. Since all of the handles in our decomposition are non-orientable, each of these $I(\gamma_i, \gamma_i)$ is one, hence the product is also one, as desired.
    \end{proof}

\section{Real Heegaard diagrams and pathspaces}\label{sec:heeg}

	In this section, we review the prerequisites necessary to define $\widehat{\it HFR}$ with $\mathbb Z/2$-coefficients, as in \cite{GM25}, and introduce the associated pathspaces, which will play an important role in studying the index theory that determines gradings and orientations {\sc(\S\ref{sec:gror})}.

	\begin{DEF}
		A (multi-based) \keywd{Heegaard diagram} $\scr H = (\Sigma, \vec w, \vec \alpha, \vec \beta)$ consists of:
		\begin{list}{$\cdot$}{}
			\item a closed, connected Riemann surface $\Sigma$ of genus $g$;
			\item a non-empty collection $\vec w = \{w_1, \ldots, w_k\} \subset \Sigma$ of distinct basepoints;
			\item 1-dimensional embedded submanifolds $\vec \alpha = \alpha_1 \sqcup \cdots \sqcup \alpha_d$ and $\vec \beta = \beta_1 \sqcup \cdots \sqcup \beta_d$ of $\Sigma$, each consisting of \keywd{$d := g + k - 1$} many circle components, such that $\Sigma \setminus \vec \alpha$ and $\Sigma \setminus \vec \beta$ each have exactly one basepoint $w_i$ in each connected component. (The order of the $\alpha_i$ and $\beta_i$ is not part of the data.)
		\end{list}
		A \keywd{real involution} $\tau$ on $\scr H$ is a complex-conjugate biholomorphism $\tau : \Sigma \to \Sigma$ such that
		\begin{list}{$\cdot$}{}
			\item the fixed-point set $C := \Sigma^\tau$ is a disjoint union of $k > 0$ circles $C_1, \ldots, C_k$,
			\item $w_i \in C_i$ for each $1 \le i \le k$,
			\item $\tau(\vec \alpha) = \vec \beta$.
		\end{list}
		We usually refer to $(\scr H, \tau)$ as a \keywd{real Heegaard diagram}, and often suppress $\tau$ from the notation.
	\end{DEF}

	\begin{CONSTR}\label{constr:morse}
		From the data of a Heegaard diagram $\scr H$, we may canonically associate to it an oriented, connected, closed topological 3-manifold $Y$ as follows:
		\begin{list}{$\cdot$}{}
			\item Attach disks $D^\alpha_i$ and $D^\beta_i$ (formally defined as unreduced cones $C\alpha_i$ and $C\beta_i$ with tips \keywd{$a_i$} and \keywd{$b_i$}, respectively) to $\Sigma$, to obtain a 2-skeleton $Y^{[2]}$;
			\item Consider the surfaces with boundary $\hat \Sigma^\alpha := \Sigma | \vec \alpha$ and $\hat \Sigma^\beta := \Sigma | \vec \beta$ obtained by cutting $\Sigma$ along the 1-submanifolds $\vec \alpha$ and $\vec \beta$; their boundaries are given by $2d$ many circles $\alpha_i^\pm$ and $\beta_i^\pm$; then, attach disks $D^{\alpha, \pm}_i$ and $D^{\beta, \pm}_i$ along them to obtain a disjoint union $\vec S^\alpha := S_1^\alpha \sqcup \cdots \sqcup S_k^\alpha$ of topological 2-spheres $S_i^\alpha$ each containing $w_i$, and likewise $\vec S^\beta$;
			\item There are natural collapse maps $\vec S^\alpha \to Y^{[2]}$ and $\vec S^\beta \to Y^{[2]}$ that identify $D^{\alpha,+}_i$ with $D^{\alpha,-}_i$ and $D^{\beta,+}_i$ with $D^{\beta,-}_i$. We obtain a 3-dimensional CW complex by attaching 3-balls $B_i^\alpha$ and $B_i^\beta$ (formally defined as unreduced cones $CS^\alpha_i$ and $CS^\beta_i$ with tips \keywd{$p_i$} and \keywd{$q_i$}, respectively) via the attaching maps $S^\alpha_i \to \vec S^\alpha \to Y^{[2]}$ and $S^\beta_i \to \vec S^\beta \to Y^{[3]}$.
		\end{list}
		By using handles instead of cells, we obtain a smooth 3-manifold $Y = Y(\scr H)$ instead, at the cost of introducing some contractible spaces of choices of thickenings, and smoothings along boundary attachments and corners, which are all fairly standard. This manifold is naturally oriented, with the convention that $TY \cong T\Sigma \oplus N\Sigma$, where $T\Sigma$ is oriented from the Riemann surface structure, and $N\Sigma$ is oriented so that the positive direction points inside the $D^\beta_i$ and $B^\beta_i$, and hence outside of $D^\alpha_i$, $B^\alpha_i$.
		
		From the aforementioned smooth handle decomposition, one can reconstruct a self-indexing Morse function $f : Y \to [0, 3]$, such that $f^{-1}(3/2) = \Sigma$, and the critical points are as follows:
		\begin{list}{$\cdot$}{}
			\item $k$-many critical points $p_i$ of index 0, which are the centers of the $B^\alpha_i$;
			\item $d$-many critical points $a_i$ of index 1, which are the centers of the $D^\alpha_i$;
			\item $d$-many critical points $b_i$ of index 2, which are the centers of the $D^\beta_i$;
			\item $k$-many critical points $q_i$ of index 3, which are the centers of the $B^\beta_i$;
		\end{list}
		\begin{figure}
			\centering\includegraphics{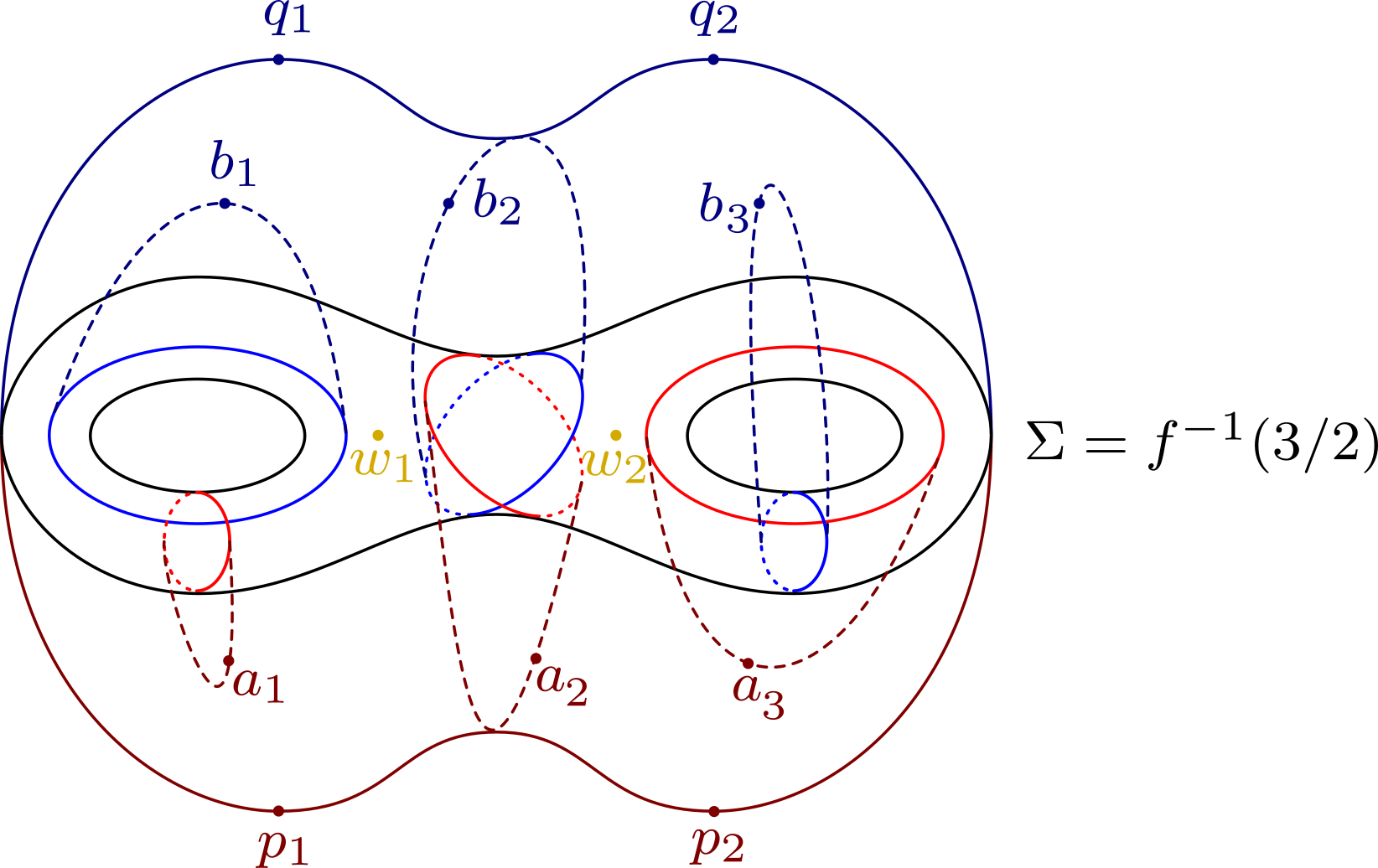}
			\caption{A depiction of the critical points of a Morse function reconstructed from $\alpha$ (red) and $\beta$ (blue) circles.}\label{fig:pabq}
		\end{figure}
		See {\sc Fig.~\ref{fig:pabq}} for a picture. The subdomains \keywd{$U_\alpha$}$~:= f^{-1}[0, 3/2]$ and \keywd{$U_\beta$} $~:= f^{-1}[3/2, 1]$ are both genus-$g$ handlebodies bounding the Heegaard surface $\Sigma$.

		When $\scr H$ is endowed with a real involution, the 3-manifold $Y$ inherits an orientation-preserving involution $\tau$ with $\tau(U_\alpha) = U_\beta$, and the Morse function can be arranged to be symmetric, i.e.~ $f(\tau(y)) = 3 - f(y)$, hence swapping all the objects labelled $\alpha$ with the corresponding ones labelled $\beta$. The fixed-point locus of $\tau$ on $Y$ is given by $C$, and the involution is locally modelled by the representation $\mathbb R_{\rm triv} \oplus \mathbb R_{\rm sign}^2$. Consequently, $Y' := Y / \tau$ is also a 3-manifold, and the projection map $\pi : Y \to Y/\tau$ is a double cover branched along a link $C \subset Y'$. The surface $\Sigma' := \Sigma / \tau$ bounds the link $C$ inside $Y'$, and its complement is a handlebody.
	\end{CONSTR}

	\begin{DEF}\label{def:realmoves}
		There are three kinds of operations that one can do to a (real) Heegaard diagram, without changing the (real) diffeomorphism type of the associated 3-manifold. The first one is given by \keywd{smooth isotopy} of the $\alpha$ and $\beta$ curves (if the diagram is to be real, the isotopy must be $\tau$-equivariant).
		
		The second one is called a \keywd{handle-slide}, which can be performed for either the $\alpha$ or the $\beta$ circles in the non-real case, or simultaneously for both in a $\tau$-equivariant way in the real case. Given two circles $\alpha_i, \alpha_j$, as well as an auxiliary curve $\alpha_i' \in \Sigma \setminus \vec \alpha$ such that $\alpha_i, \alpha_j, \alpha_i'$ cobound a pair of pants not containing any of the basepoints, we may replace $\alpha_i$ with $\alpha_i'$ (likewise for $\beta$-curves). Alternatively, one can think of $\alpha_i'$ as being an internal band sum of $\alpha_i$ and $\alpha_j$, cf. {\sc Fig.~\ref{fig:bsum}}. There is a 1-parameter family of Morse functions on $Y$ that interpolates between the Morse functions associated to the Heegaard diagram before and after the handle-slide, given by sliding the index-1 critical point $a_i$ over $a_j$ to arrive at $a_i'$ (the rest of the $a$'s stay fixed).
		\begin{figure}
			\centering\includegraphics{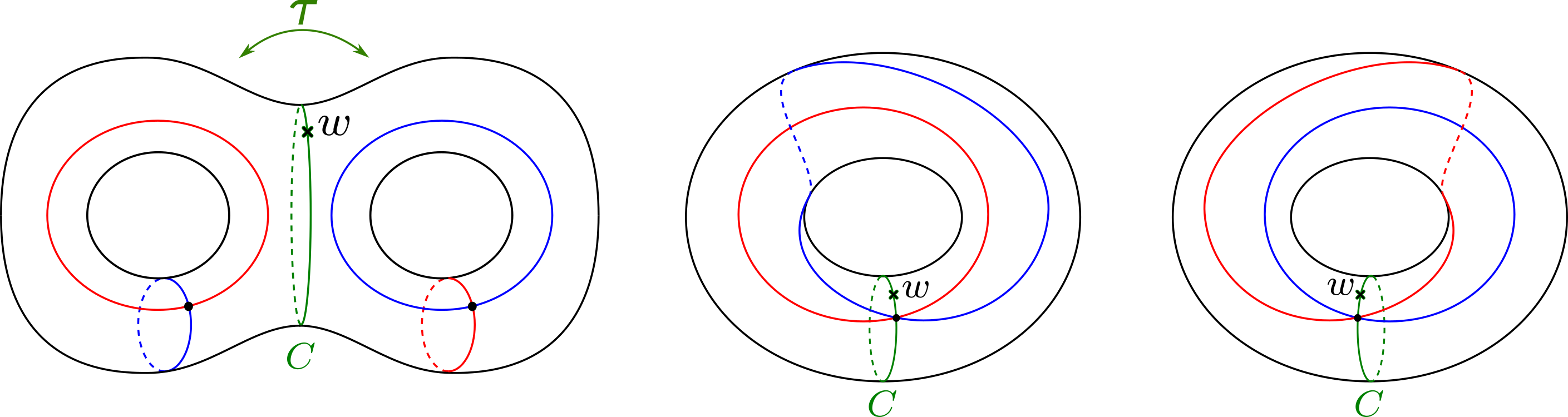}
			\caption{The three standard real Heegaard diagrams for $S^3$ from \cite[\sc Fig.~3.3]{GM25}, which under connect-sum of diagrams give rise to \emph{free, positive fixed-point,} and \emph{negative fixed-point stabilization,} respectively. The involution on the right is the obvious horizontal flip; the involution on the others is in fact determined by the $\alpha$ (red) and $\beta$ (blue) curves, and is the one whose quotient recovers the Klein bottle. See \cite{GM25} for more details.}\label{fig:rstab}
		\end{figure}
		The third kind of operation is called a \keywd{stabilization}, and comes from connect-summing a standard diagram for $S^3$ with one basepoint, near any of the $w_i$, and fusing this basepoint $w_i$ with the basepoint $w$ of the diagram for $S^3$; in the real case, there are three such standard diagrams, cf. \cite{GM25} redrawn for convenience here in {\sc Fig.~\ref{fig:rstab}}. There is a 1-parameter family of functions interpolating between the old and new Morse functions on $Y$, given by creation of one or two pairs of index 1 and 2 critical points happening near the basepoint (symmetrically if the stabilization is real).

		We use the adjective ``real'' whenever we are talking about the $\tau$-equivariant versions of the moves.
	\end{DEF}

	\begin{figure}
		\centering
		\includegraphics{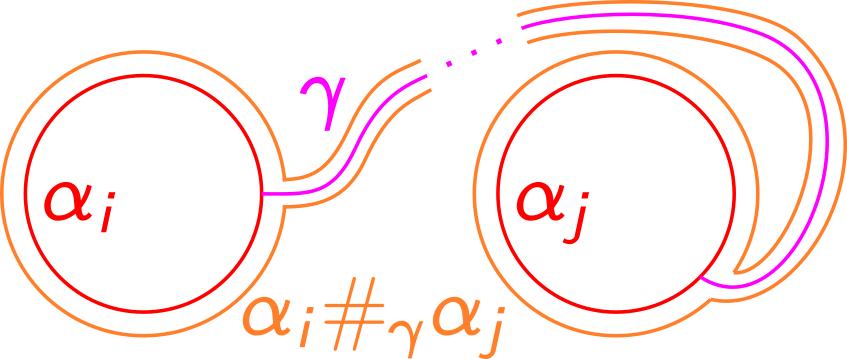}
		\caption{Pictorial description of a handle-slide, via choosing a band $\gamma$ to connect-sum $\alpha_i$ and $\alpha_j$.}\label{fig:bsum}
	\end{figure}

	Now, we recall a technical combinatorial condition on the diagram, which will ensure that $\widehat{\it HFR}$ is a well-defined invariant.

	\begin{DEF}
		Assuming that $\alpha_i \pitchfork \beta_j$ for all $i, j$ in a given Heegaard diagram $\scr H$, we define the ($\widehat{\it HF}$ version of the) \keywd{group of domains} $\hat \Delta(\scr H)$ to be the free abelian group generated by the connected components of $\Sigma \setminus (\vec \alpha \cup \vec \beta)$ that do not contain any basepoint $w_i$. A domain $D \in \hat \Delta(\scr H)$ is called a \keywd{periodic domain} if, for every intersection point $x \in \alpha_i \cap \beta_j$, the multiplicities $m_0, m_1, m_2, m_3$ of $D$ at the 4 components touching $x$, arranged in cyclic order, satisfy $m_0 + m_2 = m_1 + m_3$ (with the convention that $m_i = 0$ if the respective component contains one of the basepoints). These form a subgroup $\hat\varPi(\scr H) \subset \hat\Delta(\scr H)$.

		If $\tau$ is a real involution on $\scr H$, there are also real versions ${\it r}\hspace{-0.1em}\hat \varPi(\scr H, \tau) \subset {\it r}\hspace{-0.1em}\hat \Delta(\scr H, \tau)$ of domains that are invariant under the $\tau$-action that swaps elementary domains with the same multiplicity.
	\end{DEF}

	\begin{DEF}
		The Heegaard diagram $\scr H$ is called (weakly) \keywd{admissible} if every nonzero periodic domain $D \in \hat\varPi(\scr H)$ has both at least one negative and one positive multiplicity.
	\end{DEF}

	\begin{RMK}
		We refer to a real Heegaard diagram $\scr H$ as being admissible, though the real structure is not actually relevant in the definition of admissibility. That being said, it does suffice to check the admissibility condition on the subgroup or real periodic domains, since any non-zero non-negative domain $D$ can be made $\tau$-invariant by considering the sum $D + \tau D$.
	\end{RMK}

	Guth-Manolescu, extending on work of Nagase \cite{Na79}, prove that
	\begin{enumerate}
		\item[\sc i.] any real 3-manifold $(Y, \tau)$ as in our setup can be recovered from an admissible real Heegaard diagram;
		\item[\sc ii.] any two real Heegaard diagrams representing the same real 3-manifold $(Y, \tau)$ are related by a sequence of real isotopies, handleslides, and stabilizations, such that all intermediate diagrams are admissible.
	\end{enumerate}

	\begin{DEF}
		Given a multi-pointed Heegaard diagram $\scr H = (\Sigma, \vec w, \vec \alpha, \vec \beta)$, we denote $\mathbb T_\alpha := \alpha_1 \times \cdots \times \alpha_{d} \subset \scr S^{d} \Sigma^\circ$, and $\mathbb T_\beta$ analogously.
	\end{DEF}

	\begin{DEF} {(cf. \cite{GM25})}
		One defines $\widehat{\it HFR}(\scr H, \tau; \mathbb F_2) := {\it HL}(\mathbb T_\alpha, \rsym^d\Sigma^\circ)$ in ${\scr S^d \Sigma^\circ}$ for any \emph{admissible} real Heegaard diagram $(\scr H, \tau)$. The isomorphism type of this group is invariant under real Heegaard moves through admissible diagrams, and therefore gives a well-defined invariant $\widehat{\it HFR}(Y, \tau; \mathbb F_2)$.
	\end{DEF}

	We now introduce some objects of interest in the study of index theory:
	
	\begin{DEF}\label{def:path}
		We define the ($\widehat{\it HF}$ version of the) \keywd{pathspace $\hatP(\scr H)$} as the space of paths inside $\scr S^{d} \Sigma^\circ$ with one end on $\mathbb T_\alpha$ and the other end on $\mathbb T_\beta$. The \keywd{stable pathspace $\hatP_s(\scr H)$} is defined by replacing $\scr S^{d}\Sigma^\circ$ with $\scr D^{d}\Sigma^\circ$ in the definition.
		If moreover $\tau$ is a real involution on $\scr H$, there is an associated involution on both $\hatP(\scr H)$ and $\hatP_s(\scr H)$, given by reversing the direction of the path and applying the involution $\tau$. Let us define the \keywd{real pathspace ${\rhatP}(\scr H, \tau)$} and \keywd{stable real pathspace $\rhatP_s(\scr H, \tau)$} as the fixed points of the involution on the two pathspaces. These real versions can also be interpreted as spaces of paths, joining $\mathbb T_\alpha$ to $\rsym^{d}\Sigma^\circ$ or $\rdiv^{d} \Sigma^\circ$ inside $\scr S^{d} \Sigma^\circ$ or $\scr D^{d} \Sigma^\circ$ respectively, because a $\tau$-equivariant path is uniquely determined by its first half. When no ambiguity arises, we suppress $\tau$, or even $\scr H$ from the notation.
	\end{DEF}

	\begin{PROP}\label{prop:connsd}
		The connectivity of the inclusions $\hatP \hookrightarrow \hatP_s$ and $\rhatP \hookrightarrow \rhatP_s$ is bounded below by a function depending only on $d = g + k - 1$, which goes to $\infty$ as $d$ increases.
	\end{PROP}

	\begin{proof}
		Spaces of paths can be interpreted as a model for homotopy fiber products. Associated to any such fiber product $X \times^h_Y Z$, we have the following diagram
		\[ \begin{tikzcd}
			F \dar[equals] \rar & X \times^h_Y Z \dar\rar & Z \dar \\
			F \rar & X \rar & Y
		\end{tikzcd} \] 
		whose rows are homotopy fiber sequences. Now, suppose that we have a map of such diagrams induced by a natural transformation $(X \to Y \leftarrow Z) \Longrightarrow (X' \to Y' \leftarrow Z')$ of diagrams, and that the individual maps $X \to X', Y \to Y', Z \to Z'$ all induce isomorphisms on $\pi_{* \le k}$ for some $k$. Then, $F \to F'$ also induces an isomorphism on $\pi_{* \le k-1}$ by the 5-lemma applied to the long exact sequence of the bottom row, which in turn implies that the map induced on homotopy-fiber products induces an isomorphism on $\pi_{* \le k-2}$ by the same 5-lemma argument, applied now to the first row. Our claim then follows from this general observation, and the connectivity results mentioned in {\sc Prop.~\ref{prop:sdconn}}.
	\end{proof}

	The upshot is that once we stabilize the Heegaard diagram enough times, we enter the stable range for any particular homotopy or (co)homology group of interest, so we might as well work with the stable version when studying the obstruction to $\mathbb Z/2$-gradings or $\mathbb Z$-coefficients. This is similar to Ozsv\'ath-Szab\'o's use of $\pi_2'$ to avoid edge cases for low genus. Another motivation to consider the stable pathspaces is that, unlike the unstable versions, their homotopy type only depends on the topological input data.

	We prove this claimed invariance up to handleslides in the following theorem, and up to stabilization later {\sc(Prop.~\ref{prop:pathstab})}, though a more straight-forward proof might be possible.
	
	\begin{THM}
		We have:
		\begin{enumerate}
			\item[\sc i.] Given diagrams $\scr H, \scr H'$ related by a handle-slide, there is a homotopy-equivalence $h : \hatP_s(\scr H) \overset\sim\to \hatP_s(\scr H')$ canonical up to a contractible space of choices. If $(\scr H, \tau), (\scr H', \tau)$ are real diagrams related by a real handle-slide (i.e.~ a composition of two $\tau$-symmetric ordinary handleslides for both $\alpha$ and $\beta$), then the composition $h : \hatP_s(\scr H, \tau) \to \hatP_s(\scr H', \tau)$ can be arranged to be a $\tau$-equivariant homotopy equivalence inducing a homotopy-equivalence $h^\tau : \rhatP_s(\scr H, \tau) \overset\sim\to \rhatP_s(\scr H', \tau)$.
			\item[\sc ii.] Given diagrams $\scr H_0, \scr H_1, \ldots, \scr H_n = \scr H_0$ such that $\scr H_{i+1 ~\text{mod}~n}$ is related to $\scr H_i$ by a handle-slide with induced homotopy-equivalence $h_i$, then there is a homotopy of the composite $h_{n-1} \circ \cdots \circ h_0$ auto-equivalence of $\hatP_s(\scr H_0)$ to the identity map, canonical up to contractible space. The analogous $\tau$-equivariant fact, as well as its corollary concerning $\rhatP_s$, are true if all diagrams are real, and all handle-slides are real.
		\end{enumerate}
	\end{THM}

	\begin{RMK}
		As we will show later, these stable pathspaces have vanishing $\pi_{* \ge 2}$, so in particular higher coherence is also guaranteed, since the space of order-$n$ homotopies between two such spaces is contractible, so all the obstructions vanish. Intuitively, this statement shows that the stable pathspaces satisfy ``higher naturality'' with respect to handleslides, which is the strongest possible statement of ``independence of Heegaard diagram'' concerning something that depends on contractible spaces of choices.
	\end{RMK}

	Since both the ordinary and real pathspaces are examples of homotopy fiber-products, which satisfy a strong universal property (i.e.~ the object satisfying the said property is unique up to a contractible space of homotopy equivalences), the theorem above reduces immediately to the following ``higher naturality'' statement about the map $\mathbb T_\alpha \to \scr D^d(\Sigma^\circ)$:

	\begin{PROP}\label{prop:torihslide} We have:
		\begin{enumerate}
			\item Given $\vec \alpha, \vec \alpha'$ related by a handle-slide, there is a homotopy-coherent diagram
			\[ \begin{tikzcd}
				\mathbb T_\alpha \rar["\sim"]\dar[hookrightarrow] & \mathbb T_{\alpha'} \dar[hookrightarrow] \\
				\scr D^d \Sigma^\circ \rar[equals] & \scr D^d \Sigma^\circ
			\end{tikzcd} \]
			(i.e.~ the top horizontal homotopy equivalence, together with a \emph{choice} of homotopy witnessing the commutativity of the diagram) canonical up to contractible space of choices.
			\item Given $\vec \alpha_{(0)}, \vec \alpha_{(1)}, \ldots, \vec \alpha_{(n)} = \vec \alpha_{(0)}$ such that each $\vec \alpha_{(i+1)}$ is related to $\vec \alpha_{(i)}$ by a handle-slide, the induced composition of homotopy-coherent squares
			\[ \begin{tikzcd}
				\mathbb T_{\vec \alpha_{(0)}} \dar[hookrightarrow] \rar["\sim"] & \mathbb T_{\vec \alpha_{(0)}} \dar[hookrightarrow] \\
				\scr D^d \Sigma^\circ \rar[equals] & \scr D^d \Sigma^\circ
			\end{tikzcd} \]
			is homotopic to the identity square relative the bottom edge, canonically up to contractible choice.
		\end{enumerate}
	\end{PROP}

	The key observation allowing us to prove this result is the following:

	\begin{LEM}\label{lem:T}
		In any Heegaard diagram $\scr H$, there is a natural homotopy fiber-product square
		\[ \begin{tikzcd}
			\mathbb T_\alpha \rar[hook]\dar & \scr D^d \Sigma^\circ \dar[hook] \\
			* \rar["\vec a"] & \scr  D^d U_\alpha
		\end{tikzcd} \]
		where the bottom inclusion picks out $\vec a = \{a_1, \ldots, a_d\}$ viewed as a divisor in $U_\alpha$, and the homotopy witnessing the commutativity of the square is the radial homotopy inside the disks $D^\alpha_i$, well-defined up to a contractible space of choices.
	\end{LEM}

	\begin{proof}
		Instead of checking the universal property, it suffices to check that the induced map $\mathbb T_\alpha \to F := * \times^h_{\scr D^d \Sigma^\circ} \scr D^d U_\alpha$ is a homotopy-equivalence. To this end, we note that both $\scr D^d \Sigma^\circ$ and $\scr D^d U_\alpha$ are homotopy-equivalent to tori, and the inclusion of the former into the latter induces a surjection on $\pi_1$, isomorphic in the arrow category to $H_1(\Sigma^\circ) \to H_1(U_\alpha)$. As a result, the connecting homomorphism in the long-exact sequence in homotopy groups associated to the fiber sequence $F \to \scr D^d \Sigma^\circ \to \scr D^d U_\alpha$ is zero, so we have that $F$ is also homotopy-equivalent to a torus, with $\pi_1 F \to \pi_1 \scr D^d \Sigma^\circ$ given by inclusion of $\<[\alpha_1], \ldots, [\alpha_d]\>$ inside $H_1(\Sigma^\circ)$. Now, the composition $\mathbb T_\alpha \to F \to \scr D^d \Sigma^\circ$ induces the same inclusion on $\pi_1$, which implies that $\mathbb T_\alpha \to F$ is a map of tori inducing an isomorphism on $\pi_1$, i.e.~ it is a homotopy equivalence, hence proving the claim.
	\end{proof}

	\begin{proof}[Proof of {\sc Prop.~\ref{prop:torihslide}}]
		For part {\sc i}, we use the functoriality of the homotopy-pullback with respect to homotoping the legs of the span. In this case, the inclusion $\rdiv^d \Sigma^\circ \to \scr D^d \Sigma^\circ$ is always the same, whereas the inclusion $* \overset{\vec a}\to \scr D^d \Sigma^\circ$ will change because the $\alpha_i$ circle changed to a different one $\alpha_i'$. That being said, we obtain a natural path between $a_i$ and $a_i'$, because as we mentioned before the handle-slide induces a 1-parameter of Morse functions interpolating between the two endpoints. This produces the desired construction necessary for part {\sc i}.

		Part {\sc ii} then amounts to checking that the concatenation of all consecutive paths joining $\vec a_{(0)}, \vec a_{(1)}, \ldots, \vec a_{{n}} = \vec a_{(0)}$ is itself a contractible loop in $\scr D^d \Sigma^\circ$ up to contractible choice. To this end, we use the idea that $\vec p - \vec a \in \scr D^{d-k} \Sigma^\circ$ represents the Poincar\'e dual of the relative Euler class of $TY$ relative to the vector field along $\Sigma$ that always points into $U_\alpha$ and out of $U_\beta$. Indeed, the gradient flow of the Morse function has zeros precisely the $\vec a$ and $\vec p$ points, with the appropriate signs. The intermediary Morse functions between the consecutive $\vec \alpha_i, \vec \alpha_{i+1}$ also give well-defined gradient flows, whose zeros consist either of the same $p_i$ (which don't change in the process of handle-sliding), or of the intermediary $\vec a$ points. Now, since the space of vector fields rel a fixed vector field on the boundary is contractible, let us pick an arbitrary disk filling in the aforementioned loop. We may assume that the zeros of this 2-parameter family of vector fields are cut out transversely, giving us a surface $Z$ over $D^2$ whose restriction to the boundary points is given by the loop of ($\vec p - \vec a$)'s from before. The map $Z \to D^2$ can be arranged to be a branched cover, and hence have discrete fibers. This allows us to define a map $D^2 \to \scr D^{k-d} U_\alpha$ given by taking the preimage in $Z$ with multiplicity at a point given by the local degree of the map there. This gives a null-homotopy of the loop $\vec p - \vec a : S^1 \to \scr D^{k-d} U_\alpha$. Moreover, we can also ensure that the $\vec p$~'s stay the same throughout, by having the vector field be fixed on paths joining the $p_i$'s to $w_i$ inside $U_\alpha$. Thus, we may subtract $\vec p$ from the whole $D^2$-family of points in $\scr D^{k-d} U_\alpha$ and multiply by $-1$ in order to get a null-homotopy of $\vec a : S^1 \to \scr D^d U_\alpha$.

		Finally, we need to ensure canonicity of the homotopy back to the identity, up to contractible choice. We claim that this is automatic for reasons of vanishing homotopy groups. Indeed, the space of such homotopies takes the form of the homotopy-fiber
		\[ {\rm HoFib}\left[{\rm Map}^\simeq(\bb T_{\vec\alpha_{(0)}}, \bb T_{\vec\alpha_{(0)}}) \to {\rm Map}(\bb T_{\vec\alpha_{(0)}}, \scr D^d \Sigma^\circ)\right]. \]
		Both mapping spaces are disjoint unions of tori, and moreover the map from the first to the second induces an injection on $\pi_1$, which by the long-exact sequence associated to the fibration implies that the homotopy-fiber has contractible connected components, thus concluding the proof.
	\end{proof}

	Another convenient consequence of {\sc Lem.~\ref{lem:T}} is the following:

    \begin{COR}\label{cor:pathisfib}
    One has canonical homotopy equivalences
    \[ \hatP_s \cong \vec a \underset{\scr D^dU_\alpha}{\times^h} \bb T_\beta, \]
    \[ \rhatP_{s} \cong \vec a \underset{\scr D^dU_\alpha}{\times^h} \rdiv^d\Sigma^\circ. \]
    \end{COR}

    \begin{proof}
        This follows from the pasting law for homotopy fiber-products.
    \end{proof}

    Now, the homotopy types of these pathspaces can be calculated way more easily, because we have turned them into homotopy-fibers:

    \begin{PROP}\label{prop:spcalc}
        If $\vec \mu'$ denotes a disjoint union of meridians for the toroidal components of $\partial N'$, then we have that
        \begin{list}{$\cdot$}{}
            \item $\hatP_s$ is a disjoint union of equivalent tori, with
			    \begin{list}{--}{}
					\item $\pi_0$ a torsor over $H^2(Y, \vec w; \bb Z) \cong H^2(Y; \bb Z)$;
					\item $\pi_1$ isomorphic to $H^1(Y, \vec w; \bb Z)$;
				\end{list}
            \item $\rhatP_{s}$ is a disjoint union of equivalent tori, with
			    \begin{list}{--}{}
					\item $\pi_0$ a torsor over $H^2(Y' \setminus N', \vec \mu'; \underline{\bb Z}^v)$;
					\item $\pi_1$ isomorphic to $H^1(Y' \setminus N', \vec \mu'; \underline{\bb Z}^v) \overunderset{\text{\sc Prop.}}{\text{\sc\ref{prop:bigA}-i}}= H^1(Y' \setminus N', \partial N'; \underline{\bb Z}^v)$, i.e.~ $A(Y, \tau)$ from {\sc Def.~\ref{def:A}}.
				\end{list}
        \end{list}
    \end{PROP}

    To prove the majority of this result, we use the following observation, which will be sufficient for all except $\pi_0 \rhatP_s$ in the case $k > 1$, whose proof will be provided later through different means.
    \begin{PROP}\label{prop:fibertori}
        Let $\bb T, \bb T'$ be two tori (i.e.~ $K(\pi, 1)$'s with $\pi$ free abelian groups). An unbased map $f : \bb T \to \bb T'$ is uniquely determined up to homotopy by its induced map $f_*$ on $H_1$. Its homotopy-fiber $F$ over any point in $\bb T'$ is homotopy-equivalent to a disjoint union of equivalent tori, with
        \begin{list}{$\cdot$}{}
            \item $\pi_0F$ a torsor over ${\rm Cok}(f_*)$;
            \item $\pi_1F$ isomorphic to ${\rm Ker}(f_*)$.
        \end{list}
    \end{PROP}

    \begin{proof}
        Let us pick a $\bb Z$-basis $x^1, \ldots, x^k$ for $H^1(\bb T')$, inducing an equivalence $\bb T' \to (S^1)^{\times k}$, since $S^1$ is the classifying space for $H^1$. One has
        \[ [\bb T, \bb T'] \cong \prod_{i=1}^k [\bb T, S^1] \cong \prod_{i=1}^k H^1(\bb T; \bb Z). \]
        The equivalence sends a map $f : \bb T \to \bb T'$ to the tuple $(f^* x^1, \ldots, f^* x^k)$. This tuple uniquely determines $f^* : H^1(\bb T') \to H^1(\bb T)$, hence also its dual $f_* : H_1(\bb T) \to H_1(\bb T')$, thus concluding the proof of the first assertion.

        Next, we move on to proving the claim about $\pi_0$. We model the homotopy-fiber of $f$ over a point $t' \in \bb T'$ as parameterizing pairs consisting of a point $t \in \bb T$, together with a path $\gamma$ connecting $f(t)$ to $t'$ inside $\bb T'$. There is an action of $\pi_1(\bb T') = H_1(\bb T'; \bb Z)$ on $\pi_0$ of the homotopy-fiber by post-composing the path with a $t'$-based loop in $\bb T'$. This acts transitively, because given any two different paths $\gamma^{(1)}, \gamma^{(2)}$ connecting $f(t^{(1)})$ and $f(t^{(2)})$ to $t'$, we can homotope one of them to have identical points $t^{(1)} = t^{(2)}$ (since $\bb T$ is path connected), after which $\gamma^{(2)}$ is obtained from $\gamma^{(1)}$ by composing with the loop $\gamma^{(2)} \star (\gamma^{(1)})^{-1} \in \pi_1(\bb T')$. It remains to show that the kernel of the action is ${\rm Im}(f_*)$, so that this action descends to a free and transitive action of ${\rm Cok}(f_*)$ on $\pi_0F$. So assume that $t \in \bb T$ and $\gamma$ is a path from $f(t)$ to $t'$, and let $\delta \in \pi_1(\bb T')$ be a loop. To say that $\delta \star \gamma$ represents the same element in $\pi_0F$ of the homotopy-fiber as $\gamma$ translates into having a square of the form illustrated in {\sc Fig.~\ref{fig:square}}
		\begin{figure}
            \includegraphics[scale=.75]{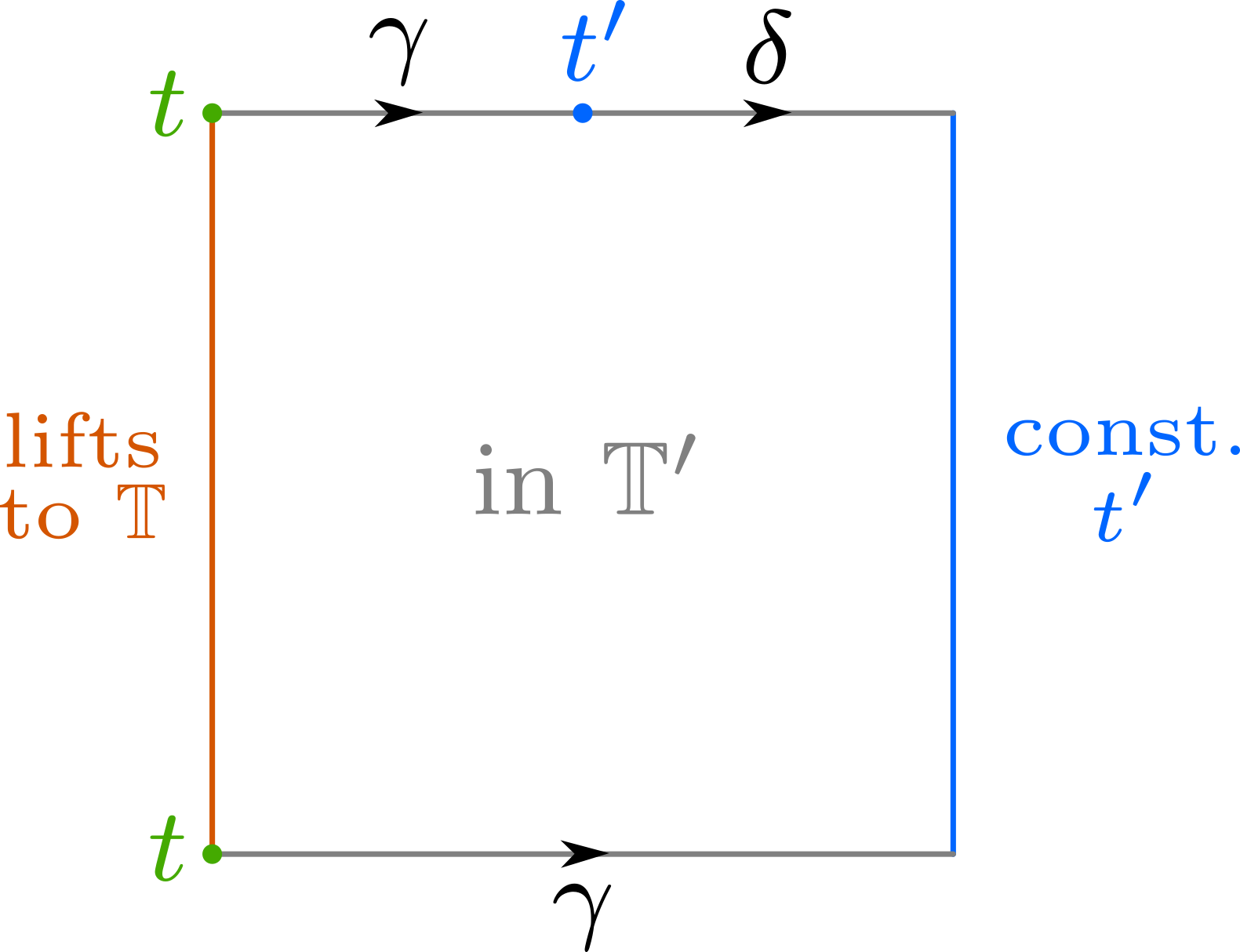}
			\caption{Illustration of a 1-parameter family of paths.}\label{fig:square}
		\end{figure}
        where the square is inside $\bb T'$ but the left side is equipped with a lift to $\bb T$. Alternatively, this diagram could be read as saying that $\delta$ is the conjugate of $f$ applied to this lift path in $\bb T$, conjugated by $\gamma$ to change basepoints. This is precisely the condition that $[\delta] \in {\rm Im}(f_*)$, thus concluding the proof that $\pi_0F$ is a torsor over ${\rm Cok}(f_*)$.
        
        Since the proof of transitivity of the action comes from a space-level construction, the argument can be extended to show that all connected components of the homotopy-fiber are homotopy-equivalent. Thus, it suffices to prove the result about $\pi_1F$ for only one of these connected components. Also, any two different choices of basepoint in $\bb T'$ lead to homotopy equivalent homotopy-fibers, since we may exploit path-connectedness of $\bb T'$ to join the two basepoints. Thus, we may assume $f$ to be a based map, i.e.~ we let $t \in \bb T$ a basepoint, and set $t' := f(t)$ a basepoint for $\bb T'$, and consider the connected component of the homotopy-fiber at the constant path. Only now may we use the long-exact sequence on homotopy groups:
        \[ 0 = \pi_2(\bb T') \to \pi_1 F \to \pi_1 \bb T \overset{f_*}\to \pi_1 \bb T', \]
        to conclude the computation of $\pi_1 F$.

        Lastly, the higher homotopy groups are seen to be zero by the same argument, which proves that the connected components of $F$ are indeed tori, up to homotopy equivalence.
    \end{proof}

    \begin{proof}[Proof of {\sc Prop.~\ref{prop:spcalc}}]
        By the Dold-Thom theorem, the homotopy groups $\pi_k$ of $\scr D^d U_\alpha$ and $\scr D^d \Sigma^\circ$ are equal to $\tilde H_k(U_\alpha; \bb Z)$ and $\tilde H_k(\Sigma^\circ; \bb Z)$ respectively. In particular, we are in position to apply {\sc Prop.~\ref{prop:fibertori}} for $\hatP_s$ and $\rhatP_s$, to the maps $\bb T_\beta \to \scr D^d U_\alpha$ and $\rdiv^d \Sigma^\circ \to \scr D^d U_\alpha$, cf. {\sc Cor.~\ref{cor:pathisfib}}.

        Beginning with $\hatP_s$, to understand the induced map on the level of $H_1$, let us note that the map in question on the level of spaces factors as
        \[ \bb T_\beta \to \scr D^d \vec \beta \to \scr D^d \Sigma^\circ \to \scr D^d U_\alpha, \]
        and the first map is furthermore the product of all the individual inclusions $\beta_i \to \scr D^1 \beta_i$. These are all equivalences, hence the first map is an equivalence, and so in effect the induced map on $H_1$ is just the composite
        \begin{equation}\label{eq:composite1}
            \mathbb Z\<\beta_i\> \hookrightarrow H_1(\Sigma) \twoheadrightarrow H_1(U_\alpha).
        \end{equation} 
        We wish to show that the kernel and cokernel of this map are isomorphic to $H_2$ and $H_1$ of $Y^\circ := Y \setminus \vec w$, respectively (this would imply the desired claim via Poincar\'e duality $H_*(Y \setminus \vec w) \cong H^{3-*}(Y, \vec w)$). We know that $Y^\circ$ deformation retracts onto $U_\alpha \cup D^{\vec \beta}$, and hence we get a long-exact sequence associated to the pair $(Y^\circ, U_\alpha)$:
        \[ 0 = H_2(U_\alpha) \to H_2(Y^\circ) \to H_2(Y^\circ, U_\alpha) = \bb Z\<\beta\> \overset\partial\to H_1(U_\alpha) \to H_1(Y^\circ) \to H_1(Y^\circ, U_\alpha) = 0. \] 
        The boundary map $\partial$ is precisely the composite (\ref{eq:composite1}) in question, hence proving the first part of the proposition, after Poincar\'e duality.

        Moving on to $\rhatP_{s}$, we recall that $\rdiv^d\Sigma^\circ$ is once again homotopy equivalent to a disjoint union of tori, each satisfying $\pi_1 = H_1(\Sigma'^\circ)$, whose inclusion in $\scr D^d \Sigma^\circ$ induces the obvious transfer map on the level of $\pi_1$. Unfortunately, we cannot apply the Lemma directly for the computation of $\pi_0$, due to there being several components of $\rdiv^d\Sigma^\circ$ when $k > 1$, and doing so would only determine the group up to an unidentified extension; for the proof of the statement on $\pi_0 \rhatP_s$, see the very last statements of both {\sc Thms.~\ref{thm:rspinctors},~\ref{thm:comp}} that we will prove later on. However, we may apply this lemma perfectly fine on $\pi_1$; indeed, this transfer map is always injective, because the pushforward along the projection provides an inverse upon inverting 2, and the groups in question are free abelian, hence not having any 2-torsion. Thus, we may once again apply {\sc Prop.~\ref{prop:fibertori}}, to reduce the computation of $\pi_1$ to computing the kernel of the composite map
        \begin{equation}\label{eq:composite2}
            H_1(\Sigma'^\circ) \hookrightarrow H_1(\Sigma^\circ) \twoheadrightarrow H_1(U_\alpha).
        \end{equation}
		\begin{figure}
			\centering\includegraphics{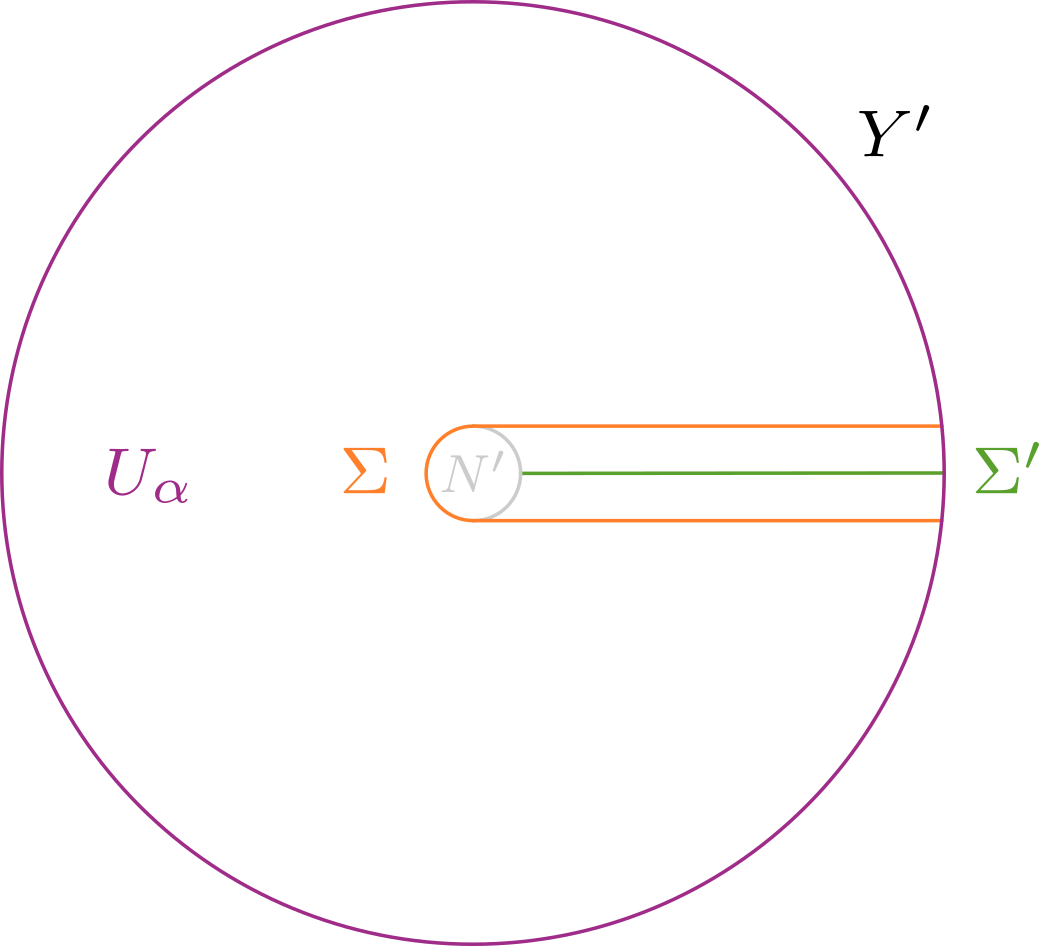}
			\caption{Depiction of a decomposition of $Y' \setminus N'$ as a union of $U_\alpha$ and $N\Sigma'$ over $\Sigma$. The natural map $U_\alpha \hookrightarrow Y \twoheadrightarrow Y'$ is not injective, but by trimming a small collar neighborhood of $U_\alpha$ we do not affect its homotopy type, and we obtain precisely this picture.}\label{fig:Ypdec}
		\end{figure}
        As before, we would like to identify this composite with the boundary map in the long-exact sequence associated to a pair; in this case, we will use $(Y' \setminus N', U_\alpha)$ illustrated and explained in {\sc Fig.~\ref{fig:Ypdec}}:
        \begin{equation*}
            \begin{gathered}
                H_2(U_\alpha; \underline{\bb Z}^v) \to H_2(Y' \setminus N'; \underline{\bb Z}^v) \to H_2(Y' \setminus N', U_\alpha; \underline{\bb Z}^v) \overset\partial\to \\
                \overset\partial\to H_1(U_\alpha; \underline{\bb Z}^v) \to H_1(Y' \setminus N'; \underline{\bb Z}^v) \to H_1(Y' \setminus N', U_\alpha; \underline{\bb Z}^v).
            \end{gathered}
        \end{equation*}
        The first term is zero, so in particular $H_2(Y' \setminus N'; \underline{\bb Z}^v) \cong H^1(Y' \setminus N', \partial N'; \underline{\bb Z}^v)$ is the kernel of the connecting homomorphism labelled $\partial$. For the third term, we apply excision and the homology Thom isomorphism  with twisted coefficients for the normal bundle $N\Sigma' := T Y' / T \Sigma'$ to get
        \begin{equation}\label{eq:excth} \begin{aligned}
			H_*(Y' \setminus N', U_\alpha; \underline{\bb Z}^v) \underset{\rm exc.}{\cong} H_*(\bb D(N\Sigma'), \bb S(N \Sigma'); \underline{\bb Z}^v) \underset{\rm Thom}{\cong} H_{*-1}(\Sigma'; \bb Z).
		\end{aligned}\end{equation}
        It remains to show that, under the identification (\ref{eq:excth}), the map $\partial$ agrees with (\ref{eq:composite2}). Indeed, given a submanifold $V$ in $\Sigma'$ representing a certain primitive homology class, its inverse image under the Thom isomorphism is given by the fundamental class of $\bb D(N\Sigma'|_V)$, which can also be viewed as a class in $H_*(Y' \setminus N', U_\alpha; \underline{\bb Z}^v)$ under excision. Its image under the connecting homomorphism is given by the boundary, i.e.~ $\bb S(N\Sigma'|_V)$, i.e.~ precisely the transfer map on $V$. This finishes the proof, after applying Poincar\'e duality.
    \end{proof}

    \begin{PROP}
        The projection map
        \[ \rhatP_s \cong \bb T_\alpha \underset{\scr D^d\Sigma^\circ}{\times^h} \rdiv^d\Sigma^\circ \twoheadrightarrow \rdiv^d\Sigma^\circ \]
        induces a homomorphism on fundamental groups which, after identifying the domain with $H^1(Y' \setminus N', \partial N'; \underline{\bb Z}^v)$, and the codomain with $H^1(\Sigma', \partial \Sigma'; \bb Z)$, is given simply by restriction.
    \end{PROP}

    \begin{proof}
        Recall from the proof of the computation of $\pi_1(\rhatP_s)$ that we have produced a commutative diagram
        \[ \begin{tikzcd}[row sep = small]
            & H_2(Y' \setminus N', U_\alpha; \underline{\bb Z}^v) \ar[dd,"\sim"] \\
            H_2(Y' \setminus N'; \underline{\bb Z}^v) \ar[ru]\ar[rd] & \\
            & H_1(\Sigma'; {\bb Z}),
        \end{tikzcd} \]
        where the bottom map is the one we are trying to interpret (up to Poincar\'e duality), the top map is given by the standard inclusion, and the vertical isomorphism is given by (\ref{eq:excth}). This latter isomorphism can alternatively be thought of as intersecting with $\Sigma'$ inside $Y' \setminus N'$, because excision is performed away from $\Sigma'$, and because intersecting with the 0-section in the disk bundle inverts the Thom isomorphism. After applying Poincar\'e duality to both sides, this boils down to mere restriction on the level of cohomology, as claimed.
    \end{proof}

\section{Pathspaces and real \texorpdfstring{${\rm Spin}^c$}{Spin-c}-structures}\label{sec:rs}

    In this section we recall the notion of (relative, real) ${\rm Spin}^c$-structure, with a focus on dimension $n = 3$, and its crucial relation to the stable (real) pathspace associated to a (real) Heegaard diagram. When it comes to ${\rm Spin}^c$-structures, there is a multitude of equivalent definitions that are all relevant to the subject, so we recall them here.

    \subsection{Various equivalent notions of real \texorpdfstring{${\rm Spin}^c$}{Spin-c}-structures} For a very thorough analysis of all the different flavors of real ${\rm Spin}^c$-structures, see Baraglia \cite[\S6]{Ba25}; the one relevant for us will be the $k = 2$ case in his terminology.
    
    The most straightforward definition in the non-equivariant setting, motivated by Seiberg-Witten and monopole theory, goes as follows:
    \begin{DEF}
        Let $V^n$ be a rank-$n$ Euclidean $\bb R$-vector bundle over a space $X$. A \keywd{complex spinor} bundle $(S, \rho)$ for $(V, X)$ of \keywd{positive}, resp. \keywd{negative type} is a complex Hermitian vector bundle $S$, together with a homomorphism
        \[ \rho : V \to \scr E{\it nd}_\bb C(S) \]
        of $\bb R$-vector bundles, such that $\rho(v)$ is an isometry satisfying $\rho(v)^2~=~$ \keywd{$+{\rm id}_S$}, resp. \keywd{$-{\rm id}_S$} for all $v \in V$ of norm $\|v\| = 1$.
    \end{DEF}

    \begin{RMK}
        If $(S, \rho^+)$ is a spinor bundle of \emph{positive type}, then $(S, \rho^-)$ with $\rho^- := i \cdot \rho^+$ is a spinor bundle of \emph{negative type}, thus proving that the two notions are easily interchangeable; this will cease to be the case when we introduce the real notions, due to sign issues stemming from complex conjugation.

        We sometimes write $\rho_v$ for $\rho(v)$, or $\rho(v, \Phi)$ for $\rho_v\Phi$ for any spinor $\Phi \in \Gamma(S)$. Another way to think about a spinor bundle is via \emph{Clifford algebras}: define the Clifford bundle
        \[ {\rm Cl}_\pm^\bb R(V) := \scr T^* V / (v \otimes v \mp \|v\|^2)_{v \in V} \]
        of algebras ($\scr T^*$ stands for the tensor algebra); then, a complex spinor bundle is simply a bundle of complex representations of this real algebra bundle. The observation at the beginning of the remark translates into the fact that the complexifications ${\rm Cl}^\bb C_\pm := \bb C \otimes {\rm Cl}^\bb R_\pm$ corresponding to the two sign choices become isomorphic to each other, but this is in fact not true if we do not complexify.
    \end{RMK}

    \begin{DEF}\label{def:spinc}
        Assume further that the bundle $V^n$ is equipped with an orientation. In the case that $n$ is even, the standard theory of Clifford algebras shows that ${\rm Cl}^\bb C(\bb R^n)$ is isomorphic to a matrix algebra over $\bb C$, and therefore it admits only one isomorphism class of irreducible representation (from now on, ``irrep''). In the case that $n$ is odd, the algebra ${\rm Cl}^\bb C(\bb R^n)$ breaks up as a direct product of two matrix algebras over $\bb C$ of equal size, so there are two distinct isomorphism classes of irreps. Given any oriented orthonormal basis $e_1, \ldots, e_n$ for a fiber $V_x$, there is a well-defined volume element $\omega_\bb C := i^{\frac{n(n+1)}2} e_1 \cdots e_n$ that lives in the center of the algebra, and that squares to $1$, inducing the two expected idempotents $\tfrac 12({\rm id} \pm \omega_\bb C)$. Thanks to the orientation on $V$, this $\omega_\bb C$ is well-defined globally, thus giving us a global choice of which irrep to be considered as standard, e.g.~ without loss of generality the kernel of $\tfrac 12({\rm id} - \omega)$.

        We may now define a \keywd{${\rm Spin}^c$-structure} of positive, resp. negative type on $(X, V)$ to be a complex spinor bundle $(S, \rho)$ of positive, resp. negative type, such that $S$ is a Clifford irrep at every fiber, and in the odd case given by the choice induced by the orientation. As mentioned previously, the positive and negative notions are equivalent in this yet non-equivariant setting.
    
        This is the definition that is most natural from the point of view of Seiberg-Witten and monopole theories. A more algebraic definition comes form understanding the structure group of this data. Indeed, we may define the \keywd{${\rm Spin}^c_n$ Lie group} as the automorphism group of the datum $(V_0, S_0, \rho_0)$, consisting of one's favorite choice of rank-$n$ oriented Euclidean $\bb R$-vector space $V_0$, and irreducible complex Clifford module $(S_0, \rho_0)$ (whose isomorphism class is that induced by the orientation on $V_0$ in the odd-dimensional case). That is,
        \[ {\rm Spin}^{\rm c}_n(V_0, S_0, \rho_0) := \{ (f, \sigma) \in {\rm SO}(V_0) \times {\rm U}(S) : \rho_0(f v, \sigma \Phi) = \sigma\rho_0(v, \Phi), \forall v \in V_0, \forall \Phi \in S_0\}. \]
        Given an isomorphism $\varphi : (V_0, S_0, \rho_0) \to (V_1, S_1, \rho_1)$, there is an induced conjugation map
        \[ \varphi \circ - \circ \varphi^{-1} : {\rm Spin}^{\rm c}_n(V_0, S_0, \rho_0) \to {\rm Spin}^{\rm c}_n(V_1, S_1, \rho_1), \]
        so in particular the structure groups are independent of the choice of $(V_0, S_0, \rho_0)$, which justifies dropping it from the notation. There is a central extension
        \begin{equation}\label{eq:sesspinc}
            1 \to S^1 \to {\rm Spin}^c_n \to {\rm SO}_n \to 1,
        \end{equation}
        where the projection is given by $(f, \sigma) \mapsto f$, and the kernel is isomorphic to $S^1$ due to the fact that the endomorphism algebra of any complex irrep is isomorphic to $\bb C$. We will say a bit more about this extension shortly.
    \end{DEF}

    \begin{DEF} \cite{Li22}
        Now, we turn to real structures, for which the positive and negative type distinction really does matter. Assume further that the Euclidean bundle $(V, X)$ is endowed with an involution $\tau$ (denoted the same both on the base and on the total space). A \keywd{real ${\rm Spin}^c$-structure} of positive, resp. negative type is an irreducible spinor bundle $(S, \rho)$ of either positive or negative type as in the previous definition, which is additionally endowed with an involution $I : S \to S$ covering $\tau$, that is conjugate-linear with respect to the complex structure on $S$, and furthermore such that it respects the Clifford multiplication $\rho$:
        \[ \rho(\tau v, I \Phi) = I\rho(v, \Phi), \quad \forall v \in V, \forall \Phi \in S. \]
        Such a structure in particular implies that the original spinor bundle $(S, \rho)$, as well as the $\tau$-pullback of the conjugate $(\bar S, \rho)$ (with the underlying $S$ and $\rho$ are the same, but the complex structure on $S$ conjugated) are isomorphic as Clifford module bundles, and hence it would imply that the complex irreps of ${\rm Cl}_\pm (\bb R^n)$ are \emph{not} swapped by the process of conjugation of complex representations. By the classification of real Clifford algebras, it follows that:
        \begin{list}{$\cdot$}{}
            \item $n \not\equiv 3 ~ {\rm mod} ~ 4$ for the case of \emph{positive} type;
            \item $n \not\equiv 1 ~ {\rm mod} ~ 4$ for the case of \emph{negative} type;
        \end{list}
        in order for such a structure to exist. In particular, in the case $n = 3$ of interest of the present paper, \emph{a real ${\it Spin}^c_\textit{3}$-structure must be negative}, and hence we will subsequently drop the adjective ``negative''.

        We also seek to understand real ${\rm Spin}^c$-structures from the point of view of structure groups, an approach that has already appeared in the literature in a few places, e.g.~ \cite{BH24, GM25}. First, we observe that there is a naturally occurring \keywd{conjugation map} on ${\rm Spin}^c_n$, defined as the group-theoretic conjugation by any conjugate-linear automorphism $I_0$ of $(S_0, \rho_0)$ as a Clifford module over $V_0$ (one again needs to ensure the respective congruence mod 4 condition above for such an automorphism to exist):
        \[ \overline{(f, \sigma)} := (f, I_0 \circ \sigma \circ I_0^{-1}). \]
        Indeed, any two such conjugate-linear automorphisms $I_0, I_1$ are related by a scalar $\lambda \cdot {\rm id} = I_1 \circ I_0^{-1}$, which does not affect the induced conjugation map. This conjugation respects the short exact sequence (\ref{eq:sesspinc}), inducing the identity on the ${\rm SO}_n$-quotient, and the standard complex conjugation map on $S^1$. Its fixed-points induce a short-exact sequence
        \[ 1 \to \pm 1 \to {\rm Spin}_n \to {\rm SO}_n \to 1 \]
        from which one can easily deduce the formula
        \[ {\rm Spin}^c_n \cong ({\rm Spin}_n \times S^1) / (-1, -1), \]
        where the conjugation map acts trivially on ${\rm Spin}_n$, and via the standard complex conjugation map on $S^1$.

        Now, we may give a principal bundle definition of a real ${\rm Spin}^c$-structure. The original oriented vector bundle $V$ can be encoded as a (right) principal ${\rm SO}_n$-bundle ${\rm Fr}(V)$, which inherits an involution also called $\tau$ from $V$. The spinor bundle $(S, \rho)$ gives us a principal ${\rm Spin}^c_n$-bundle $P$ via the fiber-wise construction
        \[ P_x := {\rm Iso}_{\bb C\text{-lin}}\big((V_0, S_0, \rho_0), (V_x, S_x, \rho|_x)\big) \]
        endowed with the right-action given by pre-composition with automorphisms of $(V_0, S_0, \rho_0)$. Now, post-composition of any element $\psi \in P_x$ with the conjugate-linear isomorphism $I : (V_x, S_x, \rho|_x) \overset\sim\to (V_{\tau x}, S_{\tau x}, \rho|_{\tau x})$ yields a \emph{conjugate}-linear isomorphism $I \circ \psi$, which factors uniquely through $I_0$:
        \[ \begin{tikzcd}
            (V_0, S_0, \rho_0) \rar["\psi"]\dar["I_0"] & (V_x, S_x, \rho|_x) \dar["I"] \\
            (V_0, S_0, \rho_0) \rar["I \circ \psi \circ I_0^{-1}"] & (V_{\tau x}, S_{\tau x}, \rho|_{\tau x})
        \end{tikzcd} \]
        resulting in an element $I \circ \psi \circ I_0^{-1} \in P_{\tau x}$. We denote this process $\psi \mapsto I \circ \psi \circ I_0^{-1}$ by the map $I_* : P \to \tau^*P$; this is not a map of ${\rm Spin}^c_n$-bundles, but rather covers the conjugation automorphism on ${\rm Spin}^c_n$. The condition $I \circ I = {\rm id}_S$ translates into $I_* \circ \tau^* I_* = \lambda^{-1} \cdot {\rm id}_P$, where $\lambda \in S^1 \subset \bb C^\times$ is the unique scalar such that 
        \begin{equation}\label{eq:sqtolam}
            I_0 \circ I_0 = \lambda \cdot {\rm id}_S,
        \end{equation}
        (such a scalar must exist, because unitary automorphisms of a complex unitary representation are always given by such scalars.)

        In the special case that $n = 3$, we can see concretely that \keywd{$\lambda = -1$}. Indeed, an explicit model of the Clifford algebra action of standard-oriented $\bb R^3 = \bb R\<e_1, e_2, e_3\>$ on the corresponding irreducible Clifford module $\bb C^2$ is given by the \emph{Pauli matrices}
        \[ E_1 = \left[\begin{matrix}
            -i & 0 \\ 0 & i
        \end{matrix}\right], E_2 = \left[\begin{matrix}
            0 & -1 \\ 1 & 0
        \end{matrix}\right], E_3 = \left[\begin{matrix}
            0 & -i \\ -i & 0
        \end{matrix}\right]. \]
        Indeed, $E_i^2 = -{\rm id}, E_i E_j = -E_j E_i$, and $E_1 E_2 E_3 = +{\rm id}$ (this latter condition is to ensure that we have picked the correct isomorphism class of irreducible Clifford module, corresponding to the choice of oriented basis $e_1, e_2, e_3$). A conjugate-linear automorphism of $\bb C^2$ is uniquely represented by the map $\Phi \mapsto A \bar \Phi, ~\forall \Phi \in \bb C^2$ for some matrix $A \in M_2(\bb C)$, and the condition of respecting the Clifford action is given by
        \[ A\bar E_1 = E_1 A, \quad A\bar E_2 = E_2 A, \quad A\bar E_3 = E_3 A,    \]
        which when solved explicitly give that $A$ must be a multiple of $\left[\begin{smallmatrix}
            0 & -1 \\ 1 & 0
        \end{smallmatrix}\right]$. Thus $I_0$ squares to $-{\rm id}$, as promised. The big conclusion is that the following group-theoretic definition is equivalent to the spinor definition of a real ${\rm Spin}^c_3$-structure:
    \end{DEF}

    \begin{DEF}
        Given an $SO_3$-vector bundle $V$, with oriented frame bundle ${\rm Fr}(V)$, and with involution denoted $\tau$ on both the total spaces and on the base $X$, a \keywd{real ${\rm Spin}^c_3$-structure} on $V$ is given by a principal ${\rm Spin}^c_3$-bundle $P$, with an identification $P \times_{{\rm Spin}^c_3} {\rm SO}_3 \cong {\rm Fr}(V)$, equipped with a map $I_* : P \to \tau^* P$ covering the conjugation automorphism on ${\rm Spin}^c_3$, and covering the structural identification ${\rm Fr}(V) \cong \tau^* {\rm Fr}(V)$ on the induced $SO_3$-bundle, and further satisfying that the composite
        \[ P \overset{I_*}\longrightarrow \tau^*P \overset{\tau^* I_*}\longrightarrow P \] 
        is given by multiplication by $-1 \in S^1 \subset {\rm Spin}^c_3$.
    \end{DEF}

    \begin{PROP}
        The natural homomorphism of Lie groups
        \[ {\rm Spin}^c_3 \cong {\rm Aut}(V_0, S_0, \rho_0) \overset{\rm forg}\longrightarrow {\rm Aut}(S_0) \cong U_2 \]
        is an isomorphism.
    \end{PROP}

    \begin{proof}
        Both the domain and the target are 4-dimensional real Lie groups, hence it suffices to show that the map is a bijective submersion. In fact, the submersion claim is automatic from bijectivity, since the set of regular values is non-empty by Sard's lemma, and in addition it is invariant under the $U_2$ action, hence must be the whole codomain. Bijectivity in turn is implied by mere injectivity, due to Brouwer's open mapping theorem and the closedness of maps between compact spaces.

        So, let us assume towards a contradiction that there is an element $(f, \sigma) \in {\rm Spin}^c_3$ that maps to the identity in ${\rm Aut}(S_0)$, i.e.~ that $\sigma = {\rm id}_S$. This means that
        \[ \rho_{f v} \Phi = \rho_v \Phi, \quad \forall v \in V_0, \Phi \in S_0.  \]
        If we assume towards a contradiction that $f \neq {\rm id}_{V_0}$, it means that there is some $v \in V_0$ such that $fv - v \neq 0$. But then, this means that $\rho_{fv - v}$ is nonzero, since nonzero elements of $V_0$ are units in the Clifford algebra. This concludes the proof of our claim.
    \end{proof}

    The following corollary gives us a very nice geometric way to visualize Clifford multiplication, and hence an alternative definition of real ${\rm Spin}^c_3$-structure without mentioning Clifford multiplication explicitly.

    \begin{COR} The following are true:
        \begin{enumerate}
            \item[\sc i.] There is a diffeomorphism
            \[ {\rm Ker}(\rho_v - i \cdot {\rm id}_S) : \bb S(V^3) \to \bb P(S^2) \]
            from the unit sphere bundle of the ${\rm SO}_3$-bundle $V$ to the projectivization of the spinor bundle $S$, which is also an isometry after scaling the metric uniformly (e.g.~ by a factor of $1/2$ on the domain to make the curvatures agree).
            \item[\sc ii.] Given any $\bb C$-linear automorphism $\sigma : S \overset\sim\to S$, there is a unique extension to an automorphism $(f, \sigma)$ of $(V, S, \rho)$, and moreover the identification in part {\sc i} respects the maps induced by $f$ and $\sigma$, respectively. Similarly, if $I : S \overset\sim\to S$ is a conjugate-linear automorphism, there is a unique automorphism $f : V \to V$ such that the identification in part {\sc i} respects the maps induced by $-f$ and $I$, respectively (note the sign change).
        \end{enumerate}
    \end{COR}

    \begin{proof}
        Let us begin with part {\sc i}. Since we are in the negative case, let us recall that $\rho_v$ is an automorphism of $S^2$ that squares to $-1$, hence it must be diagonalizable with eigenvalues belonging to $\{-i, +i\}$. In fact, each choice of $\pm i$ appears with multiplicity 1, since otherwise $\rho_v$ would commute with the rest of the Clifford action, which is a contradiction. Thus, ${\rm Ker}(\rho_v - i \cdot {\rm id}_S)$ is a 1-dimensional complex vector subspace of $S$, hence confirming that the codomain of the map is indeed correct.

        To see that is a diffeomorphism, we appeal to the proposition we just proved: the group ${\rm Spin}^c_3 \cong U_2$ acts transitively on both sides, with kernel given by the center $S^1 \subset {\rm Spin}^c_3$ of the group, hence proving the desired diffeomorphism. If we endow both sides with the homogeneous space metric, we get the desired isometry; this agrees with the more standard metric in both sides up to a scaling factor, e.g.~ if we take the metric on $\bb P(S^2) \cong \bb{CP}^1$ making the Hopf fibration $S^3 \twoheadrightarrow \bb CP^1$ a metric submersion, then the latter is a sphere of radius $1/2$. This concludes our proof of {\sc i}.

        Part {\sc ii} is quite straightforward, and the point is that a $\bb C$-linear automorphism preserves the $+i$ and $-i$ eigenspaces, whereas a conjugate-linear automorphism swaps them; the operation that takes a complex line bundle in $S$ to its orthogonal complement corresponds precisely to taking the antipode on the level of the projectivization, viewed as a metric 2-sphere of radius $1/2$, which is responsible for the minus sign in the conjugate-linear case.
    \end{proof}

    \begin{DEF}\label{def:sphproj}
        In particular, this allows us to reinterpret the notion of real ${\rm Spin}^c_3$-structure in a geometric way: such a structure consists of a rank-2 Hermitian complex vector bundle $S$, together with a metric identification $\bb S_{1/2}(V) \cong \bb P(S)$, and further equipped with a conjugate-linear involution $I : S \overset\sim\to S$ covering the involution $\tau$ on the base, such that the map $\bb P(I)$ induced on $\bb P(S)$ corresponds to the map $\bb S_{1/2}(-\tau)$ induced by $-\tau$.
    \end{DEF}

    If $A \subset X$ is a closed $\tau$-equivariant subset, and a preferred choice of real ${\rm Spin}^c$-structure on the bundle $X$ is already provided on $V|_A$, then we call any extension thereof to the whole of $V$ a \keywd{relative} real ${\rm Spin}^c$-structure. The case of interest is given by $X = Y^3$ an oriented closed 3-manifold, with the involution $\tau$ on $Y$ arising as a branching involution, the bundle $E = TY$ with involution given by $d\tau$, and the closed subset $A = \vec w \subset X$ consisting of the set of basepoints (the real ${\rm Spin}^c$-structure on said basepoints can be chosen arbitrarily, since any two choices are isomorphic). We denote by \keywd{${\rm Spin}^c(Y, \vec w)$} and \keywd{${\rm rSpin}^c(Y, \vec w)$} the sets of \emph{isomorphism classes} of relative ${\rm Spin}^c$ and relative real ${\rm Spin}^c$-structures on $Y$ rel $\vec w$. Later in this section we will refine these sets of isomorphism classes to actual spaces, and for this we will have to be more specific concerning the ${\rm Spin}^c$-structure on the basepoints.

    \begin{DEF}\label{def:rvect}
        In the context of 3-manifolds, yet another definition of (real, relative) ${\rm Spin}^c$-structure, that is easier to work with in the context of Heegaard Floer theory, is based on vector fields, cf. \cite{Tu97, GM25}. Two nowhere-vanishing vector fields on a 3-manifold $Y$ are said to be \emph{cohomologous} if and only if they are homotopic away from an embedded ball. It is a standard fact that the set of cohomology classes of nowhere vanishing vector fields on $Y$ is in canonical bijection with ${\rm Spin}^c(Y)$, and likewise if the value of the vector field is fixed on $\vec w$, the same is true regarding the relative ${\rm Spin}^c(Y, \vec w)$; we will offer a proof of this fact in the upgraded form on the level of spaces of choices, not just equivalence classes.

        When it comes to real ${\rm Spin}^c$-structures, the story is a bit more subtle. One now imposes the condition that the nowhere-vanishing vector fields be \emph{real}, i.e.~ that $\tau^* v = -v$. However, the naive modification of the condition of being cohomologous is not adequate for describing real ${\rm Spin}^c$-structures. Instead, following \cite{GM25}, two real vector fields are said to be \emph{real cohomologous} if the underlying vector fields are cohomologous, and additionally the orthogonal complements are isomorphic as Atiyah-real line bundles, whose definition we now recall:
    \end{DEF}

    \begin{DEF}
        An \keywd{Atiyah-real line bundle} on a space $X$ with involution $\tau$ is a complex line bundle $\scr L$ with an involution $I$ on its total space, covering $\tau$ on the base $X$, and which in addition must be \emph{conjugate}-linear, i.e.~ $I(\lambda \cdot s) = \bar\lambda \cdot I(s)$.
    \end{DEF}
    
    Going back to the notion of real cohomology of real vector fields, we note that the orthogonal complement of a nowhere vanishing vector field has the structure group $SO_2 \cong U_1$, and hence it is a complex line bundle; moreover, if the vector field is real, then the involution is orientation-reversing, and therefore it is conjugate-linear in the line bundle language, hence inducing an Atiyah-real structure on the orthogonal complement. It is a theorem of \cite[\sc\S3.6]{GM25} that the set of real cohomology classes of real vector fields is in natural bijection with real ${\rm Spin}^c$-structures on $Y$, and likewise in the relative case. We will also produce a space-level analogue of this theorem in this section.

    Before addressing the more sophisticated space-level constructions, we introduce here some important ideas that can be employed in the study of real ${\rm Spin}^c$-structures, which to our knowledge do not directly follow from \cite{GM25}'s work.

    \begin{DEF}
        Given a real ${\rm Spin}^c$-structure $\mathfrak s$ on a $\tau$-equivariant bundle $(E, X, \tau)$, and an Atiyah-real line bundle $\scr L$ on $(X, \tau)$, one can form the \emph{tensor product} real ${\rm Spin}^c$-structure $\mathfrak s \otimes \scr L$, induced via the homomorphism of structure groups ${\rm Spin}^c_n \times S^1 \to {\rm Spin}^c_n$, defined in the following manner:
        \[ ([u, z], w) \longmapsto [u, z \cdot w]. \]
        This map of structure groups is associative with respect to multiplying by several $S^1$-tensor factors on the right, thus making the tensor product itself associative, i.e.~ $\mathfrak s \otimes \scr L \otimes \scr L'$ is unambiguous.

        Dually, given two real ${\rm Spin}^c$-structures $\mathfrak s, \mathfrak s'$, there is a \emph{Hom construction} producing an Atiyah-real line bundle $\scr H(\mathfrak s, \mathfrak s')$, defined fiberwise as the space of ${\rm Spin}^c_n$-equivariant maps $P|_x \to P'|_{x'}$ that induce the identity on ${\rm Fr}(E)|_x$, where $P$ and $P'$ are the principal bundles that define the real ${\rm Spin}^c$-structures $\fk s, \fk s'$. This space is an $S^1$-torsor via the $S^1$-factor in the definition of ${\rm Spin}^c$, and can be interpreted as an Atiyah-real line bundle. There is a tautological isomorphism
        \[ \mathfrak s \otimes \scr H(\mathfrak s, \mathfrak s') \cong \mathfrak s' \]
        given by evaluation, which establishes the proof of the following fact:
    \end{DEF}

    \begin{PROP}
        The set ${\rm rSpin}^c(Y, \vec w)$ is naturally a torsor over the set ${\rm rLine}(Y, \vec w)$ of isomorphism classes of Atiyah-real line bundles trivialized at $\vec w$.
    \end{PROP}

    Atiyah-real line bundles can be understood in terms of equivariant sheaf cohomology and Borel cohomology \cite{St78} in the general case, as pointed out by Baraglia \cite[\sc\S2]{Ba26}. When it comes to our specific case of $Y^3$ with branching involution, it can be understood in simpler terms of an equivariant notion of Chern class not relying on any equivariant form of cohomology, which we now introduce:

    \begin{PROP}\label{prop:rc1}
        We have the following:
        \begin{enumerate}
            \item Given a space $X$ endowed with a \emph{free} involution $\tau$, and an Atiyah-real line bundle $\scr L$, there is a well-defined \keywd {real Chern class} $c_1^R(\scr L) \in H^2(X / \tau; \underline{\bb Z}^\pm)$, where the latter signifies the local system with fiber $\mathbb Z$ and monodromy given by the 1st Stiefel-Whitney class of the $C_2$-principal bundle $X \to X/\tau$. 
            \item This real Chern class $c_1^R(\scr L)$ has the property that its reduction mod 2 recovers $w_2(\scr L/I)$, and also that when pulled back to the total space, it recovers the ordinary Chern class $c_1(\scr L)$ of the line bundle.
            \item The real Chern class establishes an abelian group isomorphism 
            \[ c_1^R : {\rm rLine}(X, A) \overset\sim\to H^2(X/\tau, A/\tau; \underline{\bb Z}^\pm) \]
            for any closed $\tau$-invariant subset $A \subset X$, where the former is the set of isomorphism classes of Atiyah-real line bundles on $X$ trivialized on $A$, with group operation given by tensoring.
        \end{enumerate}
    \end{PROP}

    \begin{proof}
        For part {\sc i}, let us notice that the data of $(X, \tau, \scr L, I)$ is uniquely determined by $(X / \tau, \scr L / I)$ as an $O_2$-bundle. Indeed, we may recover the covering space $\pi : X \to X/\tau$ from the first Stiefel-Whitney class of the rank-2 Euclidean vector bundle $\scr L / I$ over $\mathbb R$, and then the pull-back of $\scr L / I$ along $\pi$ naturally becomes an $SO_2 \cong U_1$-bundle, with the deck involution being conjugate-linear. Consequently, an Atiyah-real line bundle over a space with free involution can be classified by a map $X / \tau \to BO_2$. We define $c_1^R(\scr L)$ as the pull-back of a universal element $c_1^R \in H^2(BO_2; \underline{\mathbb Z}^{w_1})$ which we now describe. Running the Leray-Serre spectral sequence for the fibration $(\mathbb {CP}^\infty \cong BSO_2, \mathbb Z) \to (BO_2, \underline{\mathbb Z}^{w_1}) \overset{\rm det}\longrightarrow BO_1$ gives us
        \[ E_2^{pq} = H^p(\mathbb{RP}^\infty; \underline H^q(\mathbb {CP}^\infty; \mathbb Z)) \]
        where $\underline H$ signifies that we are considering the corresponding cohomology group as a local system over $\mathbb {RP}^\infty$, i.e.~ as a $C_2$-representation. The involution on $H^*(\mathbb {CP}^\infty; \mathbb Z)$ is given by \emph{minus} the map induced by the conjugation on $\mathbb {CP}^\infty$, where the minus sign comes from the fact that the trivialization $\pi^* \underline{\mathbb Z}^{w_1} \cong \mathbb Z$ of the local system anti-commutes with the deck involution. In particular, we get the following relevant entries of the spectral sequence:
        \[ E_2^{02} = \mathbb Z, E_2^{11} = 0, E_2^{20} = 0, \quad E_2^{21} = 0, E_2^{30} = \mathbb Z/2. \]
        Hence, the only differential that could affect the entries on the diagonal $p + q = 2$ is $d_3 : E_3^{02} \to E_3^{30}$. This cannot be nonzero, due to the existence of a section of the determinant map $BO_2 \to BO_1$, namely the map that classifies direct summing with a trivial copy of $\mathbb R$, which is compatible with the local system $\underline{\mathbb Z}^{w_1}$. Indeed, this would provide a map of spectral sequences from the one above to the new spectral sequence $^\prime E_*^{**}$ corresponding to $* \to \mathbb{RP}^\infty \to \mathbb{RP}^\infty$, forcing ${\rm id} \circ d_3^{02}$ to factor through the new $^\prime E_3^{02} = 0$.
    
        Consequently, we learn that $H^2(BO_2; \mathbb \underline{\bb Z}^{w_1}) \to H^2(\mathbb {CP}^\infty; \mathbb Z)$ is an isomorphism and we define $c_1^R$ to be the preimage of $c_1$. From this, it also follows that the reduction mod 2 of $c_1^R$ is $w_2$, as expected from the first part of claim {\sc ii}: first of all, thanks to the commutativity of the inclusion of the fiber $\mathbb {CP}^\infty \to BO_2$ with mod-2 reduction, we deduce that $c_1^R~{\rm mod}~2$ is either $w_2$ or $w_2 + w_1^2$; secondly, we may exclude the second option by instead pulling-back along the section $BO_1 \to BO_2$ mentioned earlier. The rest of part {\sc ii} follows from the very definition of $c_1^R$ as the preimage of $c_1$ under the pull-back map.

        Part {\sc iii} is proven by obstruction theory, in a manner analogous to the classical case: by putting a cell structure on the pair $(X/ \tau, A/\tau)$ and pulling it back to a $C_2$-equivariant cell structure on $(X, A)$, it follows that any Atiyah-real line bundle rel $A$ is trivial on the relative 1-skeleton $X^{[1]} \cup A$, and the obstruction to trivialization lies in cellular $H^2$ with $\mathbb Z^\pm$-coefficients, given by $c_1^R$ of the two Atiyah-real line bundles. This establishes injectivity of the map. Surjectivity is built by representing $c_1^R$ as a cellular co-chain and filling in the line bundle cell by cell with the appropriate degree dictated by the said co-chain. That the map is a group homomorphism follows from noting that this aforementioned co-chain is additive with respect to the tensor product, by multiplicativity of the degree.
    \end{proof}

    The main upshot of the work we have done so far can be summarized in the following observation:

    \begin{THM}\label{thm:rspinctors}
        There is a canonical isomorphism induced by the real Chern class
        \begin{equation}\label{eq:rline}
            {\rm rLine}(Y, \vec w) \cong H^2(Y' \setminus N', \vec \mu'; \underline{\bb Z}^v),
        \end{equation}
        where $\vec \mu'$ is the disjoint union of meridional circles in $\partial N'$ through $\vec w$. Consequently, the set ${\rm rSpin}^c(Y, \vec w)$ is naturally a torsor over this group.
    \end{THM}

    \begin{proof}
        Let $\vec \Delta$ be the disjoint union of meridional disks with boundary $\vec \mu$, used to perform the Dehn filling from $Y \setminus N$ to $Y$. In virtue of {\sc Prop.~\ref{prop:rc1}-iii} applied to the pair $(Y \setminus N, \vec \mu)$, the desired statement reduces to proving that the restriction map
        \[ {\rm rLine}(Y, \vec w) \underset{\text{by~h.e.}}\cong {\rm rLine}(Y, \vec \Delta) \to {\rm rLine}(Y \setminus N, \vec \mu) \]
        is an isomorphism. To this end, it suffices to show that any Atiyah-real line bundle on $Y \setminus N$ already trivialized on $\vec \mu \subset \partial N$ extends up to contractible space of choices to the Dehn filling used to produce $Y$, in such a way that the line bundle is also trivialized on the meridional disks $\vec \Delta$. For each toroidal component of $\partial N$, this 3-dimensional filling can be thought of as an $S^1$-parameter family of 2-dimensional fillings of meridional disks sweeping out the torus, where at the basepoint in $S^1$ the filling is forced to be the trivial one. Thus, it suffices to show that the space of fillings of an Atiyah-real line bundle on $\partial D^2$ to the whole of $D^2$ (with real involution given by $180^\circ$ rotation around the center) has $\pi_{* \ge 1} \equiv 0$. By filling in first a ray going out of the origin to some point on the boundary, we have a space homotopy-equivalent to the homotopy-fiber of the inclusion $\bb{RP}^\infty \to \bb{CP}^\infty$, which is widely known to be $\bb{RP}^1 \cong S^1$ due to the Bockstein fiber sequence
        \[ S^1 \overset{\cdot2}\to S^1 \to \bb{RP}^\infty \overset{\beta}\to \bb{CP}^\infty \overset{\cdot 2}\to \bb{CP}^\infty \to \cdots. \]
        Filling in one ray also fills in the antipodal one; to fill in the rest of the disk equivariantly, it remains to fill in the complex line bundle on the perimeter of each half-disk to the interior of the half-disk. This is a $\mathbb Z \cong \Omega^2 \bb{CP}^\infty$-choice. So we deduce that the total space of choices fibers over $S^1$, with fiber $\bb Z$. We would like to show that this is the universal cover of $S^1$, which amounts to showing that the monodromy of the fibration over $n \in \bb Z \cong \pi_1 S^1$ is non-trivial for every $n > 0$. Indeed, going $n$ times around $S^1 \cong {\rm Fib}(\bb{RP}^\infty \to \bb{CP}^\infty)$  amounts to adding $n$ to the degree of the filling of the boundary of the half-disk to the interior of the half-disk.
    \end{proof}

    Before moving on to the spaces of real ${\rm Spin}^c$-structures and their relationship to the pathspaces, let us mention another variation of the real Chern class, which can be used to reformulate the obstruction to $\mathbb Z/2$-gradings, and to offer a conjectural obstruction for $\mathbb Z$-gradings, cf. {\sc Prop.~\ref{prop:obsc1},~Rmk.\ref{rmk:conjzgr}}.

    \begin{DEF}\label{def:rc1}
        Define the \keywd{real Chern class} $c_1^R(\fk s)$ of a real ${\rm Spin}^c$-structure $\fk s$ as the real Chern class of the Atiyah-real line bundle obtained from the map of structure groups
        \[ {\rm Spin}^c_n = ({\rm Spin}_n \times S^1)/(-1, -1) \to S^1/(-1) \cong S^1. \]
        We have $c_1^R(\fk s \otimes \scr L) = c_1^R(\fk s) + 2 \cdot c_1^R(\scr L)$. In particular, the image of $c_1^R(\fk s)$ mod 2 does not depend on the particular choice of real ${\rm Spin}^c$-structure $\fk s$. We will show later in {\sc Prop.~\ref{prop:obsc1}} that this image mod 2 can be recovered as $v^2$.
    \end{DEF}

    \subsection{A space parameterizing  \texorpdfstring{${\rm Spin}^c$}{Spin-c}-structures on a 3-manifold} In Heegaard Floer, as well as Seiberg-Witten/monopole theory, it has become customary to use the terminology ``(real, relative) ${\rm Spin}^c$ structure'' to really refer to an \emph{isomorphism class} of such structures, the set of which we have just studied in the last subsection. Now, we introduce a \emph{space of (real, relative) ${\it Spin}^c_\textit3$-structures} whose $\pi_0$ recovers the set of isomorphism classes, which we show in the next subsection to be naturally homotopy-equivalent to the stable pathspace associated to the (real) Lagrangian Floer problem.

    The biggest challenge is that, while there are abstract approaches to this question that are quite standard (e.g.~ simplicial sets, or classifying spaces), we additionally want the following extra features:
    \begin{list}{$\cdot$}{}
        \item This (real, relative) space should receive a map from the space of (real, relative) vector fields {\sc(Def.~\ref{def:rvect})};
        \item There should be an honest involution on the non-real space, whose fixed points give rise to the real space.
        \item Both the real and non-real spaces of ${\rm Spin}^c_3$-structures should be torsors over topological abelian groups.
    \end{list}
    
    The definition of ${\rm Spin}^c$-structure that we will be space-upgrading is that presented in {\sc Def.~\ref{def:sphproj}}: A ${\rm Spin}^c_3$-structure on $V$ is a $U_2$-vector bundle $S$, with a metric identification of $\bb S_{1/2}(V)$ with $\bb P(S)$; and a real structure is given by an involution $I$ on $S$ such that $\bb S(-\tau)$ agrees with $\bb P(I)$ under the correspondence. We seek to encode the spinor bundle $S$ via a geometric gadget on \keywd{$E := \bb S_{1/2}(V)$} with the structure group $SO_3$ as $V$ (a.k.a.~ K\"ahler isometries of $S^2 \cong \bb P^1$), and this can be achieved via algebraic geometry in the following way:

    The projectivization $\mathbb P(S)$ of a rank-2 vector bundle $S$ naturally carries a complex line bundle \keywd{$\scr O_S(1)$} over its total space, defined as the dual of the tautological bundle. Standard complex/algebraic geometry shows that one can recover the rank-2 vector bundle $S$ via the isomorphism of complex vector spaces
    \[ S^* \cong H^0(\mathbb P(S), \scr O_S(1)). \]
    If $S$ is endowed with a Hermitian inner product, then its tautological line bundle inherits a canonical metric as well (by restriction), and hence so does its dual $\scr O_S(1)$. There is a way to put an inner product on the right side of the identity above so that the isomorphism respects this inner product, namely
    \[ \<f, g\> = \frac 2\pi \int_{\mathbb P(S)} \<f(z), g(z)\> {\rm~dvol}.  \]
    The upshot of this discussion is that the data of the spinor $S$ can be recovered canonically from any degree-1 holomorphic line bundle on the K\"ahler sphere $E \cong \bb P(S)$. Such holomorphic line bundles can be encoded by means of divisors, which we have seen before in {\sc\S\ref{sec:sym}} in an unrelated context:

    \begin{DEF}
        Recall that, given a Riemann surface $\Sigma$, we define $\scr D\Sigma$ to be the \emph{divisor group}, i.e.~ the free abelian group on $\Sigma$ appropriately topologized, cf. {\sc Def.~\ref{def:symdiv}}. This has a degree map $\scr D\Sigma \to \mathbb Z$, which splits $\scr D\Sigma$ into the preimages $\scr D^d\Sigma := {\rm deg}^{-1}(d)$. Since $\Sigma$ has a holomorphic structure, one has a complex line bundle \keywd{$\underline{\scr O}(d)$} on $\scr D^d\Sigma) \times \Sigma$ whose restriction $\scr O(D)$ to each fiber $\{D\} \times \Sigma$ is a degree-$d$ line bundle.

        This construction can be adapted to work parametrically: if $E \to B$ is a K\"ahler $\mathbb P^1$-bundle, we may define the \keywd{divisor bundle $\scr D(E)$} on $B$, given fiber-wise by the space of divisors on the particular fiber.
    \end{DEF}
    
    It is clear by what we have just said that any section $D$ of $\scr D^1(E)$ endows $E$ with a structure group lift to ${\rm GL}_2$ via the rank-$2$ associated bundle
    \[ S(D) := H^0(E; \scr O(D))^*. \]
    We would like to see that any ${\rm Spin}^c$-lift of $E$ is isomorphic to the induced ${\rm Spin}^c$-lift from a section of $\scr D^1(E)$, up to a contractible space of choices. Indeed, to obtain a divisor from a line bundle one has to choose a nonzero meromorphic section of it, and record its zeros and poles; the space of meromorphic functions is infinite-dimensional, and hence the complement of the origin is contractible, so fiber-wise the space of choices is contractible. The same holds true globally, if the base admits a CW structure, which is the case for manifolds.

    \begin{RMK}
        We already know that the homotopy type of $\scr D^1(\mathbb P^1)$ is $K(\mathbb Z, 2)$, by the Dold-Thom theorem. However, the observation that the space of nonzero meromorphic functions is contractible provides a more direct proof, since a divisor is uniquely determined by its meromorphic section up to a $\mathbb C^\times$-scalar, hence proving that the space $\scr D^1(\mathbb P^1)$ is a classifying space for $\mathbb C^\times$.

        In particular, the divisor spaces $\scr D^1(E_x)$ of the fibers of $E$ are $K(\mathbb Z, 2)$'s. Moreover, these fibers are torsors over $\scr D^0(E_x)$, which is also a topological abelian group homotopy equivalent to $K(\bb Z, 2)$.
    \end{RMK}

    \begin{DEF}
        We define the \keywd{space of ${\rm Spin}^c$}-structures on a metric $S^2$-bundle $E \to B$ with structure group ${\rm SO}_3$ to be the space of sections of $\scr D^1(E)$, which is naturally a torsor over the topological abelian group of sections of $\scr D^0(E)$. We use the notation \keywd{$\spinc(E)$} for the space of such structures (note the italicized style, to distinguish it from the set of equivalence classes, written in roman style ${\rm Spin}^c(E)$). If $V^3$ is a rank-3 vector bundle with structure group ${\rm SO}_3$, we also write $\spinc(V)$ for $\spinc(E)$, where $E = \bb S_{1/2}(V)$; we also use the notation $\spinc(Y^3, \vec w)$ for structures on $TY$ rel $\vec w$.
    \end{DEF}

    Next, we investigate conjugation and real ${\rm Spin}^{\rm c}$-structures. A $\mathbb P^1$-bundle $E \to B$ with structure group ${\rm SO}_3$ has a natural \emph{antipodal map} $\alpha : E \to E$: indeed, standard $\mathbb P^1 \cong S^2$ with the standard K\"ahler metric has an antipodal map given by the action of the element $-{\rm id} \in {\rm O}_3$; this antipodal map is coordinate-independent because $-{\rm id}$ commutes with all elements in ${\rm SO}_3$.

    Given a divisor section $D : B \to \scr D^1(E)$, we may form the \keywd{antipodal divisor} $\alpha^* D := \alpha \circ D$. There is a \emph{conjugate-linear} isomorphism of global sections
    \[ S(D) := H^0(E, \scr O(D)) \overset\sim{\underset{\rm cj.}\longrightarrow} H^0(E, \scr O(\alpha^* D)) =: S(\alpha^* D) \]
    given by $f \mapsto \overline{f \circ \alpha}$; it is necessary to introduce the conjugation in the definition of this map because otherwise $f \circ \alpha$ would be an anti-holomorphic meromorphic map. This conjugation is the reason why the induced map on global sections is a conjugate-linear isomorphism of complex vector spaces. The composite $S(D) \to S(\alpha^*D) \to S(D)$ of this process applied twice is the identity.

    Thus, there is an induced complex-conjugate involution between the associated rank-2 complex vector bundles $S(D)$ and $S(\alpha^*D)$. 

    It is clear that if $(E, \tau)$ is a real ${\rm SO}_3$-bundle with $\mathbb P^1$-fibers, then every section $D$ of $\scr D^1(E)$ which satisfies $\tau^* D = \alpha ^* D$ gives rise to a real ${\rm Spin}^{\rm c}$-structure. This motivates us to define the \keywd{space of real ${\rm Spin}^{\rm c}$-structures} on $(E, \tau)$ to be the subspace \keywd{$\rspinc(E, \tau)$} $\subset \spinc(E)$, consisting of sections $D$ satisfying $\tau^*D = \alpha^* D$. Likewise we define $\rspinc(V, \tau)$ and $\rspinc(Y, \vec w, \tau)$ as in the non-equivariant context, and sometimes drop $\tau$ from notation.

    \begin{PROP}
        Every real ${\rm Spin}^c$-structure on $(E, \tau)$ is canonically isomorphic to one coming from a real section $D$ of $\scr D^1(E)$, up to a contractible space of choices.
    \end{PROP}

    \begin{proof}
        As before, a real ${\rm Spin}^c$-structure $(S, I)$ induces a fiberwise holomorphic line bundle $\scr O_S(1)$ on $\mathbb P(S) \cong E$; to obtain a divisor section $D$ that satisfies the reality condition $\tau^*D = \alpha^* D$, we need to pick a meromorphic section $f \in \scr O^*(D)$ that satisfies $\tau^*f = \overline{f \circ \alpha}$. It therefore suffices to ensure that the space of such functions is contractible. To this end, let us pick a $\tau$-equivariant CW structure on the base (i.e.~ such that $\tau$ is cellular, and if $e \cap \tau(e) \neq \0$ for some cell $e$, then $\tau|_e = {\rm id}_e$), and inductively show that the space of sections is contractible skeleton by skeleton. This follows once we show that the space of choices is contractible over each $\tau$-orbit on the base. When the orbit is free, i.e.~ of the form $\{x, \tau x\}$, the reality condition on the section implies that its value on $\tau x$ is uniquely determined by that on $x$, and so the space is contractible as we discussed in the non-equivariant setting. Over a fixed point, the story is not very different: the space of (potentially zero) sections is still an infinite-dimensional vector space, and so the space of nonzero sections is contractible. To verify the claim that the space of sections is infinite-dimensional, let us first verify that it is non-trivial. First, let us note that $\alpha \circ \tau$ is an orientation reversing isometry other than the antipode $\alpha$ itself, because if $\tau = {\rm id}_E$, then the conjugate-linear isomorphism $I$ of the spinor bundle would be forced to square to $-{\rm id}_S$, cf. (\ref{eq:sqtolam}), a contradiction. By standard linear algebra, the fixed-point locus of $\alpha \circ \tau$ must then be a great circle in the $\mathbb P^1$-fibers. Thus, picking one point in this great circle, and then arbitrarily many pairs $p, \alpha(\tau(p))$ in the $\mathbb P^1$-fiber, we may form infinitely many real divisors, and consequently infinitely many linearly independent divisors.
    \end{proof}

    \begin{COR}
        The set of isomorphism classes of real ${\rm Spin}^{\rm c}$-structures on $(E, \tau)$ is in canonical bijection with $\pi_0 \rspinc(E, \tau)$.
    \end{COR}

    \begin{proof}
        The previous proposition gives a map from the first to the second set. Because homotopy-invariance of bundles still works equivariantly, injectivity is clear. Finally, the construction going from divisors to ${\rm Spin}^{\rm c}$-structures guarantees surjectivity.
    \end{proof}

    The following observation will prove to be quite useful in extending partially-defined (real) ${\rm Spin}^{\rm c}$-structures:
    \begin{PROP}\label{prop:exts}
        Let $B$ be a 3-ball with boundary $S = \partial B$, and let $E \to B$ be a trivialized $\mathbb P^1$-bundle. The associated divisor bundle $\scr D^1(E)$ with fiber $K(\mathbb Z, 2)$ also inherits a trivialization. Any ${\rm Spin}^{\rm c}$-structure on $E|_S$ has a well-defined degree given by the composite of the defining divisor section $D : S \to \scr D^1(E|_S)$ with the trivialization $\scr D^1(E|_S) \cong K(\mathbb Z, 2) \times S$, and then projection onto the $K(\mathbb Z, 2)$-factor. The ${\rm Spin}^{\rm c}$-structure extends over $B$ if and only if this degree is zero, and in this case the space of such fillings is contractible.
    \end{PROP}

    \begin{proof}
        Once $\scr D^1(E)$ is trivialized, an extension of the ${\rm Spin}^{\rm c}$-structure over $B$ is the same as an extension of the map $S \to K(\mathbb Z, 2)$ over $B$. Since $[S, K(\mathbb Z, 2)] \cong H^2(S; \mathbb Z) \cong \mathbb Z$, with the latter identification induced by the degree, this proves the extendability claim. The contractiblity claim follows because the higher-order extension problems from $(S^{k-1} \times B) \cup (D^k \times S)$ to $D^k \times B$ have their obstruction in $\pi_{k+2}K(\mathbb Z, 2) = 0$.
    \end{proof}

    We also need a 2-dimensional version, including the real analogue:

    \begin{PROP}\label{prop:rext}
        Consider the standard disk $D^2$ with the standard orientation, endowed with the trivial ${\rm SO}_3$-bundle, and further equipped with the ${\rm Spin}^{\rm c}$-structure given by the upward pointing vector field along the boundary, viewed as a divisor section. Let us equip $D^2$ with the standard involution with fixed-point set $D^1$, and equip the trivial ${\rm SO}_3$-bundle with the real structure given by sending $v \longmapsto -v$, so that the aforementioned ${\rm Spin}^{\rm c}$-structure on the boundary is real. Any extension of the ${\rm Spin}^c$-structure to the interior still has an obvious notion notion of degree as before. With this setup, we have the following two facts:
        \begin{enumerate}
            \item[\sc i.] the space of ${\rm Spin}^{\rm c}$-structures of any particular degree $d$, extending the given structure on the boundary, is contractible.
            \item[\sc ii.] likewise, the space of \emph{real} ${\rm Spin}^{\rm c}$-structures extending the given structure on the boundary, with prescribed degree $d$ of the underlying ordinary ${\rm Spin}^{\rm c}$-structure, is contractible.
        \end{enumerate}
        As a corollary, if $(B, E)$ from {\sc Prop.~\ref{prop:exts}} has degree zero, is equipped with the $180^\circ$-rotation involution along an axis on both $B$ and $E \cong \mathbb R^3$, and the ${\rm Spin}^{\rm c}$-structure on $\partial B$ is already real, then
        \begin{enumerate}
            \item[\sc iii.] the space of \emph{real} ${\rm Spin}^{\rm c}$-structures extending the given structure on the boundary is also contractible.
        \end{enumerate}
    \end{PROP}

    \begin{proof}
        The first fact is proved exactly the same as {\sc Prop.~\ref{prop:exts}}. For the second fact, we note that the first choice to be made is on the fixed-point locus $D^1 \subset D^2$, where the space of choices at each point is canonically identified with $K(\mathbb Z/2, 1)$. Thus, the space of real ${\rm Spin}^{\rm c}$-structures on $D^1$ extending the one on $\partial D^1 = D^1 \cap \partial D^2$ has the homotopy type of a discrete two-point set, according to whether the loop in $K(\mathbb Z/2, 1)$ has degree zero or one. The only remaining choice is extending the ${\rm Spin}^{\rm c}$-structure on one of the two sides of $D^1 \subset D^2$, since the other side will be uniquely determined by the reality condition. Indeed, this space is abstractly homotopy-equivalent to $\Omega^2 K(\mathbb Z, 2) \cong \mathbb Z$, hence the choice here is also homotopy-discrete. It remains to compute the degree of the resulting map given the two choices mentioned before.

        Let us consider first the case of an inessential loop inside $K(\mathbb Z/2, 1)$. We may assume without loss of generality that this loop is the constant one, so that the ${\rm Spin}^c$-structure on $D^1 \subset D^2$ is also given by the upward-pointing vector field. We claim that the degree of the resulting ${\rm Spin}^{\rm c}$-structure is always twice the relative degree on one half-disk. Indeed, if we think of the resulting ${\rm Spin}^{\rm c}$-structure as a map $S^2 \cong D^2 / \partial D^2 \to K(\mathbb Z, 2)$, the ${\rm Spin}^{\rm c}$-structure on the whole disk is given by the $\pi_2$-sum of the restriction of the two half-disks. Since both the process of flipping the domain of $D$, and the process of applying $D \mapsto \alpha^*\tau^*D$ on the codomain are orientation-reversing, the combined contribution is orientation-preserving, and the claim follows, so in particular this shows that we can hit all the even degrees.
        
        We claim that in the case that the real ${\rm Spin}^{\rm c}$-structure on $D^1$ is given by the essential loop in $K(\mathbb Z/2, 1)$, we hit all the odd degrees instead. To this end, note that real ${\rm Spin}^{\rm c}$-structures form an affine space over a topological abelian group; thus, let us fix a particular example of real ${\rm Spin}^{\rm c}$-structure $\fk s_0$ representing the essential loop in $K(\mathbb Z/2, 1)$, e.g.~ {\sc Fig.~\ref{fig:loctw}}, and observe that any other real ${\rm Spin}^{\rm c}$ structure $\fk s$ representing the nontrivial element of $\pi_1 K(\mathbb Z/2, 1)$ can be uniquely written as
        \( \fk s = \fk s_0 - \fk t + \fk s', \)
        where $\fk t$ is the trivial real ${\rm Spin}^{\rm c}$-structure given by the upward-pointing vector field everywhere, and $\fk s'$ is a real ${\rm Spin}^{\rm c}$-structure representing the inessential loop in $K(\mathbb Z/2, 1)$. The degree of $\fk s_0$ can be checked explicitly to be $-1$ in the case shown in {\sc Fig.~\ref{fig:loctw}}. Then, we have an equality of degrees $d(\fk s) = d(\fk s') - 1$, hence we hit all the odd degrees, as desired.

        The third claim follows from the first two by first filling in the real ${\rm Spin}^{\rm c}$-structure over $D^2 \subset D^3$ with the appropriate degree via {\sc i}, and then filling in the two remaining chambers equivariantly via {\sc ii} on one of them.
    \end{proof}

    \subsection{Comparison with the stable pathspace} Now that we have a good notion of the space of (real, relative) ${\rm Spin}^c_3$-structures, and enough tools to build such structures, we finally arrive at the main result of this section, comparing it to the (real) stable pathspace of {\sc\S\ref{sec:heeg}}:

    \begin{THM}\label{thm:comp}
        Let $\scr H$ be a Heegaard diagram with basepoints $\vec w$, and let $\spinc(Y, \vec w)$ denote the space of ${\rm Spin}^{\rm c}$-structures on $Y = Y(\scr H)$ whose value at each $w_i$ is equal to the vector pointing into $U_\alpha$. Then, there is a natural homotopy-equivalence
        \[ \hatP_s(\scr H) \overset\sim\to \spinc(Y, \vec w) \]
        depending only on a contractible space of choices. If $\tau$ is a real involution on $\scr H$, then this homotopy-equivalence can be arranged to be equivariant with respect to the obvious involutions on both sides of the map. Moreover, the induced map
        \[ \rhatP_s(\scr H, \tau) \overset\sim\to \rspinc(Y, \vec w, \tau) \]
        on the fixed points is also a homotopy equivalence. In particular, these maps induce isomorphisms in $\pi_0$.
    \end{THM}

    \begin{figure}\label{fig:loctw}
        \centering\includegraphics{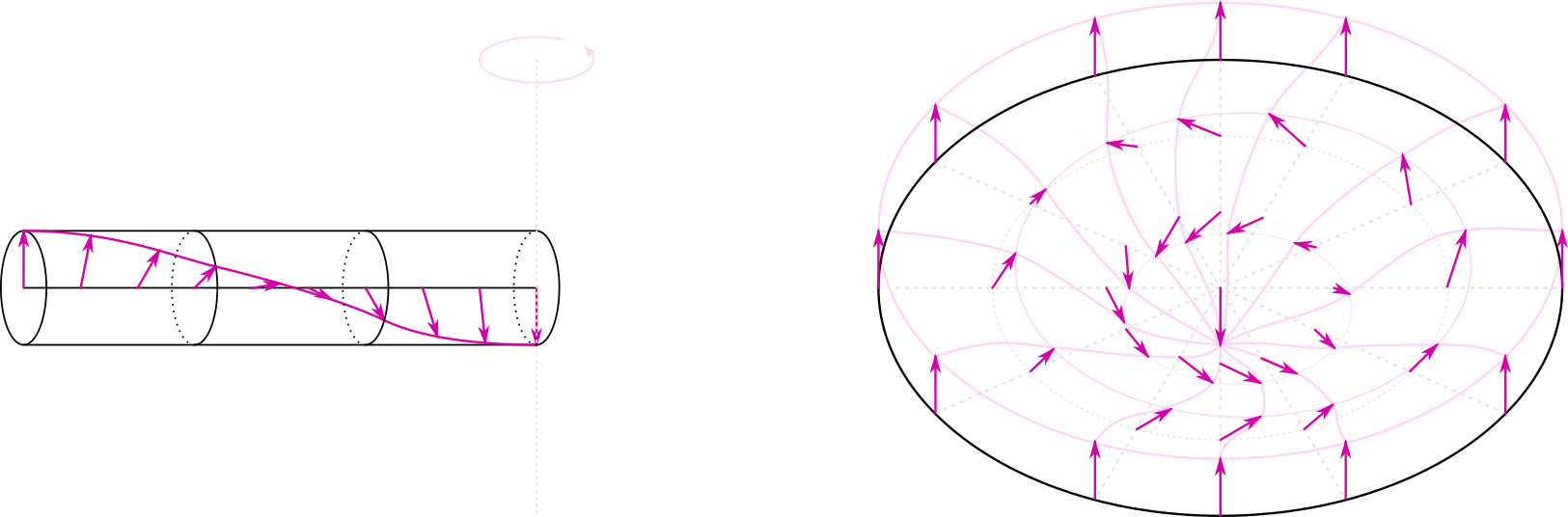}
        \caption{Local twisting of the upward-pointing vector field near a point. The picture from the right is obtained from the one on the left by rotating it around the right end.}
    \end{figure}

    \begin{CONSTR}\label{constr:map}
        To build the map $\hatP_s \to \spinc(Y, \vec w)$, we need to construct a ${\rm Spin}^{\rm c}$-structure rel $\vec w$ on $Y$ for any given path joining a point in $\mathbb T_\alpha$ to a point in $\mathbb T_\beta$ inside $\scr D^d \Sigma^\circ$. To begin with, let us note that there is a canonically existing vector field $V$ on $Y$ (up to a contractible space of choices), which is built by reconstructing a downward gradient-like vector field from the process of attaching handles to $\Sigma$ to obtain $Y$, cf. the Morse function discussed in {\sc Constr.~\ref{constr:morse}}. Assume that this vector field points inside of $U_\alpha$ and outside of $U_\beta$, and that it satisfies $V = -\tau^* V$; its only zeros are at the ``critical points'', i.e.~ one at the center $a_i$ and $b_i$ of each disk $D^\alpha_i$ and $D^\beta_i$, and finally one for the center $p_i$ or $q_i$ of each 3-ball attached (each corresponding to a basepoint, and the handlebody $U_\alpha$ or $U_\beta$ that the 3-cell is attached inside). We will be modifying this vector field (and in the process potentially replacing vectors with the more general degree-1 divisors on the sphere bundles of $TY$) according to the data of the specified element of $\hatP_s$, and then take the associated ${\rm Spin}^{\rm c}$-structure to it.

        We will be using {\sc Props.~\ref{prop:exts}, \ref{prop:rext}} in order to modify this real vector field and remove the points where it is ill-defined, i.e.~ its zeros. We begin with the centers $p_i$ and $q_i$ of the 3-balls, where the modification actually does not depend on the element in $\hatP_s$ at all (which is to be expected, since we want the real ${\rm Spin}^{\rm c}$-structure to be the same near the basepoints.) Let us pick the radial path connecting $w_i$ with the corresponding critical point $p_i$ in $U_\alpha$, and likewise the image under $\tau$ connecting $w_i$ with the corresponding critical point $q_i$ in $U_\beta$, and let $B$ be a $\tau$-equivariant regular neighborhood of the concatenation of these two paths. We modify the vector field $V$ inside this ball: first, on $B \cap \Sigma$, we use any choice in the degree $-1$ circle-component from {\sc Prop.~\ref{prop:rext}-ii} (e.g.~ the construction shown in {\sc Fig.~\ref{fig:loctw}}). Then, in order to extend the choice to the two chambers of $B \setminus \Sigma$, it suffices to check the degree condition, cf. {\sc Prop.~\ref{prop:exts}}. This can also be done by explicit calculation near a concrete model of the gradient-like vector field near an index-$3$ critical point (i.e.~ the vector field pointing everywhere outside). Finally, we note that the extension on one of the chambers uniquely determines the extension in the other chamber, if the map is to be $C_2$-equivariant.
        
        Now, we move on to the centers of the 2-cells $D^\alpha_i$ and $D^\beta_i$. Here, we will need the additional input of the path $\gamma : [-1, 1] \to \scr D^d\Sigma^\circ$, connecting $\mathbb T_\alpha$ to $\mathbb T_\beta$. Let us consider a $\tau$-equivariant tubular neighborhood of $\Sigma$ inside $Y$, parameterized as $[-1, 1] \times \Sigma$, and let us apply the deformation corresponding to degree $d = 1$ from {\sc Prop.~\ref{prop:rext}} on horizontal disks centered around the points $(t, \gamma(t)) \in [-1, 1] \times \Sigma$, in a $\tau$-equivariant manner. This means that if the point is on the fixed-point locus, we must use the real version of the proposition, i.e.~ part {\sc ii}. It is clear what this means if the points in $\gamma(t)$ are all disjoint, though some explanation is needed to handle the case that points collide. This is the crucial step where we use the fact that the space of ${\rm Spin}^{\rm c}$-structures is a torsor over a topological abelian group. The process of twisting can be equivalently described as adding a certain degree-0 divisor supported in the disk. This process can be carried out $g$ times in a row for any \emph{unordered} sequence of potentially repeating $g$ many points in $\Sigma$, so the construction makes sense even when points in the symmetric product collide. Because the points are constrained to lie in $\Sigma^\circ$, we are never affecting the vector field near the basepoints. Now, we may once again perform the same trick as for the centers of the 3-balls, i.e.~ we may take paths from the centers $a_i$ of the $D^\alpha_i$ to the points in $(-1, \gamma(-1))$, and change the divisor in a regular neighborhood of these paths, using {\sc Prop.~\ref{prop:exts}} and likewise for the centers $b_i$ of the $D^\beta_i$.
    \end{CONSTR}

    \begin{RMK}
        In the case that the path in $\hatP_s$ is constant, i.e.~ we have an actual geometric intersection points of $\mathbb T_\alpha$ and $\mathbb T_\beta$ inside $\scr S^g\Sigma^\circ$, the construction of Guth-Manolescu also modifies the gradient-like vector field in a neighborhood of gradient flow lines. Since the choices of extension are contractible, it follows that our construction recovers that of Guth-Manolescu \cite{GM25} on the level of generators of $\widehat{\it CFR}$ and real ${\rm Spin}^{\rm c}$-structures up to isomorphism. In particular, the splitting of the Floer complex according to connected components of the pathspace agrees with that built by Guth-Manolescu.
    \end{RMK}

    Now, we must prove the homotopy equivalence claim:

    \begin{proof}[Proof of {\sc Thm.~\ref{thm:comp}}]
        To this end, let us recall that both the non-equivariant and the real pathspaces are homotopy-fiber products of spaces that are abstractly homotopy-equivalent to tori, or disjoint unions of tori. The same is true for the spaces of ordinary and real ${\rm Spin}^{\rm c}$-structures, i.e.~
        \[ \begin{tikzcd}
            \spinc(Y, \vec w) \rar\dar & \spinc(U_\alpha, \vec w) \dar \\
            \spinc(U_\beta, \vec w) \rar & \spinc(TY|_\Sigma, \vec w)
        \end{tikzcd} \quad \begin{tikzcd}
            \rspinc(Y, \vec w) \rar\dar & \spinc(U_\alpha, \vec w) \dar \\
            \rspinc(TY|_\Sigma, \vec w) \rar & \spinc(TY|_\Sigma, \vec w)
        \end{tikzcd}\]
        are homotopy fiber-products. It is not difficult to see that the map $\hatP_s \to \spinc(Y, \vec w)$ built in {\sc Constr.~\ref{constr:map}} respects this fiber-product structure in the pathspace model of the homotopy-fiber product: indeed, we have split $Y$ into three pieces: the tubular neighborhood $[-1, 1] \times \Sigma$, and the two handlebodies that make up the complement; the construction of the ${\rm Spin}^c$-structure on the $U_\alpha$ and $U_\beta$ handlebodies depends only on the endpoints of the path in $\hatP_s$, and likewise the construction at $\{t\} \times \Sigma$ depends only on the value of the path at time $t$. Thus, the bigger map $\hatP_s \to \spinc(Y, \vec w)$ and its restriction to $C_2$-fixed points are built out of smaller maps
        \begin{enumerate}
            \item[\sc i.] $\mathbb T_\alpha \to \spinc(U_\alpha, \vec w)$, and similarly for $\beta$;
            \item[\sc ii.] $\scr D\Sigma^\circ \to \spinc(TY|_\Sigma, \vec w)$;
            \item[\sc iii.] $\rdiv\Sigma^\circ \to \rspinc(TY|_\Sigma, \vec w)$;
        \end{enumerate}
        and by the homotopy-invariance of homotopy-pullbacks it suffices to show that these maps are each individually homotopy-equivalences.
    
        First, we observe that each of the 6 spaces showing up in {\sc i-iii} has the homotopy type of a disjoint union $G \times K(\mathbb Z^r, 1)$ of tori for some abelian group $G$, with the same $G$ and $r$ both on the domain and codomain. Let us begin with the domains: the first one is clearly of this form; the second one follows by the Dold-Thom theorem, as we discussed in {\sc\S\ref{sec:sym}}, since it is a topological abelian group with $\pi_*$ given by $H_*(\Sigma^\circ; \mathbb Z)$, i.e.~ $r = 2g + k - 1, G = \mathbb Z$. For the third, let us recall from the discussion at the end of {\sc Def.~\ref{def:rdeg}} that
        \begin{equation}\label{eq:pi0fix}
            G = \pi_0 \rdiv\Sigma^\circ \cong \{ (d, d_1, \dots, d_k) \in \mathbb Z \oplus (\bb Z/2)^{k} : d \equiv d_1 + \cdots + d_k  ~\text{mod}~2\},
        \end{equation}
        and furthermore that the higher homotopy groups satisfy $\pi_{* \ge 1} \rdiv\Sigma^\circ \cong \pi_* \scr D\Sigma' \cong H_*(\Sigma'; \bb Z)$, so in particular
        \begin{equation}\label{eq:pi1fix}
            r = {\rm rk~} H_1(\Sigma'; \bb Z) = g
        \end{equation}
        due to the homotopy equivalence proved in {\sc Prop.~\ref{prop:treq}}, and the Dold-Thom theorem applied to $\Sigma'$.
        
        For the codomains of the maps, let us first note that they are all affine spaces over topological abelian groups ${\rm Maps}_*(U_\alpha / \vec w, K(\mathbb Z, 2))$, ${\rm Maps}_*(\Sigma / \vec w, K(\mathbb Z, 2))$ and ${\rm Maps}^{C_2}_*(\Sigma / \vec w, K(\mathbb Z, 2))$, respectively, since in each case the $SO(3)$-bundle can be trivialized, and hence so can the divisor bundle $\scr D^0$ with $K(\bb Z, 2)$-fibers. It is not hard to see that all three mapping spaces have vanishing $\pi_{* \ge 2}$, since the domains are connected and the component at the basepoint in $\Omega^2K(\mathbb Z, 2)$ is contractible. For the first two, we have $\pi_0$ equal to $H^2(U_\alpha, \vec w) \cong 0$, and $H^2(\Sigma, \vec w) \overset{\rm deg}\cong \mathbb Z$ respectively, which matches the respective invariants of the domains. For the third space, we must examine the homotopy type by breaking it up over cells. The fixed-points of $\Sigma / \vec w$ is a bouquet of circles $\bigvee_{\vec w_i \in \vec w} S^1$, and the space of choices over each point is a $K(\mathbb Z/2, 1)$, hence the space of global choices over the fixed locus has the homotopy type of a discrete group $\prod_{w_i \in \vec w} \mathbb Z/2$. If we give $\Sigma'$ a cell structure where the only 0-cells are the $w_i$, and we lift it to a $C_2$-equivariant cell structure on $\Sigma$, each free 1-cell gives an $S^1 \cong \Omega K(\mathbb Z, 2)$-parameter family of choices. These choices are independent, so the joint choice of extension to the 1-skeleton represents a torus of rank $g = {\rm rk~} H_1(\Sigma'; \bb Z)$. Finally, the 2-cell is always fillable up to contractible choice once the degree is specified, cf. {\sc Prop.~\ref{prop:rext}-i}. The degree of the whole ${\rm Spin}^{\rm c}$-structure gives us a homomorphism $\pi_0 {\rm Maps}^{C_2}_*(\Sigma/\vec w, K(\mathbb Z, 2)) \to \mathbb Z$. When all the $\mathbb Z/2$-degrees on the bouquet are zero, the degree of the whole ${\rm Spin}^{\rm c}$-structure is always twice the degree of the 2-cell; so we hit all the even numbers, in a manner very similar to the proof {\sc {Prop.~\ref{prop:rext}}}. We can then apply local twistings as in {\sc Prop.~\ref{prop:rext}-ii} on every connected component of $\Sigma^\tau$ to change the global degree by 1, and the $\mathbb Z/2$-degree of that specific component by 1. We conclude that the global degree always has the same parity as the sum of all the $\mathbb Z/2$-degrees at all the circles in the bouquet, and moreover that the class in $\pi_0$ is uniquely determined by this data, i.e.~ $\pi_0 {\rm Maps}^{C_2}_*(\Sigma/\vec w, K(\mathbb Z, 2))$ is also isomorphic to (\ref{eq:pi0fix}). As for $\pi_1$, we have already established that each component is a torus of rank the number of free 1-cells, i.e.~ the rank $g$ of $H_1(\Sigma')$, so we once again recover the expected (\ref{eq:pi1fix}).

        It remains to check that the maps of spaces {\sc i-iii} induce the (a priori non-canonical) isomorphisms on $\pi_0$ and $\pi_1$ that we have explained above. Beginning with {\sc ii}, we have $\pi_0 = \mathbb Z$ for both, given on the domain by the degree of the divisor, and on the codomain by the degree of the ${\rm Spin}^{\rm c}$-structure over the whole surface $\Sigma$. It is clear that twisting the ${\rm Spin}^{\rm c}$-structure as in {\sc Constr.~\ref{constr:map}} changes its global degree by one, hence proving the claim for $\pi_0$. As for $\pi_1$, we need to show that
        \[ H_1(\Sigma^\circ) \cong \pi_1 \scr D\Sigma^\circ \to \pi_1 \spinc(TY|_\Sigma,\vec w) \cong \pi_1{\rm Maps}_*(\Sigma/\vec w, K(\mathbb Z,2)) \cong H^1(\Sigma, \vec w) \]
        is an isomorphism. We claim that this is the Poincar\'e duality isomorphism, i.e.~ that if $\gamma$ is a closed loop in $\Sigma^\circ$, and $\delta$ is an embedded closed loop or properly embedded open curve in $\Sigma^\circ$, then the image of $[\gamma]$ paired with $[\delta]$ is the intersection number $\gamma \cdot \delta \in \bb Z$. To this end, recall that the construction of the map {\sc ii} twists the ${\rm Spin}^{\rm c}$-structure by degree 1 in a small disk near the point. Then, our claim boils down to a local computation, where we move the small disk in which the twisting is performed from one side to the other side of $\delta$, and we wish to see that the degree of the twisting is one. In other words, we may consider the ${\rm Spin}^{\rm c}$-structure on $\delta \times \gamma$ given by the restriction of the ${\rm Spin}^{\rm c}$-structure on $TY|_\delta$ using the twisting at the respective point in $\gamma$, and we want to see that this degree near the intersection points is 1. It is not hard to see that the local picture near the intersection of $\gamma$ and $\delta$ is exactly the same as the twisting by a degree-1 from {\sc Prop.~\ref{prop:rext}}; this finishes the proof that {\sc ii} is an equivalence.

        We claim that the equivalence claim about {\sc i} follows from that about {\sc ii}. Both sides of {\sc i} have trivial $\pi_0$, so it remains to show that $\pi_1$ induces an isomorphism. Indeed, there is a commutative square
        \[ \begin{tikzcd}
            \pi_1 \mathbb T_\alpha \rar\dar & \pi_1 \spinc(U_\alpha, \vec w) \dar \\
            \pi_1 \scr D\Sigma^\circ \rar & \pi_1 \spinc(TY|_\Sigma, \vec w).
        \end{tikzcd} \]
        and the vertical maps are both inclusions (a simple calculation on $H_1$ and $H^1$), while the bottom one has just been proven to be an isomorphism. Thus, the top map (i.e.~ map {\sc i}) is an isomorphism if and only if the images of the top groups into the bottom groups are equal as subgroups thereof, under the identification induced by the isomorphism {\sc ii}. This follows from the fact that $U_\alpha$ is obtained up to homotopy from $\Sigma^\circ$ by attaching 2-cells along all the $\alpha$-circles, so in particular under Poincar\'e duality the cohomology classes in $H^1(\Sigma, \vec w)$ that pair to zero with the $\alpha$ circles are precisely those that extend to $H^1(U_\alpha, \vec w)$.

        Finally, we must establish that {\sc iii} induces an isomorphism on $\pi_0$ and $\pi_1$. On $\pi_0$, the two sides are once again determined uniquely by all the degree maps (either into $\mathbb Z/2$ or $\mathbb Z$), and moreover it is clear that the twisting of the ${\rm Spin}^{\rm c}$-structures by divisors of the given degree change the corresponding degrees of the ${\rm Spin}^{\rm c}$-structures by the same amount. On $\pi_1$, there is once again a commutative square relating back to {\sc ii}:
        \[ \begin{tikzcd}
            \pi_1 \rdiv\Sigma^\circ \rar \dar & \pi_1 \rspinc(TY|_\Sigma, \vec w) \dar \\
            \pi_1 \scr D \Sigma^\circ \rar & \pi_1 \spinc(TY|_\Sigma, \vec w),
        \end{tikzcd} \]
        where the vertical maps are injections. The left vertical map is injective, thanks to our last claim in {\sc Prop.~\ref{prop:divhom}} that it is induced by the transfer map, which can be checked to be injective e.g.~ because its post-composition with the projection $\pi_*$ is equal to twice the identity, and both sides are torsion-free. For the right vertical map, recall that the toroidal components of $\rspinc(TY|_\Sigma, \vec w)$ come from a copy of $S^1 \cong \Omega K(\mathbb Z,2)$ each one of the free 1-cells on $\Sigma$ pulled back from the 1-cells of $\Sigma'$ (other than the boundary of $\Sigma').$ Thus, the induced map is just the pull-back on the level of cohomology, once both $\pi_1$'s are identified with cohomology groups $H^1(\Sigma', \vec w)$ and $H^1(\Sigma, \vec w)$. Then the images of the top groups into the bottom groups of the commutative square agree under Poincar\'e duality, hence concluding the proof.
    \end{proof}

    Next, we wish to show that the maps in {\sc Thm.~\ref{thm:comp}} are preserved under handleslides and stabilizations, at least in the homotopy category, so that the isomorphisms induced on $\pi_0$ and $\pi_1$ respect these operations. 

    \begin{PROP}
        Let $\alpha_i, \alpha_j$ be two alpha-circles in a Heegaard diagram, and $\alpha_i'$ be another embedded circle in the complement of the alpha curves such that $\alpha_i, \alpha_i', \alpha_j$ cobound a pair of pants in $\Sigma$ not containing any of the basepoints $\vec w$. The new set of curves obtained by replacing $\alpha_i$ with $\alpha_i'$ gives rise to a new torus $\mathbb T_\alpha'$. In this case, there is a canonical contractible space of homotopies witnessing the commutativity of the diagram
        \[ \begin{tikzcd}
            \mathbb T_\alpha \dar["\rm Dehn" swap, "\sim"]\rar & \spinc(U_\alpha, \vec w) \dar[equals] \\
            \mathbb T_\alpha' \rar & \spinc(U_\alpha', \vec w)
        \end{tikzcd} \]
        where the horizontal maps are the homotopy equivalences of type {\sc i} in the proof of {\sc Thm.~\ref{thm:comp}}, the left vertical map is the Dehn twist uniquely determined on generators by sending the $\alpha_\lambda$ to $\alpha_\lambda$ for $\lambda \neq i$, and $\alpha_i$ to $\alpha_{i}' \star \alpha_j$ by preserving points far away from the point of contact between $\alpha_i$ and $\alpha_j$ in the connect-sum operation, and the right vertical map is the identification of handlebodies determined by the $\alpha$-curves.
    \end{PROP}

    \begin{proof}
        We begin by noting that the map $\spinc(U_\alpha, \vec w) \to \spinc(TY|_\Sigma, \vec w)$ has homotopy-discrete fibers; indeed, the fiber has the homotopy type of ${\rm Maps}_*(U_\alpha / \partial U_\alpha, K(\mathbb Z, 2))$ by the torsor action of degree-zero divisors on degree-one divisors. The latter space has $\pi_i$ equal to $H^{2 - i}(U_\alpha, \partial U_\alpha) \overset{\rm PD}\cong H_{i+1}(U_\alpha)$, which proves the claim. Thus, it suffices to produce a homotopy witnessing the commutativity of
        \[ \begin{tikzcd}
            \mathbb T_\alpha \dar["\rm Dehn" swap, "\sim"]\rar & \scr D \Sigma^\circ \dar[equals] \\
            \mathbb T_\alpha' \rar & \scr D \Sigma^\circ
        \end{tikzcd} \]
        and check that the canonical extension of the associated ${\rm Spin}^{\rm c}$-structures on $TY|_\Sigma$ induced from the map $\scr D\Sigma^\circ \to \spinc(TY|_\Sigma, \vec w)$ lies in the same component of the homotopy-discrete set of choices of extension all the way to $\spinc(U_\alpha, \vec w)$, for both $\mathbb T_\alpha$ and $\mathbb T_\alpha'$. As far as constructing the square above, we may ignore the $\alpha$ circles other than the ones involved in the handle-slide, so it suffices to establish a contractible space of homotopies
        \[ \begin{tikzcd}
            \alpha_i \times \alpha_j \rar\dar["\rm Dehn" swap, "\sim"] & \scr S^2 P \dar[equals] \\
            \alpha_i' \times \alpha_j \rar & \scr S^2 P
        \end{tikzcd} \]
        where $P$ is the pair of pants bounding $\alpha_i, \alpha_i', \alpha_j$.
        Given that $\scr S^2 P$ is complex biholomorphic to $\mathbb C^2$ minus a union of two transverse $\mathbb C \times 0$ and $\mathbb C \times 0$, it is homotopy equivalent to a torus, and hence has $\pi_{* \ge 2}$ equal to zero. In particular, it suffices to construct the homotopy on the relative 1-skeleton, and check the obstructions for the 2-cells. We work with the standard cell structure on $S^1 \times S^1 \times I$, so the only relative 1-cell is the one connecting the origins in the two tori. We may pick the basepoint $u$ on $\alpha_i$ and $\alpha_i'$ to be the one diametrically opposite to the point of contact between the circles during handle-slide, where the homotopy can be chosen to be the obvious one, cf. {\sc Fig.~\ref{fig:aslide}}. The only relative 2-cells come from $S^1 \times * \times I$ and $* \times S^1 \times I$, which satisfy the desired fillability condition because the Dehn twist was rigged to make the diagram commutative on the level of $H_1$.

        \begin{figure}
            \centering\includegraphics[scale=1.5]{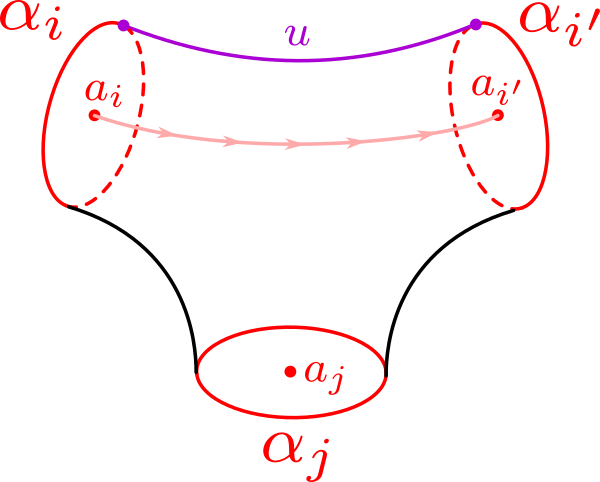}
            \caption{Depiction of a handle-slide, particularly the process of sliding the critical point $a_i$ along the pink path with arrows all the way to $a_i'$; the path labelled in purple on the very top represents the position of the basepoint $u$, which is ``the one far away from the point of contact of the handle-slide''; this path could be shrunk to a single common point $u$, but we keep it like this for aesthetic reasons, and to better recognize the pair of pants.}\label{fig:aslide}
        \end{figure}

        Finally, we must check that the extension of the associated ${\rm Spin}^{\rm c}$-structures on the whole of $U_\alpha = U_{\alpha'}$ lie in the same contractible component of the fiber of the map $\spinc(U_\alpha, \vec w) \to \spinc(TY|_\Sigma, \vec w)$. By connectedness of $\bb T_\alpha$, this can be checked over a single point in it, e.g.~ a tuple in $\mathbb T_\alpha$ where the $\alpha_i$-component is far from the point of contact of the handle-slide. Indeed, in this case we have a 1-parameter family of Morse functions given by moving the index-2 critical point, as explained in {\sc Constr.~\ref{constr:morse}}, and we may extend our modification of the pseudo-gradient along regular neighborhoods of the flow lines. This results in a path of ${\rm Spin}^c$-structures, witnessing the desired claim.
    \end{proof}

    We have an analogous proposition for stabilizations near the basepoints, which concludes the invariance of the stable pathspace with respect to the particular real Heegaard diagram:

    \begin{PROP}\label{prop:pathstab}
        Let $\scr H'$ be the stabilization of a (real) Heegaard diagram $\scr H$ by connect-summing with one of the three standard diagrams \cite{GM25}, cf. {\sc Fig.~\ref{fig:rstab}} for $S^3$ in the domain containing one of the basepoints $w_i$. The inclusion $\hatP_s(\scr H) \hookrightarrow \hatP_s(\scr H')$ given by adjoining the constant path(s) on the extra generator(s) of $\widehat{\it CFR}(S^3)$ induces a ($C_2$-equivariant) homotopy-equivalence, and moreover the diagram
        \[ \begin{tikzcd}
            \hatP_s(\scr H) \rar["\sim"] \dar["\sim"] & \spinc(Y, \vec w) \dar[equals] \\
            \hatP_s(\scr H') \rar["\sim"] & \spinc(Y, \vec w) 
        \end{tikzcd} \] 
        commutes.
    \end{PROP}

    \begin{proof}
        We note that the ${\rm Spin}^{\rm c}$-structures constructed by the horizontal maps are the same away from the 3-ball coming from the connect-summand $S^3$, and therefore they are connected by a contractible space of homotopies (both in the classical and real settings), due to {\sc Props.~\ref{prop:exts},~\ref{prop:rext}-iii}. This proves the desired commutativity, and the fact that the stabilization map on stable pathspaces is a homotopy equivalence.
        
    \end{proof}

\section{Index theory for \texorpdfstring{$\mathbb Z/2$}{Z2}-gradings and \texorpdfstring{$\mathbb Z$}{Z}-coefficients}\label{sec:gror}

	In this section, we explore the notions of $\mathbb Z/2$-gradings and of coherent orientations in the greater generality of Lagrangian Floer homology, and set up the relevant obstruction theory in terms of Stiefel-Whitney classes on connected components of the Lagrangian pathspace. This subject is very thoroughly studied, and several very standard approaches to it include \cite{FH93, Si98, Fuk00, FOOO10, Sei08, Rez21}; what these all have in common is that they endow the Lagrangians with extra structure lifts (``branes''), and potentially also the symplectic manifold itself (with ``background data''), with suitable compatibility with the branes. This is because one is typically interested in building a Fukaya category, which requires that objects be Lagrangians with some extra data depending only on themselves, or on the ambient symplectic manifold.

	Recent developments show that this insistence is at times too limiting, e.g.~ Abouzaid-Manolescu's \cite{AM25} using \emph{coupled Spin structures}, which require a kind of ``brane on the pair of the two Lagrangians together.'' Inspired by this, we give an even more general notion of coupling, defined as the vanishing of a certain characteristic class living on the \emph{pathspace} of the two Lagrangians (i.e.~ we allow the data to depend on a path connecting the two points in the two different Lagrangians.) The obstruction theory we develop is in some sense if-and-only-if, given {\sc Defs.~\ref{def:z2gr},~\ref{def:cohsys}} to follow. For the reader who is not willing to give up the hope of building Fukaya categories, we conjecture that this approach could be used to construct not an $\bb A_\infty$-category, but a category living over a more complicated dg-operad; we do not pursue this further here, though. Our index theory calculation techniques resemble \cite{Sei08} a lot, and we do not claim too much originality in this regard; the novelty comes in doing the computation over the pathspace, and furthermore in the algebraic machinery of groupoid cohomology.

	Let $M^{2n}$ be a symplectic manifold, and $L_0^n, L_1^n$ be two transversely intersecting embedded Lagrangians. We will assume that $n$ is large enough (in particular, at least $n \ge 3$ in certain places), which we may do in the case of (real) Heegaard Floer theory by stabilizing enough times; nevertheless, the final results will also apply to low genus a posteriori, see {\sc Rmk.~\ref{rmk:larged}}. Let $\scr P$ be a \emph{connected component} of the pathspace $\scr P_M(L_0, L_1)$ connecting $L_0$ and $L_1$ inside $M$, and denote by \keywd{${\it CL}_{\scr P}$ the Lagrangian Floer complex} over $\mathbb Z/2$ built by only considering intersection points whose associated constant path in $\scr P_M(L_0, L_1)$ lies in $\scr P$; we seek to endow ${\it CL}_{\scr P}$ with a $\mathbb Z/2$-grading, as well as a $\mathbb Z$-lift. We employ the notation \keywd{$L_0 \cap_\scr P L_1$} for the aforementioned subset of intersection points.

	There is a naturally occurring virtual bundle on $\scr P$, whose Stiefel-Whitney classes $w_1$ and $w_2$ will give the obstructions to the desired upgrades, which we now define.

	\begin{DEF}\label{def:lbd}
		Define the \keywd{Lagrangian boundary difference} element associated to two Lagrangians $L_0, L_1 \subset M$ in a symplectic manifold $M$ to be the map
		\[ \scr P \subset \scr P_M(L_0, L_1) \overset{TL_1 - TL_0} \longrightarrow BO \]
		obtained formally by applying the difference map $BO \times BO \overset\ominus\to BO$ to the classifying maps for both $TL_0$ and $TL_1$. All choices involved are parameterized by contractible spaces, including the aforementioned difference map, since $BO$ is a group-like $\mathbb E_\infty$-space.
	\end{DEF}

	\begin{RMK}
		As an element in $[\scr P, BO] \cong \widetilde{\it KO}(\scr P)$, this classifies the difference $[TL_1] - [TL_0]$. In the case of real Heegaard Floer homology, we have $[T\mathbb T_\alpha] = 0$, so on the level of K-theory the Lagrangian boundary difference is just $[T(\rsym^d \Sigma^\circ)] \in {\it KO}(\rhatP(\scr H, \tau))$.
		However, in order to identify the choices of orientations and gradings more canonically, one would have to work with the actual map $\scr P \to BO$, as opposed to its mere homotopy class. Since we only rigorously prove invariance up to an ambiguity of the choices, this does not concern us directly, though it will be useful in the somewhat speculative last section {\sc\S\ref{sec:abs}}.
	\end{RMK}

	The main results of this section can be summarized as follows (see also {\sc Rmk.\ref{rmk:twgr}} about twisted local coefficients to squeeze a little extra information in certain cases):

	\begin{THM}\label{thm:obs}
		The obstructions to abstract $\mathbb Z/2$-gradings {\sc(Def.~\ref{def:z2gr})} and $\bb Z$-coefficients {\sc(Def.~\ref{def:cohsys})} produced in {\sc Props.~\ref{prop:grobs},~\ref{prop:totobs}} are given by $w_1$ and $w_2$, respectively, of the Lagrangian boundary difference $TL_0 - TL_1$ on $\scr P$.
	\end{THM}

	\begin{THM}\label{thm:invar}
		The vanishing of the obstructions for $\mathbb Z/2$-gradings and $\mathbb Z$-coefficients for ${\it CF}_\scr P(L_0, L_1)$ is independent up to Hamiltonian isotopy or change of almost-complex structure. Moreover, if any of the two obstructions vanishes, there is a canonical bijective correspondence between the sets of choices of either $\bb Z/2$-gradings or $\bb Z$-coefficients, for the two chain theories before and after Hamiltonian isotopy or change of almost-complex structure.
	\end{THM}

	\begin{COR}\label{cor:isotinv}
		Change of almost-complex structure and isotopy of the $\alpha$-circles induces a bijection between the sets of choices of absolute $\mathbb Z/2$-gradings and $\mathbb Z$-coefficients for $\widehat{\it HFR}$ for the two real Heegaard diagrams.
	\end{COR}

	\begin{RMK}
		The obstructions in {\sc Thm.~\ref{thm:obs}} in particular vanish when $(L_0, L_1)$ are endowed with a coupled Spin structure in the sense of \cite{AM25}, and likewise in the case that $(L_0, L_1, M)$ admits a relative Spin structure in the sense of \cite{Fuk00}. Thus, our hypothesis is strictly weaker than these, but leads to the same conclusion.
	\end{RMK}

	\subsection{Abstract systems of \texorpdfstring{$\bb Z/2$}{Z/2}-gradings and coherent orientations} We begin with a few preliminaries:
	\begin{DEF}
		Let $p, q \in L_0 \cap_{\scr P} L_1$ be two non-degenerate critical points, together with a path $\gamma$ in $\scr P$ from $p$ to $q$, or equivalently a (topological) Whitney disk. Even though such a path $\gamma$ need not be a flow line for the gradient flow of the action functional on $\scr P$, there is still a well-defined linearization of the Floer equation, resulting in a Fredholm operator $D_\gamma$ between Hilbert spaces (the exact choices of Sobolev spaces, metric, almost complex structure will end up not affecting the definition). There is a well-defined locally constant \keywd{virtual dimension function $\delta_{p, q}$} from the space of Whitney disks $\Omega_{p, q} \scr P$ to $\mathbb Z$,
		\[ \delta_{p, q}(\gamma) := {\rm dim~Ker}(D_\gamma) - {\rm dim~Cok}(D_\gamma) \]
		as well as a well-defined \keywd{determinant line bundle} \keywd{$\Lambda_{p, q}$} over $\Omega_{p, q} \scr P$ defined point-wise by
		\[ \Lambda_{p, q}|_\gamma := \Lambda^{\rm top} {\rm Ker}(D_\gamma) \otimes \Lambda^{\rm top} {\rm Cok}(D_\gamma)^\vee.\]
		Standard linear gluing shows that $\delta_{p, q}$ and $\Lambda_{p, q}$ satisfy the following additivity properties:
		\begin{equation}\label{eq:deladd}
			\delta_{p, r}(\gamma \# \gamma') = \delta_{p, q}(\gamma) + \delta_{q, r}(\gamma')
		\end{equation}
		and
		\begin{equation}\label{eq:lamgl}
			\Lambda_{p, r}|_{\gamma \# \gamma'} \cong \Lambda_{p, q}|_\gamma \otimes \Lambda_{q, r}|_{\gamma'}
		\end{equation}
		for any two Whitney disks $\gamma$ and $\gamma'$ from $p$ to $q$ and from $q$ to $r$, respectively, and where $\gamma\# \gamma'$ is any choice of gluing, or more formally any Whitney disk from $p$ to $r$ which is the beginning endpoint of a 1-parameter family of Whitney disks converging in the Gromov sense to $(\gamma, \gamma')$ at the opposite endpoint. The isomorphism of the $\Lambda$'s should be thought of as \emph{data} rather than property, depending on the data of the 1-parameter family; it is also understood to be continuous in $\gamma, \gamma'$, and well-defined up to a contractible space of choices.
	\end{DEF}

	\begin{DEF}\label{def:z2gr}
		An \keywd{abstract $\mathbb Z/2$-grading} is an assignment ${\rm gr} : L_0 \cap_{\scr P} L_1 \to \mathbb Z/2$ such that $\delta_{p, q} = {\rm gr}(p) - {\rm gr}(q)$ mod 2, for all $p, q \in \mathbb Z/2$.
	\end{DEF}

	Such an abstract $\mathbb Z/2$-grading can be used to give a concrete $\mathbb Z/2$-grading of the chain group ${\it CL}_\scr P$ in the tautological way; this is because whenever the Floer chain complex is well-defined, the transversality condition ${\rm Cok}(D_\gamma) = 0$ is implicit, and so the virtual dimension $\delta_{p, q}(\gamma)$ is the actual dimension of the moduli space of parameterized holomorphic strips. In particular, since the Floer differential only counts strips with actual parameterized dimension 1, this means that it must flip the $\mathbb Z/2$-grading for it to be non-trivial.

	\begin{DEF}\label{def:cohsys}
		Define \keywd{$\fk o_{p, q}([\gamma])$} to be the set of choices of orientations of $\Lambda_{p, q}$ restricted to the component of $\gamma$ in the case that it is orientable, and the empty set otherwise, for any $[\gamma] \in \pi_0 \Omega_{p, q}\scr P$ (we often abuse notation and write $\gamma$ instead of $[\gamma]$). In view of (\ref{eq:lamgl}), there are well-defined maps
		\begin{equation}\label{eq:compor}
			\otimes : \fk o_{p, q}(\gamma) \times \fk o_{q, r}(\gamma') \to \fk o_{p, r}(\gamma \# \gamma'),
		\end{equation}
		given by the induced orientation on the tensor product of two oriented line bundles. Note that in our case $[\gamma\#\gamma'] \in \pi_0 \Omega_{p, r}$ depends only on $[\gamma], [\gamma']$, since the choice of interpolating 1-parameter family from $\gamma \# \gamma'$ to the broken $(\gamma, \gamma')$ in Gromov sense can be chosen to be small, and hence lie in a contractible space. Define an \keywd{abstract system of coherent orientations} to be any choice $\scr o_{p, q}(\gamma) \in \fk o_{p, q}(\gamma)$ for all $p, q, \gamma$, such that $(\scr o_{p, q}(\gamma), \scr o_{q, r}(\gamma'))$ maps to $\scr o_{p, q}(\gamma \# \gamma')$ under the map (\ref{eq:compor}) for all $p, q, r, \gamma, \gamma'$. Two such systems $\scr o, \scr o'$ are said to be \keywd{cohomologous} if there exists a function $\epsilon : L_0 \cap_{\scr P} L_1 \to \{\pm 1\}$ such that 
		\begin{equation}\label{eq:cohsys}
			\scr o'_{p, q}(\gamma) = \epsilon(p) \cdot \epsilon(q) \cdot \scr o_{p, q},
		\end{equation}
		for all $p, q, \gamma$. A cohomology class of abstract systems of coherent orientations can be called more succinctly an \keywd{abstract $\bb Z$-coefficient lift} for ${\it CF}_\scr P(L_0, L_1)$.
	\end{DEF}

	\begin{RMK}\label{rmk:signambig}
		The map (\ref{eq:compor}) is sensitive to the order of the tensor factors in (\ref{eq:lamgl}): indeed, if we switch the order, we incur an extra sign $(-1)^{\delta_{p,q}(\gamma) \cdot \delta_{q,r}(\gamma')}$. This subtlety will become relevant later, and is the cause of certain asymmetries. 
	\end{RMK}

	An abstract system $\scr o$ of coherent orientations allows us to define the Floer chain complex ${\it CL}_{\scr P}(L_0, L_1; \scr o)$ over $\bb Z$ by counting $\scr M_{p, q}$ with signs given by transporting the orientation $\scr o_{p, q}(\gamma)$ through the identification $T_\gamma\scr M_{p, q} \oplus \mathbb R \cong {\rm Ker}(D_\gamma)$, where $\mathbb R$ comes from the $\mathbb R$-translation action. The coherence ensures that the orientation on the 1-dimensional unparameterized moduli spaces respects the product orientation on the boundary, so that $d^2 = 0$. If $\scr o'$ and $\scr o$ are two cohomologous systems of coherent orientations, differing as in the equation (\ref{eq:cohsys}) above, then there is a natural isomorphism
	\[ \sigma_\epsilon : {\it CF}_\scr P(L_0, L_1; \scr o) \overset\sim\to {\it CF}_\scr P(L_0, L_1; \scr o') \]
	given by sending $[p] \mapsto \epsilon(p) [p]$. One has $\sigma_{\epsilon \cdot \epsilon'} = \sigma_{\epsilon'} \circ \sigma_{\epsilon}$ on the chain level, thus producing a transitive system of chain complexes over all the choices of systems of coherent orientation in a cohomology class.

	\begin{RMK}
		The reason for the name ``abstract'' in our definitions is that they are not necessarily in bijection with the choices of gradings and $\mathbb Z$-lifts of the actual chain complex. Indeed, some of the moduli spaces could end up being empty, in which case the choices of $\delta_{p, q}$ and $\scr o_{p, q}(\gamma)$ become irrelevant. However, the latter phenomenon is accidental, in the sense that if we apply several Heegaard moves the non-emptiness property of the moduli space need not be preserved (in fact even deforming the almost-complex structure could change this.) As such, the ``abstract'' versions should be thought of as being the ``natural choices.''
	\end{RMK}

	We are interested in finding if-and-only-if obstructions to the existence of abstract systems of $\mathbb Z/2$-gradings and coherent orientations, and also in understanding the set of such choices (up to cohomology in the latter case), when the obstructions vanish.
	
	\begin{DEF}
		Let $\Pi\scr P$ denote the fundamental groupoid of $\scr P$ with basepoints given by constant paths, i.e.~ elements of $L_0 \cap_\scr P L_1$, and let $\mathrm B\Pi \scr P$ denote its nerve as a simplicial set, i.e.~ with $n$-simplices given by composable arrows $p_0 \to p_1 \to \cdots \to p_n$ with the usual face and degeneracy maps.
	\end{DEF}

	\begin{RMK}
		Except for the degenerate case that $L_0 \cap_\scr P L_1 = \0$, the groupoid $\Pi \scr P$ is equivalent to the entire fundamental groupoid $\Pi_1 \scr P$. Since the exceptional case corresponds to the chain complex being zero on the nose, it is not very interesting, and we will henceforth assume that it does not occur.
	\end{RMK}

	We begin with $\mathbb Z/2$-gradings. The collection of all $\delta_{p, q}(\gamma)$ mod 2 give an $\mathbb F_2$-valued (simplicial) cochain
	\[ \delta \in C^1(\mathrm B\Pi\scr P; \mathbb F_2), \]
	and the additivity property (\ref{eq:deladd}) can be used to prove that this is a \emph{cocycle}. It is not hard to see that an abstract $\mathbb Z/2$-grading is precisely the data of a 0-cochain ${\rm gr} \in C^0(\mathrm B\Pi\scr P; \mathbb F_2)$ whose differential is delta: $d({\rm gr}) = \delta$. In other words, we have proven:

	\begin{PROP}\label{prop:grobs}
		There is a canonical $\mathbb F_2$-valued 1-cocycle in $B\Pi\scr P$ whose cohomology class vanishes if and only if an absolute $\mathbb Z/2$-grading exists, and such that the $\mathbb Z/2$-gradings are in bijection with 0-cochains that cobound it.
	\end{PROP}

	Next, we seek to prove an analogous statement for orientations. The first road-block is that some of the $\Lambda_{p,q}$ may be non-orientable. It turns out that if at least one $\Lambda_{p,q}|_{[\gamma]}$ is orientable, where $[\gamma] \in \pi_0 \Omega_{p,q}$ for some points $p, q$, then the same is true for all of the other ones. Indeed, given any other $[\gamma'] \in \pi_0\Omega_{p,q}$, there exists some $\gamma'' \in \Omega_{p, p}$ such that $\gamma' = \gamma'' \# \gamma$, in which case the gluing isomorphism (\ref{eq:lamgl}) gives us an isomorphism of bundles
	\[ \Lambda_{p,p}|_{\gamma''} \otimes \Lambda_{p,q}|_{[\gamma]} \overset\sim\to \Lambda_{p,q}|_{[\gamma']} \]
	over the connected components $[\gamma]$ and $[\gamma']$, so in particular one is orientable if and only if the other one is. Similarly, one can pre/post-compose with paths from/to different points in $L_0 \cap_\scr P L_1$ in order to prove invariance of the orientability property with respect to changing $p$ or $q$. Let $\Omega_0 \scr P$ denote the connected component of the loopspace $\Omega\scr P$ at the constant loop, and let us call 
	\begin{equation}\label{eq:probs}
		w_1(\Lambda_{p,p}|_{[{\rm id}_p]}) \in H^1(\Omega_0\scr P; \mathbb F_2)
	\end{equation}
	the \keywd{primary obstruction to coherent orientations}. We will be able to give an explicit formula for it later on. As the name suggests, there is a secondary obstruction which we are about to define, when the primary one does vanish.
	
	Let us assume that the primary obstruction vanishes, and let us begin by picking arbitrary orientations $\scr o_{p, q}(\gamma) \in \fk o_{p, q}(\gamma)$, for all $p, q, [\gamma]$ without any coherence assumption. In this case, given any two (homotopy classes of) composable Whitney disks $p \overset\gamma\rightsquigarrow q \overset{\gamma'}\rightsquigarrow r$, we obtain a well-defined element $\xi(\gamma, \gamma') \in \{\pm 1\}$, defined by
	\begin{equation}\label{eq:cohfail}
		\scr o_{p, r}(\gamma \# \gamma') = \xi(\gamma, \gamma') \cdot \big[\scr o_{p, q}(\gamma) \otimes \scr o_{q, r}(\gamma')\big],
	\end{equation}
	measuring the failure of coherence. This can be thought of as a $\{\pm1\}$-valued 2-cochain
	\[ \xi \in C^2(\mathrm B\Pi \scr P; \{\pm 1\}). \]
	Associativity of the $- \otimes -$ operation translates to the equation
	\[ \xi(\gamma \# \gamma', \gamma'') \cdot \xi(\gamma, \gamma') = \xi(\gamma, \gamma' \# \gamma'') \cdot \xi(\gamma', \gamma''), \]
	which after rearrangement of terms translates precisely into the 2-cocycle condition. Changing the original choice of $\scr o_{p, q}(\gamma)$ to $\scr o_{p, q}'(\gamma) := \eta_{p, q}(\gamma) \cdot \scr o_{p, q}(\gamma)$ for some 1-cochain $\eta \in C^1(\mathrm B\Pi\scr P; \{\pm 1\})$ changes the associated 2-cocycle by the formula
	\[ \xi'(\gamma, \gamma') = \eta(\gamma) \eta(\gamma') \eta(\gamma \# \gamma') \cdot \xi(\gamma, \gamma'), \]
	i.e.~ by the coboundary $d\eta$. In particular, the original choice of orientation assignment $\scr o_{p, q}(\gamma)$ can be changed to become a system of coherent orientations if and only if $[\xi] \in H^2(\mathrm B\Pi\scr P; \{\pm 1\})$ is zero. Moreover, this shows that the set of all systems of coherent orientations $\scr o_{p, q}(\gamma)$ is naturally a torsor over $C^1(\mathrm B\Pi\scr P; \{\pm 1\})$, if it is non-empty. The action of such a 1-cochain $\eta$ does not change the cohomology class of the system of coherent orientations if and only if $\eta = d \epsilon$ for a 0-cochain $\epsilon$. Thus, we have proven:

	\begin{PROP}\label{prop:orobs}
		Assuming the primary obstruction (\ref{eq:probs}) vanishes, there is a 2-cocycle in $C^2(\mathrm B\Pi\scr P; \{\pm 1\})$, called the \keywd{secondary obstruction to coherent orientations}, well-defined up to a transitive system of 2-coboundaries, whose induced class in $H^2$ is zero if and only if abstract systems of coherent orientations exist. The set of orientation choices modulo the cohomology relation is naturally a torsor over $H^1(\mathrm B\Pi\scr P; \{\pm 1\})$.
	\end{PROP}

	\subsection{Real Cauchy-Riemann operators, and calculation of obstructions} In order to elucidate these obstructions, let us recall some more facts about real Cauchy-Riemann (CR) operators. First of all, the necessary data needed to define a (real) linear Cauchy-Riemann operator are:
	\begin{list}{$\cdot$}{}
		\item A compact Riemann surface $(\Sigma, \partial \Sigma, j)$ with boundary, together with a finite set $S \subset \Sigma$ of punctures (which could occur either on the interior or on the boundary);
		\item A complex Hermitian vector bundle $(E, J, h)$ over $\Sigma \setminus S$, together with a sub-bundle $F \subset E|_{\partial \Sigma \setminus S}$ such that $E = F \oplus JF$ is an orthogonal decomposition with respect to $g := {\rm Re}(h)$;
		\item A $\bar\partial$-type connection, i.e.~ a $\mathbb C$-linear map $\bar\partial_E : C^\infty(E) \to C^\infty(\Omega^{0,1}\Sigma \otimes E)$ satisfying $\bar \partial_E (f \cdot s) = (\bar\partial f) \cdot s + f \cdot \bar \partial _E$;
		\item A volume form $d\nu$ that is asymptotically cylindrical near every puncture, which induces well-defined equivalence classes of Sobolev spaces $W^{1, 2}$ and $L^2$ on $E$.
	\end{list}
	This gives us a linear operator
	\[ D : W^{1,2}(\Sigma \setminus S; E, F) \to L^2(\Sigma \setminus S; \Omega^{0, 1}\Sigma \otimes E) \]
	defined as the Sobolev closure of $\bar \partial_E$, where the $F$ in the domain signifies that we only consider sections of $E$ whose restriction to the boundary lies in $F$. In order to get a Fredholm operator on $\Sigma \setminus S$ (a necessary precondition for having a well-defined index and determinant line), one requires the following extra asymptotic condition when approaching the punctures:
	\begin{list}{$\cdot$}{}
		\item near every point $p \in S$, either a left cylindrical end, of the form $(-\infty, R) \times [0, \pi]$ or $(-\infty, R) \times S^1$, or a right cylindrical end, of the form $(R, +\infty) \times [0, \pi]$ or $(R, +\infty) \times S^1$, parameterizing $\Sigma\setminus p$ near $p$; furthermore, an identification of $E$ over all fibers of a ray $(R, +\infty) \times t$ or $(-\infty, R) \times t$ for $t \in [0, \pi]$ or $t \in S^1$, and likewise for $F$ on $\{0, 1\}$ in the boundary case, so that:
		\item The volume form $d\nu$ is asymptotically cylindrical with respect to the aforementioned particular cylindrical end, and the operator $\partial_E$, when written in local coordinates $(s, t)$, asymptotically converges to $\partial_s + J \partial_t + A$, for $A$ a symmetric 0th order operator, so that $J \partial_t + A$ has zero kernel.
	\end{list}
	We know that the index as a number is preserved under homotoping the operator through Fredholm operators, and likewise the determinant lines are canonically isomorphic up to contractible choice induced by the homotopy, which is functorial with respect to composing homotopies. The choice of $\bar \partial_E$ connection, as well as the Hermitian structure and volume form $d\nu$ away from the punctures, are parameterized by a contractible space. As such, we may speak of the index or determinant line associated to the slightly smaller datum $(\Sigma, S, E, F, \ldots)$, where the extra asymptotical data near punctures (i.e.~ the cylindrical end, the form $d\nu$, and the ray trivializations) required for Fredholmness are written in ellipses. By abuse the notation, we refer to the datum $(\Sigma, S, E, F, \ldots)$ as ``the real CR operator'', though secretly it is an entire contractible parameter-family of such operators.
	
	The index and determinant line are also respected by any isomorphism 
	\[ (\Sigma, S, E, F, \ldots) \overset\sim\longrightarrow (\Sigma', S', E', F', \ldots'), \]
	commuting with all the data. The upshot of this discussion is that the index and determinant lines are abstract homotopical constructions associated to purely algebraic-topological input data.

	We may also change a real CR operator with a controlled effect on the index and determinant lines by gluing in either a formal disk or sphere bubble, which we now explain in slightly more detail. Indeed, let $\bar\partial_\Sigma = (\Sigma, S, E_\Sigma, F_\Sigma, \ldots)$ be a real CR operator, and $z \in \Sigma \setminus S$ be an interior point. Further, suppose $\bar\partial_{S^2} = (S^2, \0, E_{S^2}, \ldots)$ is a real CR operator on a sphere, together with an interior point $w \in S^2$, and equipped with identifications
	\[ T_z \Sigma \cong T_w^* S^2, \qquad E_\Sigma|_z \cong E_{S^2}|_w. \]
	Then, up to some choice of neck length, there is a well-defined connect-sum $\Sigma \# S^2$ at the basepoints $z, w$ which is still a Riemann surface, on which one obtains an induced operator
	\[ \bar\partial_{\Sigma \# S^2} = (\Sigma \# S^2, S, E_\Sigma \# E_{S^2}, F_\Sigma, \ldots). \]
	The index and determinant line change \cite[(11.12) of \S11c]{Sei08} as follows:
	\begin{equation}\label{eq:inds}
		{\rm Ind}(\bar\partial_{\Sigma \# S^2}) = {\rm Ind}(\bar\partial_{\Sigma}) - {\rm dim}_\mathbb R(E) + {\rm Ind}(\bar\partial_{S^2})
	\end{equation}
	\begin{equation}\label{eq:dets}
		{\rm Det}(\bar \partial_{\Sigma \# S^2}) \cong {\rm Det}(\bar\partial_{\Sigma}) \otimes (\Lambda^{\rm top}_\mathbb R E_{S^2}|_w)^{\vee} \otimes {\rm Det}(\bar\partial_{S^2}).
	\end{equation}
	(the last two tensor factors end up being trivial, but we write it like this to illustrate the similarity with gluing on a disk bubble, which does not enjoy the same cancelling property.) When gluing a bubble on top of another bubble, we may first glue the bubbles together and then glue the result to $\Sigma$, or the other way around. The two induced identifications between determinant lines end up being homotopic, so the process may be considered ``associative''. 
	
	Similarly, if $z \in \Sigma \setminus S$ is a boundary point, and $\bar \partial_{D^2} = (D^2, \0, E_{D^2}, F_{D^2}, \ldots)$ is a real CR operator on a disk, with a boundary point $w$ and an identifications
	\[ T_z (\partial \Sigma) \cong T^*_w (\partial D^2), \qquad F_{\Sigma}|_z \cong F_{D^2}|_w, \]
	there is an induced operator
	\[ \bar\partial_{\Sigma \# D^2} = (\Sigma \# D^2, S, E_\Sigma \# E_{D^2}, F_\Sigma \# F_{D^2}, \ldots), \]
	satisfying \cite[(11.11) of \S11c]{Sei08} 
	\begin{equation}\label{eq:indd}
		{\rm Ind}(\bar\partial_{\Sigma \# D^2}) = {\rm Ind}(\bar\partial_{\Sigma}) - {\rm dim}_\mathbb R(F) + {\rm Ind}(\bar\partial_{D^2})
	\end{equation}
	\begin{equation}\label{eq:detd}
		{\rm Det}(\bar \partial_{\Sigma \# D^2}) \cong {\rm Det}(\bar\partial_{\Sigma}) \otimes (\Lambda^{\rm top} F_{D^2}|_w)^{\vee} \otimes {\rm Det}(\bar\partial_{D^2}).
	\end{equation}
	The analogous associativity claim holds, with the caveat that whenever tensor factors are swapped in the commutative diagram, one must introduce the appropriate \emph{Koszul sign} (more concrete examples of this will show up in the proofs below). These operations can be thought of as keeping $\Sigma$ the same, except changing the bundles $E$ (and $F$) in a small neighborhood near $z \in \Sigma \setminus S$.

	Using these, one can deduce the following very standard (cf. \cite[Prop. 8.1.4]{FOOO10}) corollary:

	\begin{PROP}\label{prop:boundred}
		The index mod 2, as well as the determinant line, of a real CR operator on $\Sigma = D^2$ with punctures $S = \{\pm i\}$ (i.e.~ a Whitney disk), depend only $F$. The identification of determinant lines for Whitney disks with the same $F$ is compatible with the gluing isomorphism (\ref{eq:lamgl}).
	\end{PROP}

	\begin{proof}
		Assume $(D^2, \pm i, E, F, \ldots)$ and $(D^2, \pm i, E', F', \ldots)$ are two real CR operators with an identification of $E$ and $E'$ and the asymptotic ends near the punctures, as well as an identification of $F$ and $F'$ compatible with the isomorphisms $F \otimes \mathbb C \cong E$ and $F' \otimes \mathbb C \cong E'$. In particular, $E'$ is also identified with $E$ away from an interior point $z$, thanks to being complexifications of $F$ and $F'$. Hence, we may write $E' \cong E \# E_{S^2}$ for some complex bundle $E_{S^2}$ well-defined up to contractible space of choices on $S^2$ by the clutching isomorphism induced by trivializing $E$ and $E'$ outside a small neighborhood of $z$. The gluing formulas (\ref{eq:inds}, \ref{eq:dets}) give us the desired equality of index mod 2 and determinant, because the operator on $S^2$ is complex, hence having canonically trivial determinant and real dimension mod 2. In addition, for the determinant line we must show that the identifications are transitive with respect to a commutative triangle
		\[ \begin{tikzcd}
			(E^{(0)}, F^{(0)}, \ldots) \rar\ar[rr,bend left = 10] & (E^{(1)}, F^{(1)}, \ldots) \rar & (E^{(2)}, F^{(2)}, \ldots).
		\end{tikzcd} \]
		This is true thanks to the associativity of gluing two bubbles on top of each other:
		\[ E^{(2)} \cong E^{(1)} \# E_{S^2}^{1\to2} \cong E^{(0)} \# E_{S^2}^{0\to1} \# E_{S^2}^{1\to2} \cong E^{(0)} \# E_{S^2}^{0\to2}, \]
		where $E_{S^2}^{i\to j}$ represents the bundle on the bubble used to go from $E^{(i)}$ to $E^{(j)}$.

		Finally, to see compatibility of the identifications with gluing isomorphisms, we note that the outer hexagon in {\sc Fig.~\ref{fig:glsph}} can be filled in with two commutative squares, each being induced by a 2-parameter family of operators coming from doing the relevant degenerations in parallel.
		\begin{figure}
			\centering
			\includegraphics[scale=0.5]{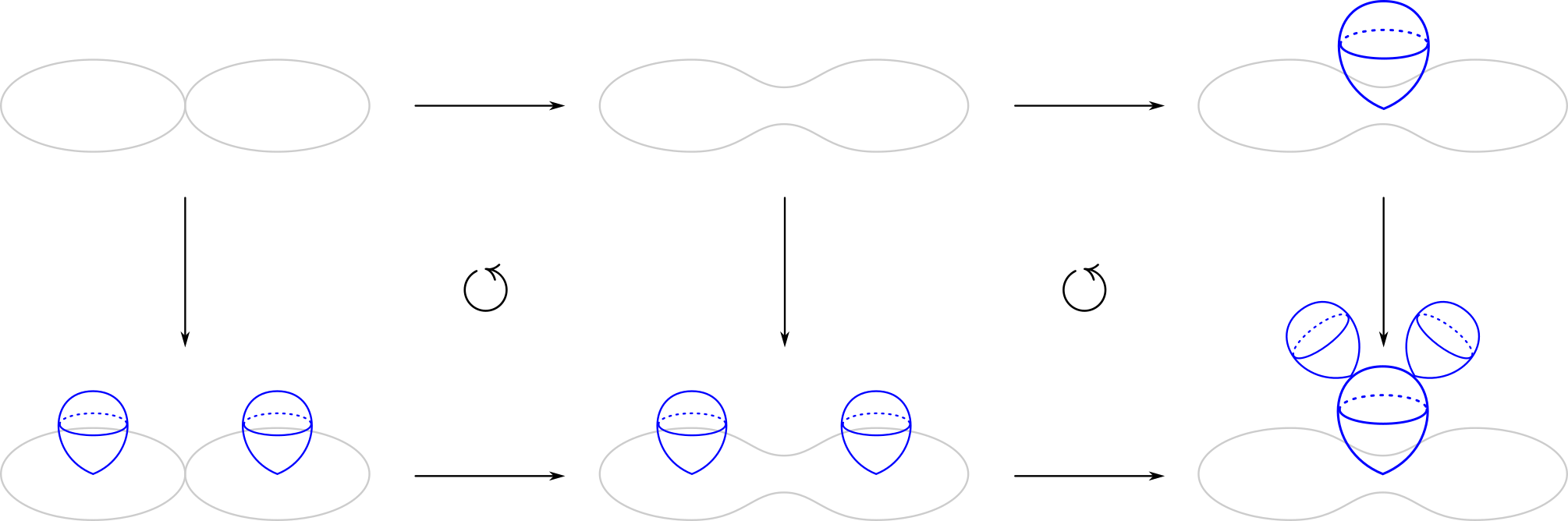}
			\caption{Depiction of gluing bubbles on composable Whitney disks, showing the compatibility of the induced maps on determinant lines. The first square commutes because the gluing of the sphere bubbles can be done independently of the gluing of the Whitney disks, hence inducing a 2-parameter family of operators depending on the gluing parameters used for both constructions. Likewise, the second square commutes because the gluing of the big bubble can be done independently of the gluing of the two smaller bubbles.}\label{fig:glsph}
		\end{figure}
	\end{proof}

	We will use this to express the obstructions in {\sc Props.~\ref{prop:grobs},~\ref{prop:orobs}} in terms of characteristic classes of the Lagrangian boundary difference introduced in {\sc Def.~\ref{def:lbd}}. First, let us reduce to the case of a groupoid with one object:

	\begin{PROP}
		The inclusion $\mathrm B\Pi_p \scr P \hookrightarrow \mathrm B\Pi \scr P$ of the full sub-groupoid with one object $p$ induces a chain-homotopy equivalence on $C^*$. If $p, q$ are two objects and $[\gamma]$ is a Whitney disk from $p$ to $q$, there is a well-defined conjugation map $c_{\gamma} : \Pi_p \overset\sim \to \Pi_q$ commuting up to canonical chain homotopy $h_\gamma$ with the inclusion in $\Pi \scr P$ on the level of $C^*$. The chain homotopies are coherent with respect to concatenation of Whitney disks:
		\[ h_{\gamma \# \gamma'} = h_{\gamma'} + h_{\gamma} \circ c_{\gamma'}. \]
	\end{PROP}

	\begin{proof}
		This is very standard: the inclusion $\Pi_p \scr P \hookrightarrow \Pi \scr P$ induces an equivalence of categories; and since natural transformations of functors induce canonical homotopies on the nerves, the first claim follows. The chain homotopy $h_{\gamma}$ is induced by the natural transformation relating $\Pi_p \hookrightarrow \Pi$ and the composite $\Pi_p \overset{c_\gamma}\to \Pi_q \hookrightarrow \Pi$ given by $\gamma$ itself, and the associativity of the $h$'s follows by the associativity of path composition.
	\end{proof}

	Thanks to this proposition, it suffices to prove any formula for the obstruction on its restriction to $C^*(\mathrm B\Pi_p \scr P; {\pm 1})$. The advantage here is that the identity in the groupoid has trivial index mod 2 and determinant line:

	\begin{PROP}\label{prop:idtriv}
		We have $\delta_{p,p}({\rm id}_p) \equiv 0 ~{\rm mod}~2$ and $\Lambda_{p,p}|_{{\rm id}_p} \cong \mathbb R$ canonically.
	\end{PROP}

	\begin{proof}
		By the gluing formulas (\ref{eq:deladd}, \ref{eq:lamgl}) applied to ${\rm id}_p$ and itself, we obtain
		\[ \delta_{p,p}({\rm id}_p) + \delta_{p,p}({\rm id}_p) = \delta_{p,p}({\rm id}_p)\]
		\[ \Lambda_{p,p}|_{{\rm id}_p} \otimes \Lambda_{p,p}|_{{\rm id}_p} \cong \Lambda_{p,p}|_{{\rm id}_p} \]
		from which we deduce the relevant triviality results.
	\end{proof}

	Now, by our boundary reduction result {\sc Prop.~\ref{prop:boundred}}, we obtain a well-defined locally constant \keywd{abstract index function $\delta_\partial$ mod 2}, as well as a well defined \keywd{abstract orientation bundle $\fk O_\partial$} on the ``parameter space of $F$'s on a Whitney disk with asymptotic data the same as ${\rm id}_p$ on both ends.'' More rigorously, let us observe that line bundles $F$ on $\mathbb R$ extending given values $F_\pm$ at $\pm \infty$ are uniquely determined up to contractible choice by the clutching function $\phi : F_- \overset\sim\to F_+$, if $F^0_-, F^0_+, F^1_-, F^1_+$ denote the four asymptotical values of the strip ${\rm id}_{p}$, with the subscript indicating the direction $\pm \infty$, and the superscript denoting the side of the Whitney disk. Then, we may think of the space
	\[ B_\partial := {\rm Iso}(F_-^0, F_+^0) \times {\rm Iso}(F_-^1, F_+^1) \]
	as parameterizing such choices of $F$ (in practice, $F^0_\pm$ and $F^1_\pm$ are the tangent spaces of the two Lagrangians.) Over each point in this space $B_\partial$, there is a well-defined $F$ on the boundary of the Whitney disk given by clutching, with the same asymptotic data as ${\rm id}_p$. Define $\delta_\partial$ to be the index mod 2 of the CR operator associated to any particular choice of filling $E$ of $F \otimes \mathbb C$ on the Whitney disk. Similarly, there is a 2-points set of orientations of $\Lambda$, well-defined up to canonical isomorphism, due to the choice of $E$ to fill it in. This can be used to define the abstract orientation double cover bundles $\fk O_\partial$ on $B_\partial$ via patching as follows: for every point $p \in B_\partial$, any small contractible neighborhood $U$, and every choice $E$ of filling on the interior, there is an honest determinant line bundle on $U$ associated to an actual CR operator up to contractible choice; there are obvious clutching functions between them, which satisfy the cocycle condition precisely because of the transitivity of the isomorphisms in {\sc Prop.~\ref{prop:boundred}}.

	Since we are working in the particular case of $F^0_- = F^0_+ = T_pL_0$ and $F^1_- = F^1_+ = T_pL_1$, thanks to the assumption that the asymptotic conditions are the same as for ${\rm id}_{p}$, we have
	\[ B_\partial = {\rm Aut}(T_pL_0) \times {\rm Aut}(T_pL_1). \]
	The following proposition describes $\delta_\partial$ and $\fk O_\partial$ on this space:

	\begin{PROP}\label{prop:dellamform}
		We have
		\begin{list}{$\cdot$}{}
			\item $\delta_{\partial}(\Phi_0, \Phi_1) = {\rm det}(\Phi_0) \cdot {\rm det}(\Phi_1) \in \{\pm 1\}$, under the obvious identification of $\{\pm 1\}$ with $\mathbb Z/2$.
			\item $\fk O_{\partial}$ can be canonically identified as the unit sphere bundle of with the external tensor product of the unique line bundles on ${\rm Aut}(T_pL_0)$ and ${\rm Aut}(T_pL_1)$ that are non-trivial on both the positive and negative components of these spaces.
		\end{list}
	\end{PROP}

	\begin{proof}
		We already know the fact that ${\rm id}_p$ (corresponding to $\Phi_0 = \Phi_1 = {\rm id}$) has canonically trivial $\delta$ and $\Lambda$, thanks to {\sc Prop.~\ref{prop:idtriv}}. We relate all other points in $B_\partial$ to this particular one by means of gluing disk bubbles, in a similar way to how we used sphere bubbles in the proof of {\sc Prop.~\ref{prop:boundred}}.

		For $\delta_\partial$, due to its local constancy, it suffices to determine its value at a single point in any connected component. We have already verified the desired formula at ${\rm id}_p$, corresponding to the connected component $(+, +)$ given by the determinant of $\Phi_0, \Phi_1$. Now, if either $\Phi_0$ or $\Phi_1$ has negative determinant, we can bubble off a disk bubble with non-trivial monodromy on the boundary in order to flip the sign of its determinant. The gluing formulas (\ref{eq:indd}, \ref{eq:detd}) tell us that the induced changes in index mod 2 are given by
		\[ {\rm Ind}(\bar \partial_{D^2}) - {\rm dim}_\mathbb R F, \] 
		which can be computed explicitly for any particular choice of real CR operator with non-trivial winding number. We can choose the $\bar \partial$ connection associated to the holomorphic vector bundle $\scr O(1) \oplus \scr O^{n-1}$ on $\mathbb {CP}^1$, and restrict this to the upper-half sphere to get a real CR operator with $F$ given by the real locus $\scr O_\mathbb R(1) \oplus \scr O_\mathbb R^{n-1}$. The kernel and cokernel correspond to the algebraic $H^0$ and $H^1$ of the given \emph{real} vector bundle over \emph{real} $\mathbb{P}^1$, and it can readily be checked by the Riemann-Roch formula that $h^0 - h^1 = n + 1$ in this particular case. More concretely, $H^0(\scr O(1))$ counts real homogeneous polynomials of degree 1 (which is a 2-dimensional space), and $H^0(\scr O)$ counts real homogeneous polynomials of degree 0 (which is a 1-dimensional space), while $H^1$ vanishes for both line bundles.

		We now turn to the orientation bundle. Since $H^1$ with $\mathbb Z/2$-coefficients of ${\rm SO}_n$ is isomorphic to $\mathbb Z/2$, by the same trick of gluing a disk bubble to either side, it suffices to find a loop of real CR operators on $D^2$ with fixed fiber at a distinguished point $w \in D^2$, whose monodromy flips the orientation on the determinant line, cf. (\ref{eq:detd}). To this end, consider the constant family of real CR operators associated to $\scr O(1)^{\oplus 2} \oplus \scr O^{\oplus (n-2)}$ on $\mathbb{CP}^1$, and glue the ends via the isomorphism $\psi$ switching the two $\scr O(1)$ summands, and flipping the sign of exactly one of the other $\scr O$ summands. Once again, real $H^1$ of the real bundle $\scr O_\mathbb R (1)^{\oplus 2} \oplus \scr O_\mathbb R^{\oplus n-2}$ is trivial, whereas real $H^0$ is given by $(\mathbb R^2)^{\oplus 2} \oplus (\mathbb R)^{\oplus(n-2)}$. Consequently, the map on determinant lines induced by the involution $\psi$ is the map that swaps the two copies of $\mathbb R^2$, and flips the sign of exactly one of the $\mathbb R$-entries. Now, when we glue this $S^1$-parameter family of disk bubbles to a Whitney disk, we need to identify the fiber of $F$ at the given point $z \in D^2$ where we are gluing with the fiber of $F_{D^2}$ at $w \in \partial D^2$. This can be done in an $S^1$-parameter family because the involution $\psi$ has positive determinant. Consequently, we obtain a loop in every connected component of $B_\partial$ whose monodromy on $\fk O_\partial$ is non-trivial, as claimed.
	\end{proof}

	Now, the total space of the orientation bundle $\fk O_\partial$, which is a $\{\pm 1\}$-principal bundle over the Lie group ${\rm Aut}(T_pL_0) \times {\rm Aut}(T_p L_1) \cong O_n \times O_n$, is itself a Lie group under gluing of orientations, and is in fact a central extension thereof:
	\begin{equation}\label{eq:fkO}
		1 \to \{\pm 1\} \to \fk O_\partial \to O_n \times O_n \to 1.
	\end{equation}
	Identifying this extension will help us make explicit the obstructions to coherent orientations, since a lift of structure group of $(TL_0, TL_1)$ from $O_n \times O_n$ to $\fk O_\partial$ is tantamount to such a choice of orientations. We begin by gathering a few facts about the extension, which will suffice for us to determine the obstructions:

	\begin{PROP}\label{prop:factsOn}
		The following are true:
		\begin{enumerate}
			\item[\sc i.] The restrictions of the extension (\ref{eq:fkO}) to the subgroups $O_n \times 1$ and $1 \times O_n$ are either ${\rm Pin}_n^+$ or ${\rm Pin}_n^-$.
			\item[\sc ii.] The signs of the Pin extensions are opposite for the said subgroups.
			\item[\sc iii.] If $o$ belongs to the preimage of $O_n \times 1$ inside $\fk O_\partial$, and $o'$ belongs to the preimage of $1 \times O_n$ inside $\fk O_\partial$, then the commutator $[o, o']$ equals $(-1)^{\frac{1 + {\rm det}(o)}2 \cdot {\frac{1 + {\rm det}(o')}2}}$, where det signifies the determinant of the matrix in $O_n$ underlying $o, o'$.
		\end{enumerate}
	\end{PROP}

	\begin{RMK}
		This proposition uniquely identifies the extension (\ref{eq:fkO}) except for the ambiguity of $\pm$ in the Pin extensions. However, as we have noted in {\sc Rmk.~\ref{rmk:signambig}}, there is a sign ambiguity coming from the convention of how to write composition of paths, which would change the flip the sign of both extensions, so this is less important (and will not change the formula for the obstruction, as we will see.)
	\end{RMK}

	\begin{proof}
		Extensions of $O_n$ are classified by $H^2(BO_n; \mathbb Z/2)$, which is freely generated by $w_2$ and $w_1^2$. The extensions corresponding to $0$ and $w_1^2$ are those that have four connected components, which are ruled out by {\sc Prop.~\ref{prop:dellamform}}, since $\fk O_\partial$ is non-trivial. The remaining two cohomology classes $w_2$ and $w_2 + w_1^2$ classify the ${\rm Pin}^\pm_n$ extensions, cf. \cite{GM25}, which finishes our proof of claim {\sc i.}

		For claim {\sc ii}, we note that the restriction of the ${\rm Pin}^\pm_n$ extension to the 2-element subgroup generated by $\iota := {\rm diag}(-1, 1, \ldots, 1) \subset O_n$ is $\mathbb Z/2 \oplus \mathbb Z/2$ for ${\rm Pin}^+$ and $\mathbb Z/4$ for ${\rm Pin}^-$, cf. \cite{GM25}. In particular, what distinguishes the two extensions is that the element $\iota$ lifts to elements of order 2 in ${\rm Pin}^+$, and to elements of order 4 in ${\rm Pin}^-$.
		
		Let us from here on trivialize $T_pL_0$ and $T_pL_1$, so that we can write $O_n$ instead of ${\rm Aut}(T_pL_0)$ and ${\rm Aut}(T_pL_1)$. Let $o_0 \in \fk O_\partial$ be any choice of orientation at ${(\iota, 1)}$, and likewise $o_1 \in \fk O_\partial$ for $(1, \iota)$. We will show that
		\begin{equation}\label{eq:orsignchange}
			o_0^2 = -o_1^2 \text{ in the group } \fk O_\partial,
		\end{equation}
		which easily implies claim {\sc ii}, in light of our remark about the orders of $o_0$ and $o_1$. To prove this, we use a trick similar to that in {\sc Fig.~\ref{fig:glsph}}, but with disk bubbles instead of sphere bubbles. Indeed, given any element $(x, 1)$ in $O_n \times 1$, we may compare its determinant line to that of $1 \times 1$ (which we have shown to be canonically trivial in {\sc Prop.~\ref{prop:idtriv}}), by bubbling off a disk bubble, which gives
		\begin{equation}\label{eq:dbubboff}
			\Lambda|_{(x, 1)} \cong ({\Lambda}^{\rm top}T_pL_0)^\vee \otimes {\rm Det}(\bar\partial_{D^2, x}^0)
		\end{equation}
		by applying (\ref{eq:detd}), where $\bar\partial_{D^2, x}^0$ is an operator on the disk, determined by $x$ and $T_pL_0$ (which we have trivialized, but we would like to keep it notationally like this for a while, for the sake of clarity). The left-hand side of this identification is \emph{multiplicative} in $x$, and in fact so is the right-hand side by gluing two disk operators to the trivial disk operator:
		\[ \begin{aligned}
			{\rm Det}(\bar\partial_{D^2, x' \cdot x}^0) &= {\rm Det}(\bar\partial_{D^2, 1}^0) \otimes ({\Lambda}^{\rm top}T_pL_0)^\vee \otimes {\rm Det}(\bar\partial_{D^2, x'}^0) \otimes ({\Lambda}^{\rm top}T_pL_0)^\vee \otimes {\rm Det}(\bar\partial_{D^2, x}^0) \\ &= {\rm Det}(\bar\partial_{D^2, x'}^0) \otimes ({\Lambda}^{\rm top}T_pL_0)^\vee \otimes {\rm Det}(\bar\partial_{D^2, x}^0)
		\end{aligned} \]
		where we have used the fact that $\bar\partial_{D^2, 1}^0$ is canonically identified with ${\Lambda}^{\rm top}(T_p L_0)$. The identification (\ref{eq:dbubboff}) respects the aforementioned multiplicative structures on both sides, by an argument very similar to that shown in {\sc Fig.~\ref{fig:glsph}}. In a very similar fashion, there is also a multiplicative identification on the other Lagrangian:
		\[ \Lambda|_{(1, x)} \cong (\Lambda^{\rm top}T_p L_0)^\vee \otimes {\rm Det}(\bar\partial^1_{D^2, x}), \]
		where $\bar \partial^1_{D^2, x}$ is a real CR operator on the disk, completely determined by $x$ and and $T_p L_1$.

		Now, recall that we have trivialized $T_p L_0$ and $T_p L_1$, so that in particular they are identified with each other. By abuse of notation, $x \in {\rm Aut}(T_p L_0)$ corresponds to another element that we also denote $x \in {\rm Aut}(T_p L_1)$. In this language, there is a canonical identification of $\bar\partial_{D^2, x}^0$ and $\bar\partial_{D^2, x^{-1}}^1$ (observe the inverse in $x^{-1}$), due to the fact that monodromy around the boundary is measured in different directions on $L_0$ and $L_1$ sides. In particular, there is an identification
		\[ (\Lambda^{\rm top}T_pL_0)^\vee \otimes {\rm Det}(\bar\partial^0_{D^2, x}) \cong (\Lambda^{\rm top}T_pL_1)^\vee \otimes {\rm Det}(\bar\partial^1_{D^2, x^{-1}}). \ \]
		This respects the multiplicative structures on both signs, bot only up to a Koszul sign, due to the need to swap the tensor factors (since $(xy)^{-1} = y^{-1}x^{-1}$, not $x^{-1}y^{-1}$). The sign is given by
		\[ (-1)^{\big({\rm Ind}(\bar\partial^1_{D^2, x^{-1}}) - {\rm dim}_{\mathbb R}(T_pL_1)\big) ~\cdot~ \big({\rm Ind}(\bar\partial^1_{D^2, x'{}^{-1}}) - {\rm dim}_{\mathbb R}(T_pL_1)\big)}. \]
		In the case that $x = x' = \iota$, we have already seen a concrete model for the operator on the disk, namely the real CR operator coming from the real line bundle $\scr O(1) \oplus \scr O^{n-1}$, so the quantity in either of the two big parentheses in the exponent is $1$, resulting in a non-trivial sign change, as promised in (\ref{eq:orsignchange}). This concludes the proof of claim {\sc ii}.

		Claim {\sc iii} also boils down to a Koszul sign: indeed, if $x, x'$ denote the underlying elements of the orthogonal group for $o, o'$, then both compositions $(x, 1) \cdot (1, x')$ and $(1, x') \cdot (x, 1)$ lead to the same element $(x, x')$, resulting in gluing isomorphisms
		\[ {\Lambda}_\partial|_{(x, 1)} \otimes \Lambda_\partial|_{(1, x')} \cong \Lambda_\partial|_{(x, x')} \cong {\Lambda}_\partial|_{(1, x')} \otimes \Lambda_\partial|_{(x, 1)}. \]
		We want to argue that the composite is induced by the expected Koszul sign. Indeed, we may use the identification (\ref{eq:dbubboff}) and its analogue for the $L_1$-side, to reduce the left side to
		\[ \big((\Lambda^{\rm top}T_pL_0)^\vee \otimes {\rm Det}(\bar\partial^0_{D^2, x})\big) \otimes \big((\Lambda^{\rm top}T_pL_1)^\vee \otimes {\rm Det}(\bar\partial^1_{D^2, x'})\big)  \]
		and the right side to the same expression with the two big parentheticals swapped. The resulting isomorphism is indeed the Koszul sign because gluing a disk bubble on either side of the same Whitney disk (in this case $(x, x')$) can be done in either order, and gluing disk bubbles in different orders results in a Koszul sign difference. This concludes the proof.
	\end{proof}

Before proving the main result of this section, we introduce some necessary algebraic topology that relates the various relevant cohomology groups.

	\begin{PROP}
		There is an exact sequence
		\begin{equation}\label{eq:lowterm}
			0 \to H^2(B\Pi\scr P; \mathbb Z/2) \to H^2(\scr P; \mathbb Z/2) \to H^1(\Omega_0 \scr P; \mathbb Z/2)^{\pi_1 \scr P} \to H^3(B\Pi \scr P; \mathbb Z/2),
		\end{equation}
		where the superscript $\pi_1 \scr P$ is used to indicate the fixed-points of the natural action of $\pi_1 \scr P \cong \pi_0 \Omega \scr P$ on the cohomology group via the conjugation of $\Omega \scr P$ on $\Omega_0 \scr P$.
	\end{PROP}

	\begin{proof}
		Note that $B\Pi \scr P$ is the 1st Postnikov truncation of $\scr P$, and consequently there is a fiber sequence
		\[ F \to \scr P \to B\Pi\scr P \]
		where $\pi_0 F = \pi_1 F = 0$ and $\pi_2 F = H_2 F = \pi_2 \scr P = \pi_1 \Omega_0 \scr P$. Thanks to the vanishing of $H^1 F$, the $q = 1$ row in the Serre spectral sequence (with local coefficients due to the potential non-simple-connectivity of $B\Pi\scr P$)
		\[  E_r^{p, q} = H^p(B\Pi \scr P; \underline H^q(F; \bb Z/2)) \cong H^p_{\it gr}(\pi_1 \scr P; H^q(F; \bb Z/2)) \]
		associated to this fibration vanishes, which leaves us with a low term exact sequence
		\[ 0 \to E_\infty^{2,0} \to H^2(\scr P; \mathbb Z/2) \to E_\infty^{0,2} \to 0. \]
		the subgroup $E_\infty^{2,0}$ is $E_2^{2,0} = H^2(B\Pi \scr P; \mathbb Z/2)$, since there can be no non-trivial differential into this spot. When it comes to $E_\infty^{0,2}$, this is equal to 
		\[ {\rm Ker}\big(d_3^{0,2} : H^2(F; \bb Z/2)^{\pi_1\scr P} \to H^3(B\Pi\scr P; \bb Z/2)\big),\]
		which proves the claim due to the isomorphism 
		\[H^2(F; \mathbb Z/2) = {\rm Hom}(\pi_2 \scr P, \mathbb Z/2) = {\rm Hom}(\pi_1 \Omega \scr P; \mathbb Z/2) = H^1(\Omega_0 \scr P; \mathbb Z/2). \qedhere \]
	\end{proof}

	\begin{PROP}\label{prop:totobs}
		Let $\xi \in H^2(\scr P; \mathbb Z/2)$ be the obstruction to lifting the structure group of the pair $(TL_0, TL_1)$ on $\scr P$ from $O_n \times O_n$ to $\fk O_\partial$. The image of $\xi$ under the second map in (\ref{eq:lowterm}) recovers the primary obstruction to coherent orientations. When this vanishes, the unique lift along the first map (\ref{eq:lowterm}) is equal to the secondary obstruction to coherent orientations. In particular, both obstructions to coherent orientations vanish if and only if $\xi$ vanishes, which motivates us to name it the \keywd{total obstruction to coherent orientations}.
	\end{PROP}

	\begin{proof}
		Let us recall that the primary obstruction, which is an element in $H^1(\Omega_0 \scr P; \mathbb Z/2)$, is equal to $w_1$ of the determinant line bundle of the real CR operators coming from Whitney disks from the basepoint $p$ to itself that are homotopic to the identity. By definition of $\fk O_\partial$, a choice of orientation of $\Lambda_{p,p}|_{[\rm id_p]}$ is equivalent to a lift of the map $\Omega_0\scr P \to B_\partial \cong O_n \times O_n$ along the 2-to-1 cover $\fk O_\partial \to O_n \times O_n$. Thus, the primary obstruction is the obstruction to this lifting problem. We now reformulate this lifting problem: if $c$ is a loop in $\Omega_0 \scr P$, then it traces out a 2-sphere called $\Sigma c$ inside $\scr P$, on which the $O_n \times O_n$-bundle $(TL_0, TL_1)|_{\Sigma c}$ can be reconstructed via the clutching function given by $S^1 \to O_n \times O_n$; a lift of this equatorial map to $\fk O_\partial$ would be equivalent to a lift of structure group of $(TL_0, TL_1)|_{\Sigma c}$ from $O_n \times O_n$ to $\fk O_\partial$. In summary, we have established that the following two lifting problems have the same obstruction:
		\begin{enumerate}
			\item[\sc i.] lifting the map $S^1 \overset c\to \Omega_0 \scr P \to O_n \times O_n$ to $\fk O_\partial$, whose obstruction is given by pairing the primary obstruction to coherent orientations with $c$;
			\item[\sc ii.] lifting the structure group of the bundle $(TL_0, TL_1)|_{\Sigma c}$ from $O_n \times O_n$ to $\fk O_\partial$, whose obstruction is given by pairing $\xi$ with $\Sigma c$.
		\end{enumerate}
		Since this is true for all $[c] \in \pi_1 \Omega_0 \scr P \cong \pi_2 \scr P$, and we have the equality $H^1(\Omega_0 \scr P; \bb Z/2) \cong {\rm Hom}(\pi_2 \scr P, \bb Z/2)$, this concludes the proof of the first claim.
		
		For the second claim, let us assume that the primary obstruction to coherent orientations is zero. In this case, any orientation of $\Lambda_{p,p}|_\gamma$ can be transported to an orientation of $\Lambda_{p,p}|_{\gamma'}$ for any two homotopic Whitney disks $\gamma, \gamma'$, independently of the chosen homotopy. In particular, the obstruction $\xi \in H^2(\scr P; \mathbb Z/2)$, when evaluated to a 2-simplex, only depends on the homotopy classes of its edges, i.e.~ it is zero if the $\fk O_\partial$-lift over the compositions $0 \to 1 \to 2$ and $0 \to 2$ agree, and nonzero otherwise. This is precisely the definition of the secondary obstruction to coherent orientations, which finishes the proof of the second claim.
	\end{proof}

	\begin{proof}[Proof of {Thm.~\ref{thm:obs}}]
		For gradings, note that $\delta_{p,p}(\gamma)$ is given by the sum of the monodromies of $TL_0$ and $TL_1$ under the loops traced by the two sides of the Whitney disk $\gamma$ inside $L_0$ and $L_1$. This is precisely the evaluation of $w_1(TL_0) + w_1(TL_1) = w_1(TL_0 - TL_1)$ on the loop $\gamma$, hence representing the desired cohomology class.

		Now, let us move on to coherent orientations, where thanks to {\sc Prop.~\ref{prop:totobs}}, it remains to check what the obstructions to lifting are for the map $B\fk O_\partial \to BO_n \times BO_n$. Indeed, there is a fiber sequence
		\[ B\{\pm 1\} \to B\fk O_\partial \to BO_n \times BO_n \to B^2\{\pm 1\}, \]
		where the last map is the one classifying the extension (\ref{eq:fkO}). The second cohomology group of $BO_n \times BO_n$ with $\mathbb Z/2$-coefficients is freely generated by:
		\[ w_{1, (0)}^2,~ w_{1, (1)}^2,~ w_{1,(0)} w_{1, (1)},~ w_{2, (0)},~ w_{2, (1)}, \]
		where the number in parentheses denotes whose factor the Stiefel-Whitney class is taken with respect to. First, let us see that the coefficient of $w_{1, (0)} w_{1, (1)}$ must be nonzero. Indeed, the subgroup of $H^2(BO_n \times BO_n; \mathbb F_2)$ consisting of linear combinations with no $w_{1, (0)} w_{1, (1)}$ is precisely the joint image of the inclusion maps of the two factors of $O_n \times O_n$, and in particular the extensions $E$ obtained in this way have the property that if $g, g' \in E$ are elements such that $g$ lands in $O_n \times 1$ and $g'$ lands in $1 \times O_n$, then $gg' = g'g$. But this is false for $E = \fk O_\partial$, as we have shown in part {\sc iii} of {\sc Prop.~\ref{prop:factsOn}}.

		Next, we see that the $w_2$'s must both have nonzero coefficients. Indeed, the restriction to each factor $BO_n \times 1$ and $1 \times BO_n$ must be $w_2$ or $w_2 + w_1^2$, thanks to {\sc Prop.~\ref{prop:factsOn}-ii}. Finally, we see that precisely one of $w_{1, (0)}^2$ and $w_{1, (1)}^2$ must be nonzero, thanks to {\sc Prop.~\ref{prop:factsOn}-iii}. The ambiguity concerning which one survives becomes irrelevant when we pull-back the class to the pathspace: indeed, on $\scr P$ we have that $TL_0 \otimes \mathbb C$ and $TL_1 \otimes \mathbb C$ are both identified with $TM$, so in particular $w_2(TL_0^{\oplus 2}) = w_2(TL_1^{\oplus 2})$, i.e.~ $w_1^2(TL_0) = w_1^2(TL_1)$. From here on, a straightforward calculation using Cartan's formula shows that the cohomology class on the pathspace is given by $w_2(TL_1 - TL_0) = w_2(TL_0 - TL_1)$, as desired:
		\[ \begin{aligned}
			w_2(TL_1 - TL_0) &= w_2(TL_1) + w_1(TL_1) w_1(-TL_0) + w_2(-TL_0) \\ 
			&= w_2(TL_1) + w_1(TL_1) w_1(TL_0) + w_2(TL_0) + w_1^2(TL_0),
		\end{aligned}  \]
		where the general formulas $w_1(-E) = w_1(E)$ and $w_2(-E) = w_2(E) + w_1^2(E)$ for any virtual bundle $E$ follow by solving for the total Stiefel-Whitney class $w(-E)$ in the Cartan formula $1 = w(E \oplus (-E)) = w(E) \cdot w(-E)$.
	\end{proof}

	\begin{RMK}\label{rmk:larged}
		In the context of real Heegaard Floer theory, our computation of the Stiefel-Whitney classes in {\sc Thm.~\ref{thm:wi}} shows that the condition for vanishing of $w_1$ and $w_2$ is independent of $d$, for $d$ sufficiently large. Note also that in this section we have been assuming that $n \ge 3$, which was needed in some of the specific computations involving real vector bundles on $\bb P^1$, which is another independent reason to require $d$ to be large. This is not an issue in our case, since we can always apply enough stabilizations to achieve the desired bound on $d$. In the process of stabilizing, it is possible that some Whitney disks may become homotopic that weren't before.
	\end{RMK}

	\begin{RMK}\label{rmk:twgr}
		Even in the case that only the primary obstruction to orientability vanishes, i.e.~ the image of $w_2$ in $H^1(\Omega_0 \scr P; \mathbb Z/2)$ is zero, it is still possible to define a \emph{local coefficient} chain complex over a \keywd{twisted group algebra $\mathbb Z[\pi_1 \scr P, \xi]$}, twisted by the group 2-cocycle $\xi$ that is the secondary obstruction to coherent orientations. This follows an idea of Rezchikov \cite{Rez21}, who twists $\pi_1 L_i$ instead of $\pi_1 \scr P$. This twisted group algebra is constructed, as an abelian group, in the same way as the usual group algebra, but is endowed with a twisted multiplication:
		\[ [g] \cdot_\xi [h] := \xi(g,h) \cdot [gh]. \]
		The 2-cocycle condition implies associativity, with unit $\xi(1, 1) \cdot [1]$. Adjusting $\xi$ by a 2-coboundary $d\eta$ induces an isomorphism $\mathbb Z[G, \xi] \overset\sim\to \mathbb Z[G, \xi + d\eta]$ given by sending $g \mapsto \eta(g) \cdot g$. Thus, the isomorphism class of the group algebra only depends on the group cohomology class. An augmentation map $\epsilon : \mathbb Z[G, \xi] \to \mathbb Z$ sending each $g$ to $\pm 1$ exists if and only if the twisted group algebra is standard.

		The chain complex over this twisted group algebra is defined in the following manner: choose a privileged point $z \in L_0 \cap_\scr P L_1$, and let
		\[ \underline{\it CL}_\scr P := \bigoplus_{p \in L_0 \cap_\scr P L_1} \mathbb Z[\pi_0 \Omega_{z,p} \scr P] \]
		with differential given by
		\[ \pi_0 \Omega_{z,p} \scr P \ni [\gamma_{z,p}] \longmapsto \sum_{\substack{\gamma_{p, q} \in \scr M_{p, q}\\ \mu(\gamma_{p,q}) = 1}} \xi([\gamma_{z,p}], [\gamma_{p,q}]) \cdot [\gamma_{z,p} \# \gamma_{p,q}],  \]
		with respect to the \emph{non-coherent} orientations on the $\scr M_{p, q}$, whose failure of coherence is measured by $\xi$, cf (\ref{eq:cohfail}). The 2-cocycle condition ensures that $d^2 = 0$. There is an obvious action of $\mathbb Z[\pi_1 \scr P, \xi] = \mathbb Z[\pi_0 \Omega_{z,z}]$ by pre-composition, and $d$ respects this action.
	\end{RMK}

	\subsection{Continuation maps}\label{ssec:cmaps} The definition of Lagrangian Floer homology depends on some choices, and in order to prove independence of said choices, one must construct chain maps between Floer chains using non-translation invariant CR equations, and then provide chain homotopies witnessing the fact that they are mutual inverses, via 1-parameter families of CR equations. In this section, we extend the obstruction theory that we discussed above, in order to account for these generalizations of the CR equations studied previously. That being said, the formal apparatus we are about to construct applies equally well to any kind of map between Floer homologies, not necessarily one inducing a quasi-isomorphism, and will probably be essential in developing functoriality of the theory, though we do not pursue this in the present paper.

	\emph{Continuation maps} can be induced either by a Hamiltonian isotopy of one of the two Lagrangians, or by a 1-parameter family of almost-complex structures interpolating between the two endpoints. The formalism we are about to develop is so similar for the two, that we only discuss the Hamiltonian isotopy explicitly, and leave it to the reader to see that the same story applies for almost-complex structures (in fact this latter one is simpler, since the pathspace does not change). Assume, without loss of generality, that the first Lagrangian $L_1$ is being isotoped. That it, assume that $L_1' := \Phi_1 L_1$ is the time-1 image of a time-dependent Hamiltonian diffeomorphism $\Phi_t$ with $\Phi_0 = {\rm id}$. Since isotopy in particular implies homotopy, it follows that there is an induced homotopy-equivalence relating the two pathspaces, so that connected components thereof are in canonical bijection. Let $\scr P$ and $\scr P'$ denote two components of the pathspaces, that correspond to each other under this bijection. The way to construct a map
	\[ {\it CL}_\scr P(L_0, L_1) \to {\it CL}_{\scr P'}(L_0, L_1'), \]
	is by considering pseudo-holomorphic strips where the 0th order term of the equation is perturbed by the Hamiltonian $H_t$ generating $\Phi_t$. The linearization of these PDEs is only slightly more general than the real CR operators considered previously, in that the 0th order term can now have a complex-conjugate linear term. Nevertheless, this does not change the homotopy type of the space of operators, so the difference is inconsequential.

	We are now led to make definitions of virtual dimensions, determinant lines, and coherent gradings/orientations for the continuation maps that closely mimic the classical versions. To begin with, if $p \in L_0 \cap_\scr P L_1$ and $q' \in L_0 \cap_{\scr P'} L_1'$ are two generators of the two complexes we wish to compare, we may define a space \keywd{$\Omega_{p,q'}^\Phi$} of paths connecting $p$ to $q$ through the intermediary pathspaces, i.e.~ more formally defined to be a time-dependent family of paths $\gamma(t) \in \scr P_M(L_0, \Phi^t(L_1))$ as $t$ varies from 0 to 1, such that $\gamma(0) = p, \gamma(1) = q'$ (the regularity of $\gamma$ is irrelevant as far as the homotopy type is concerned.) There are pre- and post-concatenation maps
	\[ \Omega_{\tilde p, p}\scr P \times \Omega^\Phi_{p, q} \to \Omega^\Phi_{\tilde p, q}, \qquad \Omega^\Phi_{p, q} \times \Omega_{q, \tilde q} \scr P' \to \Omega^\Phi_{p, \tilde q}. \]
	As in the classical case, we still have \keywd{virtual dimension functions $\delta_{p, q'}$} and \keywd{determinant line bundles $\Lambda_{p, q'}$} over $\Omega^\Phi_{p, q'}$, which satisfy very similar relations to (\ref{eq:deladd}, \ref{eq:lamgl}) for pre- and post-concatenation. We still use the notation $\mathfrak o_{p, q'}([\gamma])$ to denote the (possibly empty) set of orientations of $\Lambda_{p, q'}$ on the connected component of $\gamma$ inside $\Omega^\Phi_{p, q'}$.
	\begin{DEF}
		Based on this, we may define:
		\begin{enumerate}
			\item[\sc i.] An \keywd{abstract joint $\mathbb Z/2$-grading} is an assignment ${\rm gr} : (L_0 \cap_\scr P L_1) \sqcup (L_0 \cap_{\scr P'} L_1') \to \mathbb Z/2$ such that
			\[ {\rm gr}(\tilde p) - {\rm gr}(p) \equiv \delta_{\tilde p, p}, \qquad {\rm gr}(p) - {\rm gr}(q') \equiv \delta_{p, q'}, \qquad {\rm gr}(q') - {\rm gr}(\tilde q') \equiv \delta_{q', \tilde q'} \]
			mod 2 for all $\tilde p, p, q', \tilde q'$.
			\item[\sc ii.] An \keywd{abstract joint system of coherent orientations} is a choice of elements 
			\[ \scr o_{\tilde p, p}([\gamma_{\tilde p, p}]) \in \fk o_{\tilde p, p}([\gamma_{\tilde p, p}]), \quad \scr o_{p, q'}([\gamma_{p, q'}]) \in \mathfrak o_{p, q'}([\gamma_{p, q'}]), \quad \scr o_{q', \tilde q'}([\gamma_{q', \tilde q'}]) \in \mathfrak o_{q', \tilde q'}([\gamma_{q', \tilde q'}]) \]
			such that they are multiplicative with respect to the tensor product of orientations induced by all forms of concatenation.
		\end{enumerate}
	\end{DEF}

	These joint versions in particular contain the data of classical $\mathbb Z/2$-gradings/coherent orientations for the two particular Floer chains, but in addition also include the necessary properties/data to ensure that the continuation map also respects said structures. As in the classical case, there is a primary obstruction to orientability for the continuation map, which lives in $H^1(\Omega^\Phi_{p, q'}; \pm 1)$, for all $p$ and $q'$.

	\begin{PROP}
		The following are equivalent:
		\begin{enumerate}
			\item The primary obstruction to coherent orientations for ${\it CL}_\scr P(L_0, L_1)$ vanishes;
			\item The primary obstruction to coherent orientations for ${\it CL}_{\scr P'}(L_0, L_1')$ vanishes;
			\item The primary obstruction to coherent orientations for the continuation map relating the two vanishes, for all $p, q'$.
			\item The same obstruction vanishes for any particular choice of $p, q'$.
		\end{enumerate}
	\end{PROP}

	\begin{proof}
		This boils down to the existence of homotopy equivalences
		\[ \Omega_{\tilde p, p} \scr P \cong \Omega^\Phi_{p, q'} \cong \Omega_{q', \tilde q'} \scr P' \]
		compatible with the determinant line bundles. Indeed, these homotopy equivalences are induced by pre- or post-composition with any particular element of $\Omega_{\tilde p, p} \scr P$ or $\Omega_{q', \tilde q'}\scr P'$, and the inverse homotopy equivalence is given by composing with the reverse path.
	\end{proof}

	This now leads us to consider the obstruction to joint $\mathbb Z/2$-gradings and the secondary obstruction to joint coherent orientations. These also live in the simplicial cohomology groups of the nerve of a particular category, now no longer a groupoid:

	\begin{DEF}
		The \keywd{homotopy category $\Pi^\Phi$} associated to the continuation map is given by:
		\begin{list}{$\cdot$}{}
			\item Object set consisting of the disjoint union $(L_0 \cap_\scr P L_1) \sqcup (L_0 \cap_{\scr P'} L_1)$;
			\item Morphisms given by $\pi_0 \Omega_{\tilde p, p} \scr P, \pi_0 \Omega^\Phi_{p, q'}, \pi_0 \Omega_{q', \tilde q'} \scr P'$ for all $\tilde p, p, q', \tilde q'$ (and understood to be empty from $q'$ to $p$).
			\item Composition given by concatenation.
		\end{list}
	\end{DEF}

	\begin{PROP}
		The obstructions to abstract joint $\mathbb Z/2$-gradings and joint systems of coherent orientations (assuming the primary obstruction already vanishes) lie in simplicial $H^1$ and $H^2$ of the nerve $\scr N\Pi^\Phi$ with $\mathbb Z/2$-coefficients. When the obstructions vanish, the sets of choices of joint $\mathbb Z/2$-gradings and systems of joint coherent orientations modulo cohomology are torsors over $H^0$ and $H^1$, respectively.
	\end{PROP}

	The proof is so similar to the classical case that we omit it.

	\begin{RMK}
		Regarding the structure of $\Pi^\Phi$, we notice that it is no longer a groupoid, but rather it is equivalent to a 2-object category -- let us call the objects $0$ and $1$, respectively -- with no morphisms going from 1 to 0, and such that all endomorphisms are invertible. Such a category is uniquely determined by groups $G_0 = {\rm Aut}(0), G_1 = {\rm Aut}(1)$, as well as a set $X_{01} := {\rm Hom}(0, 1)$ endowed with a balanced $(G_0, G_1)$-bimodule action.
		
		The simplicial chains on the nerve of such a category are given by the two-sided bar construction $B_\bullet (\mathbb Z[G_0], \mathbb Z[X_{01}], \mathbb Z[G_1])$, and constitute a concrete model for the derived tensor product
		\[ \mathbb Z \otimes^{\bf L}_{\mathbb Z[G_0]} \mathbb Z[X_{01}] \otimes^{\bf L}_{\mathbb Z[G_1]} \mathbb Z \]
		whose (co)homology computes group (co)homology of the module $\mathbb Z[X_{01}]$ over $G_0 \times G_1$. This formula may be of use in the broader study of functorial maps in Heegaard Floer type theories, though in this paper we content ourselves to continuation maps, for which this derived algebra language will not be necessary.
	
		Indeed, $\Pi_\Phi$ has one additional important property, namely that $X_{01}$ is a torsor over both $G_0$ and $G_1$. Indeed, any two paths can be subtracted by gluing at either end, and deforming to get an element of $G_0$ or $G_1$. This key observation allows us to conclude the following:
	\end{RMK}

	\begin{PROP}
		The inclusions of $\Pi\scr P$ and $\Pi \scr P'$ into $\Pi^\Phi$ induce chain-homotopy equivalences on the level of simplicial (co)chains. In particular, the obstructions for abstract $\mathbb Z/2$-gradings and abstract systems of coherent orientations for ${\it CF}_\scr P(L_0, L_1)$ vanish if and only if they vanish for ${\it CF}_{\scr P'}(L_0, L_1')$, if and only if they vanish for the joint versions associated to the continuation maps. Moreover, any choice of abstract $\mathbb Z/2$-gradings or system of coherent orientation for any of the two chain theories automatically induces a canonical choice for the joint version associated to the continuation map, and hence also for the other chain theory.
	\end{PROP}

	\begin{proof}
		The fact the $X_{01}$ is a torsor over both $G_0$ and $G_1$ implies that the category $\Pi^\Phi$ is equivalent to $\Pi \scr P \times (0 \to 1)$ and $\Pi \scr P' \times (0 \to 1)$, in a way compatible with the inclusion of $\Pi \scr P \times 0$ and $\Pi \scr P' \times 1$. Indeed, if we choose a privileged element $x \in X_{01}$, we get a bijection $G_0 \overset\sim\to X_{01}$ given by $g \mapsto gx$, and also a group isomorphism $\phi : G_0 \overset\sim\to G_1$ by sending $g \in G_0$ to the unique element $\phi(g) \in G_1$ satisfying $gx = x \phi(g)$. This gives a morphism $(G_0, G_0, G_0) \to (G_0, X_{01}, G_1)$ of groups and bimodules, which can be used to provide the desired equivalence; the $G_1$ case is very similar.
	\end{proof}

	The upshot of what we have just proved is that once the job has been done for one of the chain theories, everything desirable also follows for free for the other chain theory, as well as the continuation map between them. The last step is to establish a similar statement also for the chain homotopies relating the mutual composites of the forwards and backwards continuation maps to the identity maps, hence proving the chain-homotopy invariance of the construction with respect to Hamiltonian isotopies. As far as absolute $\mathbb Z/2$-gradings, there is no extra data to be provided, or conditions to be checked, since the condition that a chain map be a quasi-isomorphism is independent of the gradings. However, the same is not true with respect to $\mathbb Z$-gradings, since it is possible for a map to be a $\mathbb Z/2$-quasi-isomorphism, but not a $\mathbb Z$-quasi-isomorphism.

	The equations one uses to produce the chain-homotopies now come in a 1-parameter family, interpolating between the concatenation $\Phi \star \bar\Phi$, where the bar denotes the reverse, i.e.~ $\bar \Phi(t) = \Phi(1 - t)$, and the identity isotopy (which induces the identity chain map). This extra parameter is incorporated also in the moduli space, i.e.~ we are looking at the moduli space of solutions to \emph{any} of the intermediate equations all at once, and smoothness of the moduli space can only be guaranteed jointly. As a result, to compute the tangent space to the moduli space, one must add the extra parameter to the domain of the linearized real CR operator. Nevertheless, it still holds true that the virtual dimension and the determinant line of this ``augmented'' Fredholm operator are homotopy-invariant, so we can generalize all the concepts from before, in particular the determinant line $\Lambda_{p,q}^{\rm aug}$ for these augmented linear PDEs.

	There is a canonical identification $\Lambda_{p,q}^{\rm aug} \cong \Lambda_{p, q}$, coming from the fact that the two linear problems are related by adding a copy of $\mathbb R$ to the domain of the non-parametric equation, to obtain the parametric one. Thus, once we have chosen systems of joint coherent orientations for both directions, compatible with the coherent orientations for the two chain theories we are comparing, the remaining coherence condition for the chain-homotopies takes the form
	\[ \scr o_{p, q'}([\gamma_{p, q'}]) \otimes \scr o_{q', r}([\gamma_{q', r}]) = \scr o_{p, r}([\gamma_{p, q'} \# \gamma_{q', r}]) \]
	and likewise for the concatenation going the other way around. The existence of such coherent choices can also be interpreted as the vanishing of a certain obstruction in $H^2$ of the nerve of a category.

	Indeed, we may define the bi-directional groupoid $\Pi_{\leftrightarrows}^\Phi$ in the same way as the one-directional homotopy category $\Pi^\Phi$ as before, but now allowing morphisms both from $\scr P$ to $\scr P'$ and from $\scr P$ to $\scr P'$. This is now a connected groupoid, and hence the inclusions of the groupoids $\Pi \scr P$ and $\Pi \scr P'$ are both homotopy-equivalences of Kan complexes, thus inducing isomorphisms on simplicial (co)chains of their nerves. Consequently, once again the vanishing of the obstruction for the chain-homotopy is equivalent to the vanishing of the obstruction for either of the two chain theories, and moreover any choice of witness that the obstruction is zero extends uniquely to the chain-homotopies. This provides the proof for {\sc Thm.~\ref{thm:invar}}.

\section{Proof of the main result}\label{sec:prf}

	Now, we assemble all the previous facts that we have established, to give a proof of {\sc Thm.~\ref{thm:main}}. In {\sc\S\ref{sec:gror}}, we have identified the necessary and sufficient conditions for (abstract) $\mathbb Z/2$-gradings and $\mathbb Z$-coefficients {\sc (Thms.~\ref{thm:obs},~\ref{thm:invar})}.
	In particular, this establishes that given any real Heegaard diagram $(\scr H, \tau)$, these obstructions lie in $H^1$ and $H^2$ of the particular connected component of the pathspace $\rhatP(\scr H, \tau)$ that we are interested in. The connected components of the pathspace, at least when considered in the stable range {\sc(Prop.~\ref{prop:connsd})}, correspond canonically to relative real ${\rm Spin}^c$-structures on $(Y, \tau)$, cf. {\sc Thm.~\ref{thm:comp}}, and the Lagrangian Floer chain complex at any connected component computes the real Heegaard Floer chain complex at the corresponding real ${\rm Spin}^c$-structure.

	The topology of the connected components of the stable pathspace has been studied in {\sc Prop.~\ref{prop:spcalc}}, where we have concluded that they are homotopy-tori with $\pi_1 = A \overset{\rm def}= H^1(Y' \setminus N', \partial N', \underline{\bb Z}^v)$. Consequently, elements of $H^1(\scr P; \mathbb Z/2)$ are identified with $\mathbb Z/2$-valued functionals on $A$, and similarly elements of $H^2(\scr P; \mathbb Z/2)$ are identified with alternating bilinear forms on $A$ with values in $\mathbb Z/2$. The obstructions are given by $w_1$ and $w_2$ of the formal difference of the two Lagrangian tangent bundles, cf. {\sc Thm.~\ref{thm:obs}}; in our case, the torus is parallelizable, so the latter are the same as $w_1$ and $w_2$ of the tangent bundle to $\rsym^d\Sigma^\circ$, which we have computed in {\sc Thm.~\ref{thm:wi}} to be
	\[ w_1(x) = F(x, x), \qquad w_2(x, y) = F(x, y) - F(x,x) \cdot F(y, y) \overset{\rm def}= Q(x, y). \]
	This is precisely the claim of {\sc Thm.~\ref{thm:main}}, for a particular given real Heegaard diagram $(\scr H, \tau)$. It therefore remains to establish invariance of the $\mathbb Z/2$-gradings or $\mathbb Z$-lift with respect to real Heegaard moves. Invariance up to isotopy of the $\alpha, \beta$ curves has already been proven in {\sc Cor.~\ref{cor:isotinv}}, on the basis of {\sc Thm.~\ref{thm:invar}}. It remains to explain invariance up to real handleslides and real stabilizations.

	Handleslides can be treated either in the more classical Ozsv\'ath-Szab\'o \cite{OS-I} way via triangle-counting maps, or in the more modern \cite{Per08} way due to Perutz, of building a Hamiltonian isotopy between the Lagrangian tori before and after the handle-slide. We choose the latter, to conform with \cite{GM25}, though for the sake of future work on functoriality it would be good to also develop the classical technique. The result of \cite{Per08} states the a handle-slide induces a Hamiltonian isotopy of the Lagrangian torus, so the same result {\sc Thm.~\ref{thm:invar}} also gives us handle-slide invariance.

	It remains to establish invariance up to stabilization, which in the real case comes in the form of connect-summing with one of three distinct real diagrams for $S^3$, cf. {\sc Def.~\ref{def:realmoves}}. Since we are only doing the $\widehat{\it HFR}$ version, there is a one-to-one correspondence between generators and holomorphic disks from one diagram to the stabilized one, thanks to domains not being able to pass through the basepoints. That is, the chain complexes are honestly isomorphic over $\mathbb Z/2$, and it only remains to show that the underlying bijective correspondence can also be made to respect the choices of $\mathbb Z/2$-gradings or coherent orientations. The effect of stabilization on the linearized real CR operator is equivalent to direct-summing the unstabilized operator $D$ with another real CR operator $D^\perp$ on the orthogonal complement $(E^\perp, F^\perp)$ of the inclusion $(E, F) \subset (E', F')$ of the bundles before stabilization inside the ones after stabilization:
	\[ D' \cong D \oplus D^\perp. \]
	This orthogonal complement $(E^\perp, F^\perp)$ is trivial, and hence has zero index and trivial determinant line bundle. In other words, if we use $\delta_{p, q}, \Lambda_{p, q}$ to denote the index and determinant line before stabilization, and $\delta_{p,q}'$, $\Lambda_{p, q}'$ after stabilization, we have $\delta_{p, q}' = \delta_{p, q} + 0$ and $\Lambda_{p, q}' \cong \Lambda_{p, q} \otimes \underline{\mathbb R}$. This means that $\mathbb Z/2$-gradings and coherent orientations can be extended canonically via this isomorphism, hence concluding the proof of invariance up to real Heegaard moves.

	The resulting invariant is well-defined up to an ambiguity of the choice of $\mathbb Z/2$-grading or $\mathbb Z$-coefficients, and it is perfectly possible that a loop of Heegaard moves might not preserve this choice. This is the reason why we only state {\sc Thm.~\ref{thm:main}} as invariance of a \emph{relatively graded} object. In the last section, we offer some speculation as to why we expect this grading to be canonical, though we do not offer any rigorous proof of this additional conjecture.

\section{Algebraic properties of \texorpdfstring{$A$}{A}, \texorpdfstring{$F$}{F}, and \texorpdfstring{$Q$}{Q}}\label{sec:alg}

    In this section, we study some important algebraic properties of the group 
        \[A = H^1(Y' \setminus N', \partial N'; \underline{\bb Z}^v) \] 
    that appears in {\sc Thm.~\ref{thm:main}}, as well as the $\bb Z/2$-valued bilinear forms $F$ and $Q$ on this space.

    \begin{PROP}\label{prop:bigA}
        The following are equivalent descriptions of the group $A$ up to canonical isomorphism, where $\vec \mu$, $\vec \mu'$ denote the disjoint unions of meridional circles for the link $C$ in $Y$ and $Y'$, respectively:
        \begin{enumerate}
            \item $H^1(Y' \setminus N', \partial N'; \underline{\bb Z}^v) = H^1(Y' \setminus N', \vec \mu'; \underline{\bb Z}^v)$;
            \item $H^1(Y \setminus N, \partial N; \bb Z)^{-\tau^*} = H^1(Y \setminus N, \vec \mu; \bb Z)^{-\tau^*}$;
            \item $H^1(Y, C; \bb Z)^{-\tau^*} = H^1(Y, \vec w; \bb Z)^{-\tau^*}$; and in addition if $C$ is a knot, also $H^1(Y)^{-\tau^*}$.
        \end{enumerate}
        Furthermore, there are inclusions of finite index equal to a power of two
        \begin{enumerate}
            \item[\sc iv.] $A = H^1(Y' \setminus N', \vec \mu'; \underline{\bb Z}^v) \subset H^1(Y' \setminus N'; \underline{\bb Z}^v)$;
            \item[\sc v.] $A = H^1(Y \setminus N, \vec \mu; {\bb Z})^{-\tau^*} \subset H^1(Y \setminus N; \bb Z)^{-\tau^*}$;
            \item[\sc vi.] $A = H^1(Y, C; \bb Z)^{-\tau^*} = H^1(Y, \vec w; \bb Z)^{-\tau^*} \subset H^1(Y; \bb Z)^{-\tau^*}$; 
        \end{enumerate}
        which therefore preserve rank and become isomorphisms rationally. Rationalization can be obtained by replacing $\mathbb Z$ with $\mathbb Q$ in all coefficients that show up in the various cohomology groups.
    \end{PROP}

    \begin{proof}
        To establish {\sc i}, let us look at the associated long-exact sequence associated to the triple $(Y' \setminus N', \partial N', \vec \mu')$. It suffices to show that $H^0(\partial N', \vec \mu'; \underline{\bb Z}^v) = H^1(\partial N', \vec \mu'; \underline{\bb Z}^v) = 0$. This follows from the K\"unneth formula applied to the smash product of pairs with local systems:
        \[ (C, \vec w; \mathbb Z) \otimes (\vec \mu', \0; \underline{\bb Z}^v) \cong (\partial N', \vec \mu'; \underline{\bb Z}^v). \]
        Indeed, the cohomology of either factor is concentrated in degree 1, and moreover the cohomology of the former is torsion-free, so the naive K\"unneth formula (without Tor) applies, implying the desired conclusion.

        Part {\sc ii} follows from the more general observation that, for any pair $(X, A)$ with $X$ connected and $A \neq\0$, and every local system $\alpha : \pi_1 X \to \mathbb Z/2$, there is an isomorphism
        \[ H^1(\tilde X^\alpha, \tilde A^\alpha; \mathbb Z)^{-\tau^*} \cong H^1(X, A; \underline{\bb Z}^\alpha), \]
        where $\tilde X^\alpha, \tilde A^\alpha$ are the associated covers to $\alpha$. Indeed, by inspecting the Leray-Serre spectral sequence for the fiber sequence $(\tilde X^\alpha, \tilde A^\alpha; \mathbb Z) \to (X, A; \underline{\bb Z}^v) \to \mathbb{RP}^\infty$, we find that there is a spectral sequence
        \[ E_2^{p,q} = H^p(\mathbb{RP}^\infty; \underline{H}^q(\tilde X^\alpha, \tilde A^\alpha; \mathbb Z)) \]
        where $\underline H$ signifies the local system over $\mathbb{RP}^\infty$; when viewed as a $\mathbb Z/2$-representation, this is given by the involution $-\tau^*$, where the minus sign is due to the fact that the trivialization of the pull-back of $\underline{\bb Z}^\alpha$ to $\tilde X^\alpha$ \emph{anti-}commutes with the deck transformation $\tau$. The row $q = 0$ is entirely zero, thanks to the connectivity ans non-emptiness assumptions, and therefore the only value on the $p + q = 1$ diagonal is given by $E^{0,1}_2 = E^{0,1}_\infty$, which equals the fixed-points of $-\tau^*$ on $H^1(\tilde X^\alpha, \tilde A^\alpha; \underline{\bb Z}^\alpha)$.

        The isomorphism in {\sc iii} claiming $A \cong H^1(Y, C; \bb Z)^{-\tau^*}$ follows by filling in the solid tori through excision. To see the equality $H^1(Y, C; \bb Z)^{-\tau^*} = H^1(Y, \vec w; \bb Z)^{-\tau^*}$, we can explore the long-exact sequence
        \[ 0 = H^0(C, \vec w) \overset\delta\to H^1(Y, C) \to H^1(Y, \vec w) \to H^1(C, \vec w), \]
        together with the observation that $\tau^*$ acts trivially on the free abelian group $H^1(C, \vec w)$.
        
        Claim {\sc iv} follows from the long-exact sequence in cohomology associated to the pair $(Y' \setminus N', \vec \mu')$, and the observation that $H^0(\vec \mu'; \underline{\bb Z}^v) = 0$ and $H^1(\vec \mu'; \underline{\bb Z}^v) = (\mathbb Z/2)^{k}$; indeed, the co-chain complex for the cohomology of $S^1$ valued in the sign representation is given by $\mathbb Z \overset{\cdot 2}\to \mathbb Z$.

        For proving {\sc v}, let us investigate the long-exact sequence
        \[ H^0(\vec \mu; \bb Z) \overset\delta\to H^1(Y \setminus N, \vec \mu; \bb Z) \to H^1(Y \setminus N; \bb Z) \to H^1(\vec \mu; \bb Z). \]
        To see injectivity of the middle map on $(-\tau^*)$-fixed points, assume that an element $x$ of the second term satisfies $x = -\tau^* x$, and maps to zero in the third term. By exactness, there exists $y$ in the first term such that $\delta y = x$. But all elements in the first term are $\tau$-invariant, and since $\delta$ commutes with the $\tau^*$-action, it follows that the element $x$ in question is simultaneously $\tau^*$-invariant and $(-\tau)^*$-invariant. Due to the torsion-freeness of $H^1$, this forces $x$ to be zero, as desired. For the claim that the index of this inclusion is a power of two, let us assume that $x$ now is a $(-\tau^*)$-invariant element in the third term, which maps to zero in the fourth. By exactness, it is the image of an element $x$ in the second term, though this $x$ need not necessarily be $(-\tau)^*$-invariant. Nevertheless, the difference $x - \tau^*x$ is indeed $(-\tau)^*$-invariant, and its image is given by $y - \tau^* y = 2y$. This proves the claim.

        Lastly, claim {\sc vi} follows by the exact same formal argument applied to the long-exact sequence
        \[ H^0(C; \mathbb Z) \overset\delta\to H^1(Y, C; \mathbb Z) \to H^1(Y; \mathbb Z) \to H^1(C; \mathbb Z). \]
        The statement that rationalization of the groups in question amounts to replacing $\mathbb Z$ with $\mathbb Q$ in the coefficient systems comes from the fact that rationalization is an exact functor that commutes with taking invariants of finite group actions.
    \end{proof}

    \begin{COR}\label{cor:AisrPi}
        Given any real Heegaard diagram $(\scr H, \tau)$, the group $A$ of the associated real 3-manifold $(Y(\scr H), \tau)$ is naturally isomorphic to the group ${\it r}\hspace{-0.1em}\hat\varPi(\scr H, \tau)$ of real periodic domains.
    \end{COR}

    \begin{proof}
        It is standard that in classical Heegaard Floer theory, the group $H^1(Y, \vec w; \bb Z)$ is isomorphic to the group $\hat\varPi(\scr H)$ of periodic domains. The action of $-\tau^*$ on the former translates to the action that takes a domain and mirrors it with respect to the $\tau$-action; the minus sign comes from the orientation-reversing nature of $\tau$ on elementary domains.
    \end{proof}

    \begin{PROP}\label{prop:Apr1}
        The group $A = H^1(Y' \setminus N', \partial N'; \underline{\bb Z}^v)$ is a free abelian group, whose quotient $A/2A$ embeds into $H^1(Y' \setminus N', \partial N'; \bb F_2) \overset{\rm exc}= H^1(Y', C; \bb F_2)$. In fact, this embedding factors through the subgroup of elements $x$ satisfying the relation
        \[ x \smile x = x \smile v \]
        in $H^2(Y' \setminus N', \partial N'; \bb F_2)$.
    \end{PROP}

    \begin{COR}\label{cor:YpCF2}
        If $H^1(Y', C; \bb F_2) = 0$, or equivalently $H^1(Y'; \mathbb F_2) = 0$ and $C$ is a knot, then it follows that $A = 0$.
    \end{COR}

    \begin{COR}\label{cor:trivtauact}
        The group $A$ vanishes if and only if $\tau^*$ acts trivially on $H^1(Y; \bb Q)$.
    \end{COR}

    \begin{proof}[Proof of {\sc Cor.~\ref{cor:YpCF2}}.]
        If $H^1(Y, C; \bb F_2) = 0$, then in particular, $A/2A$ embeds inside the zero group, and hence $A$ itself must be zero (since $A$ is free abelian). It remains to establish the equivalence between this condition and the statement that $H^1(Y'; \mathbb F_2) = 0$ and $C$ is a knot. To this end, let us investigate the long-exact sequence 
        \[ H^0(Y'; \bb F_2) \to H^0(C; \bb F_2) \to H^1(Y', C; \bb F_2) \to H^1(Y'; \bb F_2) \to H^1(C; \bb F_2). \]
        For the $\Rightarrow$ implication, note that $H^0(Y'; \bb F_2) \to H^0(C; \bb F_2)$ must be a surjection, and therefore $C$ must be connected; in addition, it follows that $H^1(Y'; \bb F_2) \to H^1(C; \bb F_2)$ must be an injection, but since $C$ is nullhomologous in $Y'$ we also learn that this map must be zero, thus forcing $H^1(Y'; \bb F_2) = 0$. The converse $\Leftarrow$ implication follows once again because the first map in the sequence is an isomorphism, and the last map is zero.
    \end{proof}

    \begin{proof}[Proof of {\sc Cor.~\ref{cor:trivtauact}}]
        This follows from the fact that $A$ is free abelian, and the alternative description of $A \otimes \bb Q$ given in {\sc Prop.~\ref{prop:bigA}-vi}.
    \end{proof}

    Before proving {\sc Prop.~\ref{prop:Apr1}}, we must better understand the Bockstein homomorphism associated to $\underline{\bb Z}^v$ and $\underline{\bb Z/4}^v$, which will be the main tool in comparing the mod-2 reduction map in the Proposition.

    \begin{LEM}\label{lem:bock}
        The connecting homomorphism in both $H^*(-; \bb Z/2)$ and $H^*(-,-;\bb Z/2)$, associated to the short-exact sequence
        \[ 0 \to \bb Z/2 \to \underline{\bb Z/4}^v \to \bb Z/2 \to 0, \]
        of local coefficient systems, takes the form
        \begin{equation}\label{eq:twbock}
             x ~\longmapsto~ {\rm Sq}^1 (x) + v \smile x.
        \end{equation}
        In particular, the connecting homomorphism associated to the short-exact sequence
        \[ 0 \to \underline{\bb Z}^v \overset{\cdot 2}\to \underline{\bb Z}^v \to \bb Z/2 \to 0, \]
        after post-composition with mod-2 reduction also takes the same form (\ref{eq:twbock}).
    \end{LEM}

    \begin{proof}[Proof of Lemma.]
        It suffices to compute this in the universal case, but we must be precise about what this means. The space $K(\bb Z/2, k)$ classifies elements of $H^k(-;\bb F_2)$, but in our case we also need to keep track of the extra element $v$, and so our universal space will be $K(\bb F_2, k) \times \bb{RP}^\infty$, where the first factor keeps track of $x$, and the second one keeps track of $v$. The local coefficients systems $\underline{\bb Z}^v$ and $\underline{\bb Z/4}^v$ are implicitly understood to be pulled back from this universal $v$ (hence the elements $x$ and $v$ will now have no relation to the original 3-manifold, but rather live on the universal space $K(\bb F_2, k) \times \bb{RP}^\infty$).

        In this case, $H^{k+1}(K(\bb F_2, k) \times \bb{RP}^\infty)$ is given by the 3-dimensional $\bb F_2$-vector space generated by ${\rm Sq^1}(x)$, $v \smile x$, and $v^{k+1}$. We know that the Bockstein must be linear in $x$, so the coefficient of $v^{k+1}$ must be zero. Furthermore, in the classical case that $v = 0$, we must recover the usual Bockstein, which forces the coefficient of ${\rm Sq}^1(x)$ to be one. So it remains to distinguish between ${\rm Sq^1}(x)$ and ${\rm Sq}^1 + v \smile x$. Let us suppose towards a contradiction that it is the former, and analyze the situation in the case of $S^k \times S^1$, with $x$ being the pull-back of the fundamental class of $S^k$, and $v$ being the pull-back of the fundamental class of $S^1$. In this case, ${\rm Sq}^1(x)$ is zero, which under our assumption would imply that $x$ lifts to a class in $H^k(S^k \times S^1; \underline{\bb Z/4}^v)$. We show by direct computation that this is not true. By the K\"unneth formula, the cohomology of $S^k$ can be factored out, resulting in
        \[ H^*(S^k \times S^1; \bb Z \otimes \underline{M}^v) \cong \bb Z\<1, x\> \otimes H^*(S^1; \underline{M}^v), \]
        for any abelian group $M$. Based on the minimal cell structure on $S^1$, we find that $C^*(S^1; \underline{M}^v)$ is computed by the 2-step co-cochain complex $M \overset {\cdot 2}\to M$. We can see that the generator in degree zero for $M = \bb Z/4$ maps to zero into the analogous cohomology group for $M = \bb Z/2$, so in particular $x \otimes 1$ does not lie in the image of reduction mod 2 from $\underline{\bb Z/4}^v$, a contradiction.

        The claim about $\underline{\bb Z}^v$ follows in the same way as in the classical case, since it boils down merely to the naturality of the connecting homomorphism coming from a map of short-exact sequences.
    \end{proof}

    \begin{proof}[Proof of Proposition.]
        For the freeness claim, we note that it is equivalent to the claim that multiplication by any prime number $p$ is injective. We may prove this by analyzing the long-exact sequence associated to $\bb Z \overset{\cdot p}\to \bb Z \to \bb F_p$, and then showing that $H^0$ with $\underline{\bb F}_p^v$-coefficients is zero. Indeed, recall that $H^0$ with local coefficients is nothing but global sections of the sheaf representing the local system; at odd primes, the non-triviality of the sign monodromy makes it impossible to have a global section, while at $p = 2$ the monodromy is trivial, and so the cohomology group vanishes for connectedness reasons.
        
        For the next claim, we will be using this long-exact sequence at $p = 2$ in more depth:
        \[ \begin{gathered}
            H^1(Y' \setminus N', \partial N'; \underline{\bb Z}^v) \overset{2\cdot}\to H^1(Y' \setminus N', \partial N'; \underline{\bb Z}^v) \\ \to H^1(Y' \setminus N', \partial N'; \bb F_2) 
            \to H^2(Y' \setminus N', \partial N'; \underline{\bb Z}^v).
        \end{gathered} \]
        The cokernel of the first map is $A/2A$, so we arrive at
        \[ 0 \to A/2A \to H^1(Y' \setminus N', \partial N'; \bb F_2) \to H^2(Y' \setminus N', \partial N'; \underline{\bb Z}^v). \]
        This proves that $A/2A$ embeds in $H^1(Y' \setminus N', \partial N'; \bb F_2)$. For the more refined claim that this embedding factors through the elements $x$ satisfying $x \smile x = v \smile x$, note that the image of this $A/2A$ subgroup under the composite 
        \[ H^1(Y' \setminus N', \partial N'; \bb F_2) \overset\delta\to H^2(Y' \setminus N', \partial N'; \underline{\bb Z}^v) \overset{\text{mod }2}\longrightarrow H^2(Y' \setminus N', \partial N'; \bb F_2) \]
        must be zero. But this composite is precisely the Bockstein discussed in the Lemma, i.e.~ ${\rm Sq}^1(x) + v \smile x = x \smile x + v \smile x$.
    \end{proof}

    \begin{COR}\label{cor:F2F2summ}
        The form $F(x, y)$ may alternatively be calculated via the formula
        \begin{equation}\label{eq:altQ}
            \int_{Y' \setminus N'} x^2 \smile y = \int_{Y' \setminus N'} x \smile y^2 \quad\text{mod}~ 2.
        \end{equation}
        Consequently, $F$ descends to the image of $A$ under the map
        \begin{equation}\label{eq:pmap}
            A \overunderset{\text{\sc Prop}}{\text{\sc\ref{prop:Apr1}}}\longrightarrow H^1(Y', C; \bb F_2) \to H^1(Y'; \bb F_2) \to H^1(Y'; \bb F_2) / {\rm Ker({\rm Sq^1})}.
        \end{equation}
        In particular, $F \equiv 0$ if ${\rm Sq}^1 \equiv 0$, and $Q \equiv 0$ if ${\rm codim~ker} ({\rm Sq}^1 : H^1Y' \to H^2Y') \le 1$. Put differently, $H_1(Y'; \bb Z)$ must have at least a $\bb Z/2$ summand for $F$ to have a chance at being nonzero, and at least a $\bb Z/2 \oplus \bb Z/2$ summand for $Q$ to have a chance at being nonzero.
    \end{COR}

  \begin{proof}
        The first claim follows because $x \smile x = x \smile v$ mod 2, and likewise for $y$. Since evaluating the triple cup product on the top class in $H^1(Y', C; \bb F_2)$ versus $H^1(Y'; \bb F_2)$ is the same, it follows that $F$ factors through $H^1(Y'; \bb F_2)$. Since ${\rm Sq}^1$ is the same as squaring for a degree-1 class, $F$ must further descend all the way to the last term of (\ref{eq:pmap}).
        
        The claims in the last sentence of our statement follow by analyzing the contribution of cyclic summands to the Bockstein. Indeed, recall that any chain complex with finitely generated homology is quasi-isomorphic to a direct sum of either $\bb Z$ concentrated in a single degree, or shifts of complexes of the form $\bb Z \overset{p^r}\to \bb Z$. The former has trivial Bockstein, and so does the latter at odd primes. At $p = 2$, direct calculation shows that when $r \ge 2$, the $\bb Z/4$-Bockstein is zero, so the only possible contribution comes from the $r = 1$ summands, where the Bockstein is non-trivial.
    \end{proof}

    \begin{RMK}
        Somewhat remarkably, the formula (\ref{eq:altQ}) does not depend on $v$, although the domain of definition of $F$ and $Q$ of course does.
    \end{RMK}

    \begin{PROP}\label{prop:Apr2}
        If $Y' = Y_0 \# Y_1$ with separating sphere $S$, i.e.~ $Y = Y^\circ_0 \cup_S Y^\circ_1$, the link $C$ is contained entirely in $Y^\circ_1$, and the class $v \in H^1(Y' \setminus N'; \mathbb F_2)$ is supported in $Y^\circ_1$, i.e.~ $v$ is the image of $v_1 \in H^1(Y'_1 \setminus N'; \mathbb F_2)$ under extension by zero, then one has
        \[ A(Y', v) \cong H^1(Y_0; \mathbb Z) \oplus A(Y_1, v_1) \]
        and moreover
        \[ F_{Y', v} = 0 \oplus F_{Y_1, v_1}, \]
        as a block-sum decomposition of symmetric bilinear forms.
    \end{PROP}

    \begin{proof}
        The proof of the first statement concerning $A$ is proved in a manner quite analogous to the usual connect sum formula for ordinary cohomology, suitably adapted to the relative case. Indeed, let $\gamma$ be any embedded simple curve connecting $S$ to $\partial N'$ inside $Y_1^\circ$, and let us consider the long-exact sequence associated to the triple $(Y' \setminus N', S \cup \gamma \cup \partial N', \partial N')$:
        \[ 0 = \tilde H^0(S; \bb Z) \to H^1(Y_0; \bb Z) \oplus H^1(Y_1' \setminus N', \partial N'; \underline{\bb Z}^v) \to H^1(Y' \setminus N', \partial N'; \underline{\bb Z}^{v_1}) \overset0\to \tilde H^1(S; \bb Z) \]
        where we have used the following facts:
        \begin{list}{$\cdot$}{}
            \item that $(S \cup \gamma \cup \partial N', \partial N')$ is excisively equivalent to $(S \cup \gamma, \gamma) \overset{\rm h.e.}\simeq (S, *)$;
            \item that $(Y' \setminus N', S \cup \gamma \cup \partial N', \partial N')$ is excisively equivalent to the disjoint union of $(Y_0^\circ, S)$ and $(Y_1^\circ \setminus N', S \cup \gamma \cup \partial N')$, which can be further filled in with a ball $B$ to get rid of $S$, and then shrink back the appendage $B \cup \gamma$ in the second case;
            \item that $\underline{\bb Z}^v$ becomes trivial as a local system when restricted to $Y_0^\circ$;
            \item that the inclusion of $S$ into $Y' \setminus N'$ is null-homologous, thanks to bounding $Y_0^\circ$.
        \end{list}
        This concludes the first claim, concerning the additivity of $A$. As for the additivity of the bilinear form $F$, this follows easily from the fact that $v$ vanishes on the $Y_0^\circ$ part of the connect sum, and is equal to $v_1$ on the other part $Y_1^\circ$.
    \end{proof}

    \begin{COR}\label{cor:loclink}
        If $C$ is a local link in $Y'$, or more generally contained in an $\mathbb F_2$-homology ball, then $F$ vanishes identically.
    \end{COR}

    \begin{proof}
        This follows from the vanishing of $F$ for an $\mathbb F_2$-homology sphere, by {\sc Cor.~\ref{cor:F2F2summ}}.
    \end{proof}

    Here is a variation of {\sc Cor.~\ref{cor:F2F2summ}}, this time concerning the complement of the knot. It does not seem to obviously imply or be implied by the previous version of the corollary, due to the map $H_1(Y' \setminus N') \to H_1(Y')$ potentially having arbitrarily complicated kernel.

    \begin{PROP}\label{prop:2torcompl}
        In order for the form $F$ to be nonzero, the group
        \[ H^2(Y' \setminus N', \partial N'; \bb Z) \cong H_1(Y' \setminus N'; \bb Z) \]
        must have at least one $\bb Z/2$ summand. Likewise, for the form $Q$ to be nonzero, it must have at least a $\bb Z/2 \oplus \mathbb Z/2$ summand.
    \end{PROP}

    \begin{proof}
        We note that $v \in H^1(Y' \setminus N'; \bb Z/2)$ admits a lift $\tilde v \in H^1(Y' \setminus N'; \underline{\bb Z}^v)$, defined as the twisted Bockstein of {\sc Lem.~\ref{lem:bock}} applied to $1 \in H^0(Y' \setminus N'; \bb Z/2)$. We have that $x \smile y \smile \tilde v$ in $H^3(Y' \setminus N'; \underline{\bb Z}^v) \cong \bb Z/2$ lifts $x \smile y \smile v$ in $H^3(Y' \setminus N'; \bb F_2) \cong \bb Z/2$, and so in particular if $F$ does not vanish, then there exist $x, y \in A$ such that $x \smile y \smile \tilde v$ is nonzero. But note that $\tilde v$ is a 2-torsion element already, which means that $y \smile \tilde v \in H^2(Y' \setminus N', \partial N'; \bb Z)$ is a nonzero 2-torsion element, hence implying that the group in question has a $\mathbb Z/2^r$ summand. To see that $r$ must be 1, note that this 2-torsion element cannot be 2-divisible, since this would force $x \smile y \smile v \in \bb F_2$ to be 2-divisible, i.e.~ zero. This forces $H_1(Y' \setminus N'; \bb Z)$ to have a $\mathbb Z/2$-summand: indeed, in general if $G$ is an abelian group with a 2-torsion and non-2-divisible element $x$, then the inclusion $\mathbb Z/2 \cong \<x\> \subset G$ can be split by choosing a functional on the $\mathbb F_2$-vector space $G/2G$ sending the nonzero class $[x]$ to 1, and pre-composing with the projection $G \twoheadrightarrow G/2G$. This concludes our claim about $F$.

        For $Q$ not to vanish, it follows from {\sc Prop.~\ref{prop:checkQ}} below that there must exist elements $x, y \in A$ such that $F(x, x) = 0$, but $F(x, y) = 1$. In the same language we used in the previous paragraph, this means that
        \[ x \smile x \smile \tilde v = 0, \quad x \smile y \smile \tilde v = 1. \]
        From the latter we deduce that $x \smile \tilde v$ and $y \smile \tilde v$ are both 2-torsion and non-2-divisible elements. However, in conjunction with the former, this implies that the difference $x \smile \tilde v - y \smile \tilde v$ is non-2-divisible. As a result, the inclusion of the subgroup $\mathbb Z/2 \oplus \mathbb Z/2 \cong \<x, y\> \subset H_1(Y' \setminus N'; \bb Z)$ can be split by a similar trick to the one before, since the classes $[x], [y]$ are linearly independent in the vector space $G / 2G$.
    \end{proof}

    There is an interesting equivalent formulation for the diagonal terms $F(x, x)$ of the bilinear form, which show up as the obstruction to $\mathbb Z/2$-gradings:

    \begin{PROP}\label{prop:obsc1} (cf. \cite[Prop.~3.4]{Nak13-II})
        For any real ${\rm Spin}^c$-structure $\mathfrak s$ on $(Y,\tau)$, we have the identity $v^2 \equiv c_1^R(\mathfrak s)$ mod 2, where $c_1^R(\mathfrak s) \in H^1(Y' \setminus N'; \underline{\bb Z}^v)$ was defined in {\sc Def.~\ref{def:rc1}}. Thus, the condition for $\mathbb Z/2$-gradings can be restated in the following two ways:
        \begin{enumerate}
            \item $x \smile c_1^R(\mathfrak s) \in H^3(Y' \setminus N', \partial N'; \bb Z) \cong \bb Z$ is even, $\forall x \in A(Y, \tau) = H^1(Y' \setminus N', \partial N'; \underline{\bb Z}^v)$;
            \item $x \smile c_1(\mathfrak s) \in H^3(Y, C; \bb Z) \cong \bb Z$ is divisible by 4, $\forall x \in A(Y, \tau) = H^1(Y, C; \bb Z)^{-\tau^*}$.
        \end{enumerate}
    \end{PROP}

    \begin{proof}
        Let us represent the real ${\rm Spin}^c$ structure $\mathfrak s$ by a real vector field $X$ on $Y$. This induces a $\tau$-equivariant splitting of $TY|_{Y \setminus N}$ into $\<X\> \oplus L$, where $L$ is the $SO_2 \cong U_1$-bundle obtained as the orthogonal complement of $\< X \>$; this is an Atiyah-real line bundle. On the quotient $Y' \setminus N'$, the splitting descends to $TY' |_{Y' \setminus N'}$ into an $O_1$-bundle $\underline{\mathbb R}^v$ with monodromy $v$, and an $O_2$-bundle $L' := L / \tau$. The first Chern class of the real ${\rm Spin}^c$-structure is $c_1^R(L) \in H^2(Y' \setminus N'; \underline{\mathbb Z}^v)$, which restricts to $w_2(L')$ mod 2, as shown in {\sc Prop.~\ref{prop:rc1}}. Since $TY'$ is trivial, we have
        \[ 0 = w_2(\underline{\mathbb R}^v \oplus L') = w_2(\underline{\mathbb R}^v) + w_1(\underline{\mathbb R}^v) \cdot w_1(L') + w_2(L'). \]
        The first term is zero for degree reasons, and the second term is $v \cdot v$, since both bundles become canonically oriented when pulling back along the covering map, with the deck involution swapping the two orientations. This concludes the proof that $v^2 = c_1^R$.
        
        The reinterpretation {\sc i} of the obstruction to $\mathbb Z/2$-gradings in {\sc Main~Thm.~\ref{thm:main}} follows from the relation $x^2 \smile v = x \smile v^2$, by {\sc Prop.~\ref{prop:Apr1}-ii}. For reformulation {\sc ii}, we use the fact that pull-back along the projection map $\pi : Y \setminus N \twoheadrightarrow Y' \setminus N'$ induces an isomorphism
        \[ H^1(Y' \setminus N', \partial N'; \underline{\bb Z}^v) \overset\sim\longrightarrow H^1(Y \setminus N, \partial N; \bb Z)^{-\tau^*} \cong H^1(Y, C; \bb Z)^{-\tau^*}, \]
        as proved in {\sc Prop.~\ref{prop:bigA}}. Since ${\rm deg}(\pi) = 2$, we deduce that
        \[ 2 \cdot \int_{Y'} x \smile c_1^R(\fk s) = \int_Y \pi^*(x \smile c_1^R(\fk s)) = \int_Y \pi^*x \smile c_1(\fk s) \] 
        by {\sc Prop.~\ref{prop:rc1}-ii}, as desired.
    \end{proof}

    \begin{RMK}\label{rmk:conjzgr}
        This in particular tells us that $v^2$ has a lift to $\underline{\mathbb Z}^{v}$-coefficients, which is canonical up to the choice of real ${\rm Spin}^c$-structure $\mathfrak s$. As such, this choice of $\mathfrak s$ also determines a canonical lift of the functional $x \smile v^2$ to a $\mathbb Z$-valued functional. It is therefore natural to conjecture that $\widehat{\it HFR}$ is relatively graded modulo the image of this functional. Private communication from Jiakai Li suggests that this conjecture holds true in the monopole setting; cf. also \cite[\S2.5]{GM25}.
    \end{RMK}

    We end with a proposition that elucidates the precise relationship between the vanishing of $Q$ and that of $F$. Of course, vanishing of $F$ implies vanishing of $Q$, but one would like to know if there is a smaller subset $S$ of pairs $(x, y)$, such that $F|_S \equiv 0$ if and only if $Q \equiv 0$. We answer this in the affirmative:
    \begin{PROP}\label{prop:checkQ}
        Let $F(-,-)$ be a symmetric bilinear form on a finite-dimensional $\bb F_2$-vector space $V$. If the following implication holds
        \[ F(x, x) = 0 \quad \implies \quad F(x, y) = 0,\]
        then in fact the alternating bilinear form
        \[ Q(x, y) := F(x, y) - F(x,x) \cdot F(y,y) \]
        is identically zero.
    \end{PROP}

    \begin{proof}
        The hypothesis already implies that $Q(x, y) = 0$ when $F(x, x)$ or $F(y,y)$ vanishes. So let us assume that we are in the remaining case $F(x, x) = F(y, y) = 1$, and let us show that $F(x, y) = 1$. Indeed,
        \[ F(x, y) = F(x, x) + F(x, x + y) = 1 + F(x, x + y),\]
        so it suffices to show that $F(x, x + y) = 0$. But this follows from the hypothesis, together with the observation that $F(q, q)$ is linear in $q$ in characteristic 2, and hence
        \[ F(x + y, x + y) = F(x, x) + F(y, y) = 1 + 1 = 0. \qedhere \]
    \end{proof}

\section{Examples}\label{sec:ex}

	In this section, we provide various examples illustrating either the vanishing or non-vanishing of the obstructions in {\sc Thm.~\ref{thm:main}}, as well as concrete calculations and diagram-level interpretations in certain cases.

	\subsection{Different choices leading to different invariants}\label{ssec:s1s2} There is a very straightforward case in which both obstructions vanish, and in which different choices of $\bb Z/2$-gradings and $\bb Z$-coefficients lead to non-isomorphic invariants. Consider the real Heegaard diagram in {\sc Fig.~\ref{fig:ds1s2}}, where $Y = (S^1 \times S^2)^{\# 2}$ and $Y' = S^1 \times S^2$ with $C \subset Y'$ the unknot and $v$ the Poincar\'e dual of the disk bounding $C$. This falls under the conditions of {\sc Cor.~\ref{cor:loclink}} of a local unknot, and therefore $A = H^1(S^1 \times S^2; \bb Z) = \bb Z$, and the obstructions vanish.

	\begin{figure}
		\centering\includegraphics[scale=1.5]{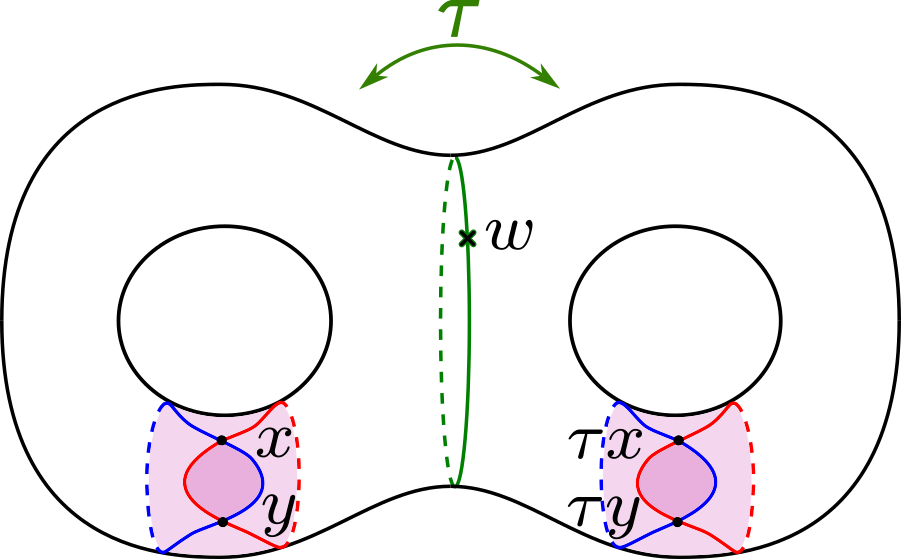}
		\caption{A weakly admissible real Heegaard diagram for $Y' = S^1 \times S^2$, $C$ the unknot, and $v$ the Poincar\'e dual of the bounding disk. The real involution $\tau$ is given by reflection along the vertical plane}\label{fig:ds1s2}
	\end{figure}

	We arrange that the $\alpha$-curves (red) and $\beta$-curves (blue) intersect as shown in {\sc Fig.~\ref{fig:ds1s2}}, so that the diagram becomes admissible. Indeed, in this case the group $A$ of real periodic domains is rank-1, freely generated by the domain highlighted in pink, where the forward-facing elementary domains have multiplicity $+1$, and the backward-facing ones have multiplicity $-1$. The presence of the basepoint makes it so that every holomorphic strip must use the aforementioned four elementary domains exclusively, and the $\tau$-invariance means that the coefficients must agree on the left and on the right. From here, it is straightforward to see that there are no domains with wholly non-negative coefficients from $(y, \tau y)$ to $(x, \tau x)$, and that the only such domains from $(x, \tau x)$ to $(y, \tau y)$ are given by forward-facing, and respectively backward-facing, elementary domains. By the Riemann Mapping Theorem, each such domain comes equipped with precisely one holomorphic representative, up to translation.

	We conclude that the chain complex in the unique real ${\rm Spin}^c$-structure containing $(x, \tau x)$ and $(y, \tau y)$ is in fact relatively graded over $\mathbb Z$ (and not just $\mathbb Z/2$), and takes the form
	\[ \mathbb Z \overset{\pm 1\pm 1}\longrightarrow \mathbb Z \]
	where the signs are determined by the orientation on the 0-dimensional moduli space. We know that changing the choice of coherent orientation by a 1-cocycle in $\varphi \in H^1(\scr P; \pm 1) = A^\vee \cong \mathbb Z/2$ changes the sign assignment of every holomorphic strip by $\pm 1$ according solely to its homotopy class in the space of Whitney disks. The group $A$ is precisely the group of real periodic domains, so that changing the orientations by the unique non-trivial cocycle in $A^\vee$ will result in multiplying the two domains from $x$ to $y$ by different elements of $\{\pm 1\}$. That is, if the signs were originally agreeing, the subsequent signs differ, and vice versa. This means that the two chain complexes corresponding to the two choices of coherent orientations are, up to isomorphism, given by
	\begin{equation}\label{eq:twochoices}
		\mathbb Z \overset{\cdot 2}\longrightarrow \mathbb Z, \quad \text{and} \quad \mathbb Z \overset{0}\longrightarrow \mathbb Z.
	\end{equation}
	They visibly have distinct chain homotopy types, which means in particular that the sign choice really affects the invariant. Also, it follows that in this particular case the conjectural set $\scr O$ of choices of $\mathbb Z$-lifts (cf. {\sc Rmk.~\ref{rmk:fkO}}) is an actual invariant of $(Y, \tau)$. Furthermore, since the choice of $\mathbb Z/2$-grading is independent of the choice of $\mathbb Z$-lift, and since in one of the two cases of $\mathbb Z$-coefficient choices, the groups $\widehat{\it HFR}_0$ and $\widehat{\it HFR}_1$ are distinct, it follows therefore that the set of choices of $\mathbb Z/2$-gradings is also a well-defined invariant of $(Y, \tau)$ (though we already knew that from \cite{Sr26}). We may decree that the privileged choice of $\mathbb Z/2$-grading is the one for which $\widehat{\it HFR}_0 = \mathbb Z/2$ for the first choice in (\ref{eq:twochoices}).

	We further note that this computation also handles the case of the other possible class $v \in H^1(Y' \setminus C; \bb F_2)$. Indeed, if we think of the unknot $C$ as being embedded in $\{1\} \times S^2 \subset S^1 \times S^2 = Y'$, it bounds two different disks, and in fact there is an isotopy (not respecting $C$) that takes one disk to the other. By invariance of real Heegaard-Floer homology, this shows that both cases give rise to isomorphic $\widehat{\it HFR}$.

	\subsection{Non-vanishing obstruction for \texorpdfstring{$\mathbb Z/2$-gradings}{Z/2-gradings}}\label{ssec:rp3} We know that this obstruction is given by the diagonal terms
	\[ F(x, x) = \int x^2 \smile v = \int x^3 \in \mathbb Z/2.  \]
	Naturally, the simplest case to study would be $Y' = \mathbb{RP}^3$, due to the necessity of having a non-trivial third cup-power mod 2. Let us take $C$ be the local unknot as before; then, there are precisely two choices of element $v \in H^1(Y \setminus C; \mathbb F_2)$, namely the Poincar\'e dual of the disk $D$ bounding it, and the Poincar\'e dual of the M\"obius band $M$ bounding it (such that the union $M \cup D$ of the two forms the standardly embedded $\mathbb{RP}^2 \subset \mathbb{RP}^3$.) The former leads, by virtue of {\sc Cor.~\ref{cor:loclink}}, to $A = H^1(\mathbb{RP}^3; \bb Z) = 0$, so it is not a good candidate. This leaves us with the option that $[\Sigma'] = [M]$.

	\begin{figure}
		\includegraphics[scale=1.5]{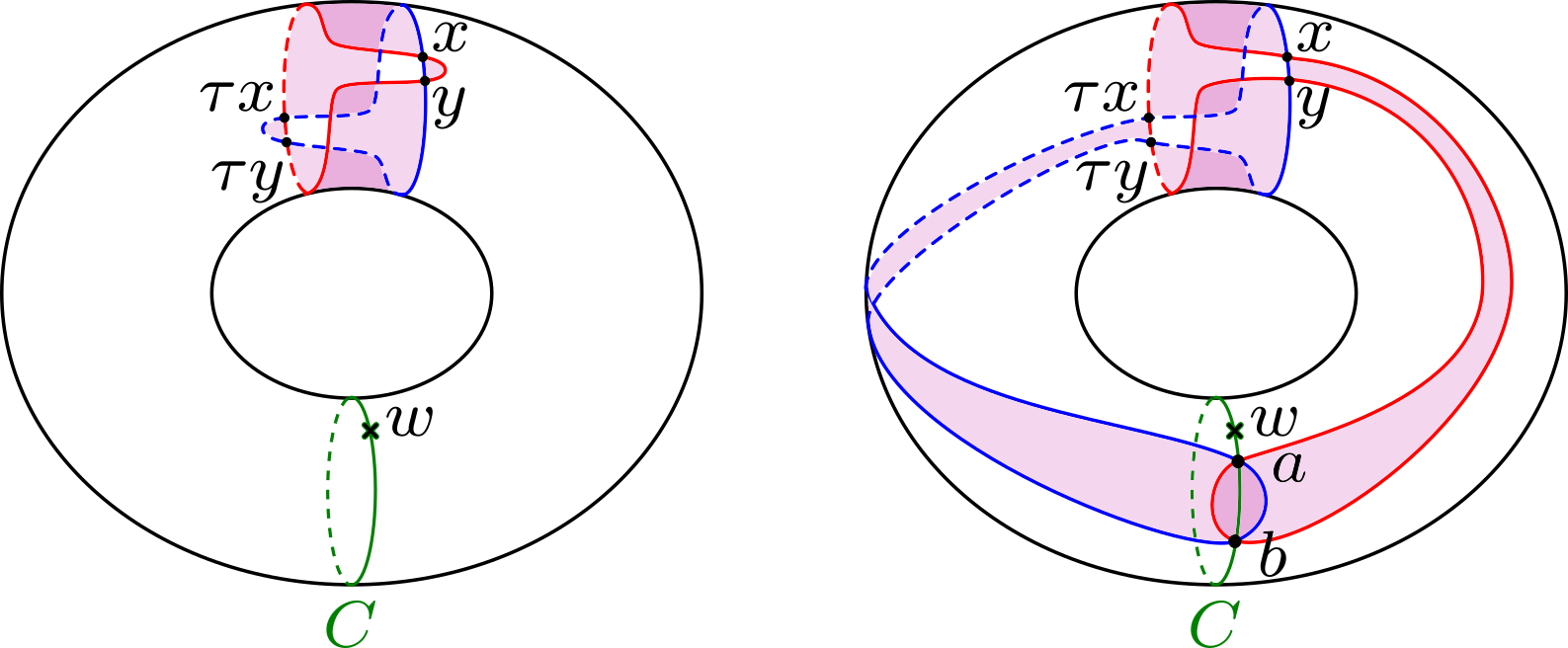}
		\caption{A weakly admissible real Heegaard diagram on the left, for $Y' = \mathbb{RP}^3$, $C$ the unknot, and $v$ the Poincar\'e dual class of the M\"obius band. The involution looks like reflection near the bottom green circle $C$, but looks like 3-dimensional reflection near a point on the opposite side. The finger moves are continued on the right, to have some actual generators for $\widehat{\it CFR}$.}\label{fig:rp3}
	\end{figure}

	Indeed, this M\"obius band already has the property that its complement is a handlebody, since the complement of $\mathbb{RP}^2$ in $\mathbb{RP}^3$ is a ball, and the result of filling in $D \subset \mathbb{RP}^2$ is the same as attaching a 1-handle. In other words, we obtain a real Heegaard diagram with $\Sigma' = M$, and $\Sigma$ a torus with $\alpha = \beta = C$, since the knot $C$ itself bounds a disk in the complement of $\Sigma'$. The left side of {\sc Fig.~\ref{fig:rp3}} depicts one example of a weakly admissible real Heegaard diagram for this branched double cover. A rigorous way to define the involution on the torus is to think of the 3-dimensional reflection across the center of a symmetric standard cylinder in $\mathbb R^3$, descended to the quotient identifying the opposite ends of the cylinder.

	The finger moves depicted in {\sc Fig.~\ref{fig:rp3}} are once again essential for ensuring weak admissibility, and indeed the group $A$ of periodic domains is cyclic generated by the domain highlighted in pink, with $-1$ multiplicities on the big elementary domains, and $+1$ on the small ones. There are 4 intersection points, but no intersection point that is $\tau$-invariant. As a result, the chain complex is zero on the nose, resulting in $\widehat{\it HFR} \equiv 0$. One can artificially continue the finger moves, cf. right picture of {\sc Fig.~\ref{fig:rp3}} in order that the red and blue circles intersect the green $C$ below in two new points $a$ and $b$. In that case, our main result would predict that the aforementioned periodic domain $D$ would have real Maslov index $F(x, x) = 1$ mod 2, since the map $A \overunderset{\text{\sc Prop.}}{\text{\sc\ref{prop:Apr1}}}\longrightarrow H^1(Y', C; \bb F_2) \cong \bb F_2$ is surjective, and the generator $x$ of the latter satisfies $\int x^3 = 1$; this is indeed verified by Guth-Manolescu's formula \cite[{\sc Cor.~4.3}]{GM25}:
	\[ \tfrac 12 (n_{a}(D) + n_b(D) + e(D)) = \tfrac 12 (1 + 1 + 0) = 1. \]
	That is to say, if we add $k \cdot D$ to the obvious bigon from $a$ to $b$, we get a domain of real Maslov index $k + 1$, so in particular there is no $\mathbb Z/2$-grading due to the possibility of choosing different $k$, as expected. Of course, these other domains from $a$ to $b$ do not admit any actual holomorphic representatives, but our definition of abstract $\mathbb Z/2$-gradings forces us to take them into account, cf. {\sc Def.~\ref{def:z2gr}}.

	\begin{RMK}
		At present, we have not been able to find an example where $F(x, x) \neq 0$ for some $x$, but $\widehat{\it HFR} \neq 0$. One might think of trying more complicated 3-manifolds having elements $x \in H^1(Y'; \bb Z/2)$ satisfying $x^3 \neq 0$, but as we now briefly show, they would likely be somewhat complicated. Such an $x$ would be represented by an embedded closed surface $F$, which must necessarily be non-orientable, and in fact with odd $b_1$. Indeed, the triple intersection number can be computed locally near the surface, and ends up being given by $\int_{Y'} x^3 = \int_{F} x^2 = \int_F w_1^2$, and the latter is a characteristic number, hence additive under connect sum, which easily shows that it must equal $b_1$ mod 2. The case $b_1 = 1$ corresponds to $F = \mathbb{RP}^2$, in which case the boundary of a tubular neighborhood thereof must necessarily be $S^2$, resulting in $Y' = Y'' \# \mathbb{RP}^3$. Since a K\"unneth formula for connect sums likely holds for $\widehat{\it HFR}$, we do not expect such cases to be interesting either. Thus, one must try $b_1 \ge 3$, so that there is an embedded incompressible $(\mathbb{RP}^2)^{\#(2n+1)}$ inside $Y'$. Furthermore, it is unlikely that $C$ equals the unknot and $v = (x, \mu)$ would work, since in this case it seems that one should be able to produce a real Heegaard diagram with odd genus but with no fixed intersection points, hence having no $\widehat{\it CFR}$ generators.
	\end{RMK}

	\subsection{An example where \texorpdfstring{$v$}{v} lifts to \texorpdfstring{$\mathbb Z/4$}{Z/4}\label{ssec:z4} but not to \texorpdfstring{$\mathbb Z$}{Z}}
	\begin{figure}
		\centering\includegraphics[scale=1.5]{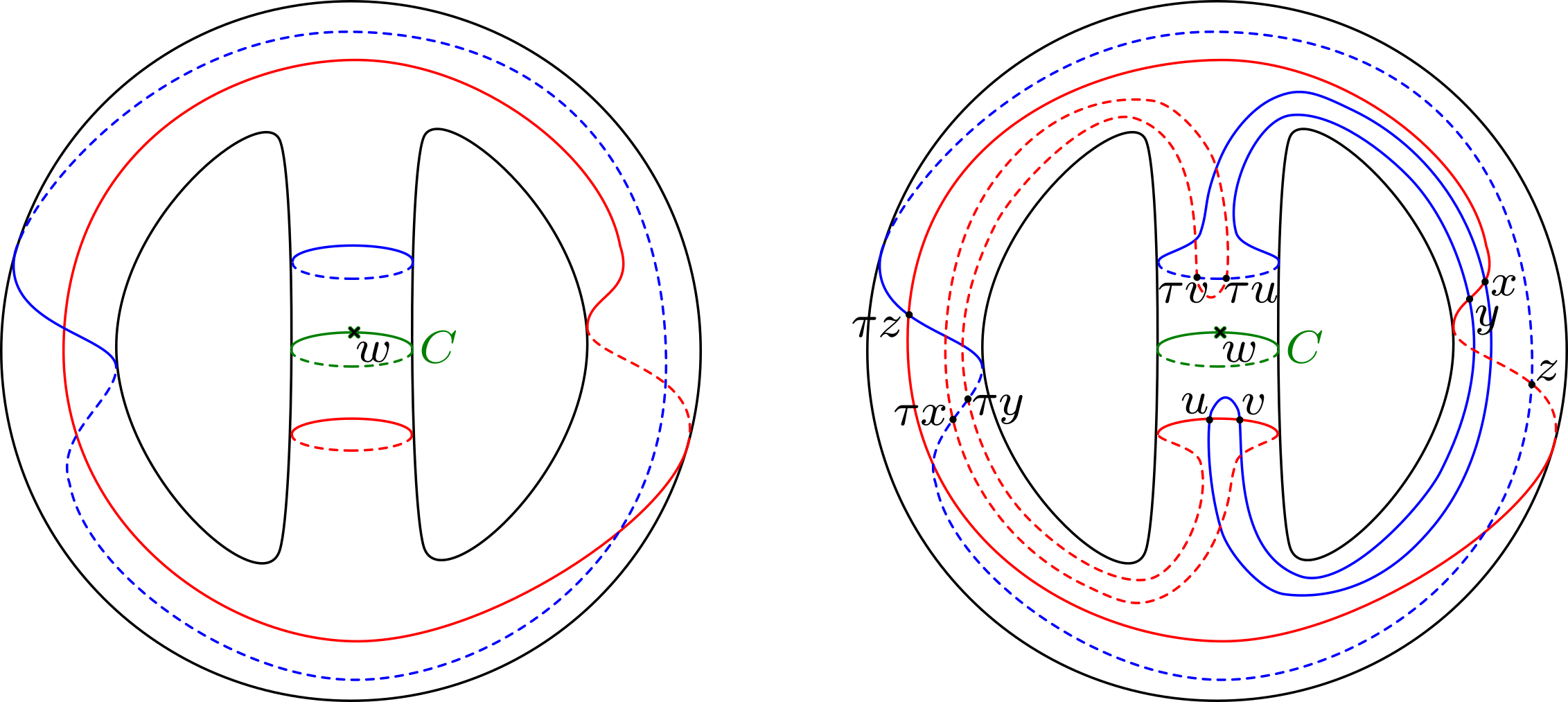}
		\caption{Real Heegaard diagrams for a $\bb Z/4$-lens space, $C$ the unknot, and $v$ nontrivial; the diagram on the right is obtained form the one on the left by doing a finger move introducing new intersection points $x, y, u, v$ and their images under $\tau$, which is necessary in order to achieve weak admissibility. The $\tau$-action away from the central tube is given by 3-dimensional reflection around the center of mass, but on the central tube it is warped in a Dehn twist fashion so that the central green circle $C$ is fixed by the action.}\label{fig:l41}
	\end{figure}
	We know from \cite{Sr26} that \emph{absolute} $\mathbb Z/2$-gradings exist provided that one lifts $v \in H^1(Y' \setminus C; \bb F_2)$ to a primitive element $V \in H(Y' \setminus C; \bb Z)$. This agrees with our weaker result. However, it is possible that $v$ lifts to $\mathbb Z/4$ and not to $\mathbb Z$, and our main theorem still implies that $\mathbb Z/2$-gradings do exist, though they need not be absolute. We now show a concrete example in which this occurs, and moreover in which $\widehat{HFR}$ is nonzero.

	The most natural place to look is the $\bb Z/4$-lens space $Y'$, since its generator $x \in H^1(Y; \mathbb F_2) \cong \bb F_2$ satisfies ${\rm Sq}^1 x = 0$, i.e.~ $x$ lifts to $\mathbb Z/4$, though it does not lift to $\mathbb Z$. Let us now try to express this class $x$ as the Poincar\'e dual of an embedded closed surface $F$, which would necessarily have to be non-orientable, since otherwise it would lift to $\mathbb Z$. As mentioned in the previous subsection, $x^3 \equiv b_1(F) ~\text{mod}~ 2$, so we would need $b_1$ to be even for ${\rm Sq}^1 x = x^2 = 0$.
	
	The smallest such surface to look for is a Klein bottle $K$; recall that
	\[ \pi_1 K = \<a, b | a b a^{-1} = b^{-1} \>, \]
	where we may think of $b$ as the homotopy class of the meridional fiber of the Klein bottle as fibered over $S^1$, and $a$ as the longitude. Thickening $K$ to an oriented 3-manifold with boundary $\mathbb T^2 \overset{2:1}\longrightarrow K$, we wish to close it off to obtain a $\bb Z/4$-lens space. Let us perform a Dehn filling by killing $a^2b$ inside
	\[ \pi_1 \bb T^2 = \<a^2, b | a^2 b a^{-2} = b\>, \]
	which produces a 3-manifold $Y'$ with
	\[ \pi_1 Y' = \<a, b | a b a^{-1} = b^{-1}, a^2 = b^{-1}\> = \mathbb Z/4. \]
	This means that we have an embedded Klein bottle $K$ in a $\bb Z/4$-lens space with complement given by a solid torus. If we remove a small disk $D$ from $K$, we obtain a surface $\Sigma'$ whose complement is a genus-2 handlebody, which can be thought of as attaching a 1-handle to $\mathbb T^2$ symmetrically with respect to the deck transformation on $\mathbb T^2$.

	{\sc Fig.~\ref{fig:l41}} shows the resulting real Heegaard diagram on the left (where the central tube corresponds to the aforementioned 1-handle), and the necessary modification on the right in order to achieve admissibility, via a symmetric finger move. The intersection points $z, u, v$ and their images under $\tau$ cannot contribute to a $\tau$-invariant intersection point, since that would force one $\alpha$ or $\beta$ circle to have two intersection points on it. The only remaining intersections are the mixed ones, i.e.~ $x, y$ and their images $\tau x, \tau y$, resulting in precisely two generators $(x, \tau x)$ and $(y, \tau y)$ for $\widehat{\it CFR}$.

	We claim that there are only two possible positive domains connecting $(x, \tau x)$ to $(y, \tau y)$: the most obvious one given by a free disjoint union of two bigons, obtained by following $x, y$ in the direction of the finger move until the end, and likewise for the image under $\tau$. The less obvious one is a genus 1 domain defined by the property that all elementary domains are given multiplicity 1, except the one containing $w$, and the tiny bigons connecting $u$ to $v$ and $\tau u$ to $\tau v$, which are assigned multiplicity zero.
	
	The difference between these two domains from $(x, \tau x)$ to $(y, \tau y)$ consists of the generator of the group of periodic domains (depicted in the leftmost picture of {\sc Fig.~\ref{fig:doms}}), and has multiplicities given by 1 everywhere except for: the elementary domain containing $w$, with multiplicity zero; the elementary quadrilateral with corners $x, y, v, u$ and its image under $\tau$, with multiplicity zero; and the small bigons connecting $u$ to $v$ and $\tau u$ to $\tau v$, with multiplicity $-1$. This periodic domain is $\tau$-invariant, and it also freely generates the group of real periodic domains, isomorphic to $A \cong \mathbb Z$. From here it is not difficult to see that any other domain connecting $(x, \tau x)$ to $(y, \tau y)$ must have some negative multiplicity, and hence cannot contribute to the differential.
	
	Both aforementioned domains are free, so by \cite{GM25}'s results, their real indices are equal to half their classical index. Direct computation shows that the classical index is $2$, and therefore the real index is 1, in both cases. From here, we conclude that the chain complex is in fact relatively graded over $\mathbb Z$, concentrated in two adjacent degrees. The free disjoint union of two bigons mentioned before contributes precisely one holomorphic strip, so it remains to investigate the remaining one. To this end, we use the following trick:

	\begin{PROP}\label{prop:trick}
		Given a real domain $D_{p,q} \in \pi^R_2(p, q)$ of classical index $2$ and real index $1$, its associated count $\# \scr M^R (D_{p,q})$ of real holomorphic representatives is equal mod 2 to the half-count of broken trajectories
		\begin{equation}\label{eq:broken}
			\frac 12 \sum_{D_{p,r} \star D_{r, q} = D_{p, q}} \# \scr M(D_{p, r}) \cdot \# \scr M(D_{r,q}),
		\end{equation}
		where $r$ ranges through not necessarily real generators, and $D_{p, r}$ and $D_{r, q}$ range through not necessarily real domains, of classical index 1.
	\end{PROP}

	\begin{proof}
		Since the classical index of $D_{p,q}$ is one more than its real index, it follows by transversality that every real holomorphic solution has a $\tau$-symmetric 1-parameter extension of non-real holomorphic solutions, i.e.~ $\scr M^R(D_{p, q})$ is the transversely cut out fixed-point locus of the involution $\tau$ on the compact 1-manifold $\overline{\scr M}(D_{p,q})$. By applying the Lefschetz fixed point formula mod 2, it follows that the count $\# \scr M^R(D_{p,q})$ is equal to the sum of traces of $\tau_*$ on $H_0$ and $H_1$ with $\mathbb Z/2$-coefficients. The induced maps are always permutations of the generators, and the trace is equal to the number of fixed generators, which mod 2 is the same as the full number of generators. The circle components of $\overline{\scr M}(D_{p,q})$ contribute two generators each, so they do not contribute anything mod 2. Consequently, the count $\#\scr M^R(D_{p,q})$ is equal mod 2 to the number of interval components of $\overline{\scr M}(D_{p,q})$, which is precisely (\ref{eq:broken}).
	\end{proof}

	Therefore, in order to understand the differential in $\widehat{\it CFR}$ from $(x, \tau x)$ to $(y, \tau y)$, we must understand the classical $\widehat{\it CF}$ of the given Heegaard diagram. Let us begin by noting that this diagram represents $Y = \mathbb{RP}^3 \# (S^1 \times S^2)$, so we expect $\widehat{\it CF}$ to split up into two non-trivial ${\rm Spin}^c$-structures, each being chain homotopy-equivalent to $\mathbb F_2 \overset{0}\to \mathbb F_2$.

	There are 12 generators, cf. {\sc Fig.~\ref{fig:gens}}, which break up into two chain complexes with 4 and 8 generators, respectively, depending on which ${\rm Spin}^c$-structure they belong to. We now proceed with the explanation of the differentials; thanks to the $\tau$-symmetry, reflected in {\sc Fig.~\ref{fig:gens}} by flipping across the central vertical axis, it suffices to treat only one representative form each orbit. Regarding the smaller chain complex on the left, we have:
	\begin{list}{$\cdot$}{}
		\item $(u, \tau z) \to (v, \tau z)$, and its reflection, shown vertically: there is an obvious tiny bigon which contributes 1 mod 2. Adding any non-zero multiple of the generating periodic domain leads to negative multiplicities somewhere, hence there are no more contributions.
		\item $(u, \tau z) \to (\tau v, z)$, and its reflection, shown diagonally: they must be the same value $a \in \bb F_2$. If $a = 0$, then the complex is acyclic, which contradicts the expectation that it recover $\bb F_2 \overset 0 \to \bb F_2$ up to chain-homotopy equivalence. Thus, we must have $a = 1$, which indeed matches the expectation, since the matrix $\left[\begin{smallmatrix} 1 & 1 \\ 1 & 1 \end{smallmatrix}\right]$ has rank 1.
	\end{list}

	\begin{figure}
		\centering \begin{tikzcd}[column sep = 0]
			(u, \tau z) \dar[dash, "\bf 1" description]\ar[rd, dash, "1" description] & (\tau u, z) \dar[dash, "\bf 1" description]\ar[ld, dash, "1" description] \\
			(v, \tau z) & (\tau v, z)
		\end{tikzcd} \quad
		\begin{tikzcd}[column sep = 0, row sep = large]
			&&& (x, \tau x) \ar[dash, llld, "\bf 1" description]\ar[dash, ld, "\bf 1" description] \ar[dash, rrrd, "\bf 1" description]\ar[dash, rd, "\bf 1" description] &&& \\
			(u, z) \ar[dash, rd, "\bf 1" description]\ar[rrrd, dash, "1" description, pos = 0.6]\ar[rrrrrd, dash, dotted, "\phantom 0" description, pos = 0.6] && (x, \tau y) \ar[dash, rd, "\bf 1" description]\ar[dash, ld, dotted, "\bf 0" description, pos = 0.75]\ar[dash, rrrd, "\bf 1" description] && (y, \tau x)\ar[dash, ld, "\bf 1" description]\ar[dash, rd, dotted, "\bf 0" description, pos = 0.75]\ar[dash, llld, "\bf 1" description] && (\tau u, \tau z)\ar[dash, ld, "\bf 1" description]\ar[llld, dash, "1" description, pos = 0.6]\ar[llllld, dash, dotted, "0" description, pos = 0.63] \\
			& (v, z) && (y, \tau y) && (\tau v, \tau z) &
		\end{tikzcd}
		\caption{Depiction of the two nontrivial chain complexes constituting $\widehat{\it CF}$ over $\bb F_2$, broken over ${\rm Spin}^c$-structures. Height signifies Maslov degree, in decreasing order from top to bottom. The majority of the differentials are marked with boldface, and those suggest that the asserted fact is ``obvious'' from the diagram, while the other labels must be filled in by some other principle.}\label{fig:gens}
	\end{figure}
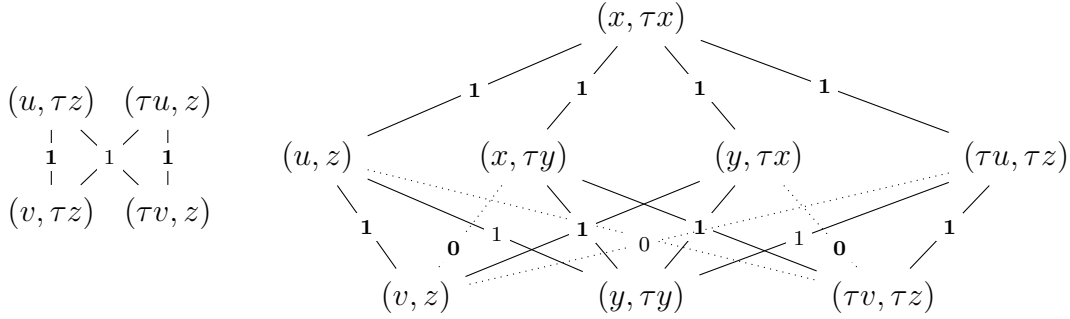

	Regarding the bigger complex, shown on the right, we have:
	\begin{list}{$\cdot$}{}
		\item The top four arrows emerging out of $(x, \tau x)$: each one admits either an obvious rectangle (i.e.~ the arrow landing in $(u, z)$ or $(\tau u, \tau z)$, cf. central picture of {\sc Fig.~\ref{fig:doms}}), or an obvious bigon (i.e.~ landing in $(x, \tau y)$ or $(y, \tau x)$), contributing 1 mod 2. Adding any non-zero multiple of the generating periodic domain leads to negative multiplicities somewhere, hence there are no more contributions.
		\item $(u, z) \to (v, z)$ and its reflection: as before, there is an obvious bigon contributing 1 mod 2, and any non-trivial shift by a periodic domain leads to negative multiplicities.
		\item $(x, \tau y) \to (y, \tau y)$ and its reflection: exactly the same justification.
		\item $(x, \tau y) \to (\tau v, \tau z)$ and its reflection: the same justification, except the 1 mod 2 comes from a rectangle instead of a bigon (cf. the rightmost picture in {\sc Fig.~\ref{fig:doms}}, depicting the reflection of said rectangle, from $(y, \tau x)$ to $(v, z)$, which differs from the previous rectangle in the middle picture by a smaller rectangular strip).
		\item $(x, \tau y) \to (v, z)$ and its reflection: any representative of this class has a negative multiplicity somewhere, hence it cannot contribute.
		\item $(u, z) \to (\tau v, \tau z)$, and its reflection: the same explanation, or alternatively this follows by $d^2 = 0$ and the previous information.
		\item $(u, z) \mapsto (y, \tau y)$, and its reflection: this is the only remaining one, and the only positive domain is a genus-1 domain, which is not that easy to calculate directly. Also, $d^2 = 0$ does not help. Instead, we rely on the fact that we expect the homology of this chain complex to be given by two $\mathbb F_2$'s in adjacent degrees: the top generator already dies in homology, so it must be that the two bottom layers contribute to this non-trivial homology. Equivalently, the matrix representing the differential going from the middle to the bottom row must have rank 2. From the information we have gathered so far, it must look like
		\[ \left[\begin{matrix}
			1 & 0 & 1 & 0 \\
			a & 1 & 1 & a \\
			0 & 1 & 0 & 1
		\end{matrix}\right], \]
		where the common value $a$ is the one that remains to be elucidated. From the rank condition, it follows that $a$ must be 1, hence completing the analysis of $\widehat{\it CF}$.
	\end{list}

	\begin{figure}
		\centering\includegraphics[scale=1]{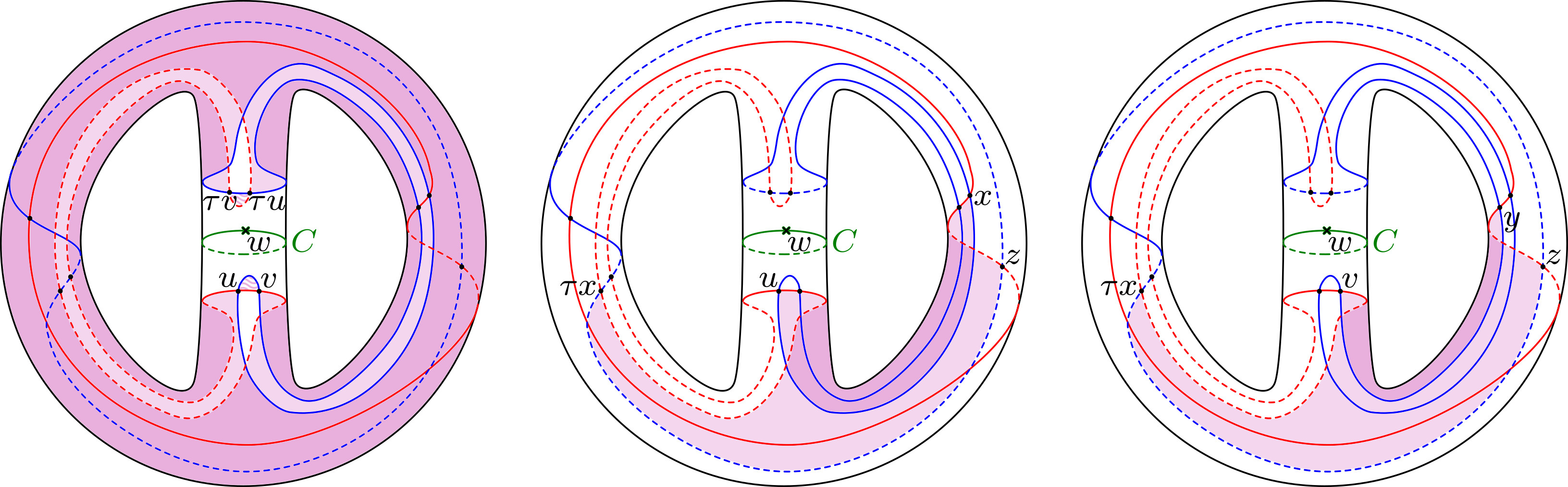}
		\caption{In order from left to right: a depiction of the generator of the group of periodic domains, with multiplicity equal to 1 except for the two small bigons from $u$ to $v$ and $\tau u$ to $\tau v$ with multiplicity equal to $-1$; a rectangle from $(x, \tau x)$ to $(u, z)$; a rectangle from $(x, \tau y)$ to $(v, z)$.}\label{fig:doms}
	\end{figure}

	From here, the formula (\ref{eq:broken}) implies that the differential in $\widehat{\it CFR}$ from $(x, \tau x)$ to $(y, \tau y)$ is given by $\tfrac 12 \cdot 4 = 2 \equiv 0$ mod 2. We recognize the contribution of the free disjoint union of two bigons coming from the middle two terms of the middle row of {\sc Fig.~\ref{fig:gens}}, and therefore that the other homotopy class also contributes 1 mod 2, due to the outer terms of the middle row.

	If we are doing the computation over $\mathbb Z$, it follows that the real Floer homology $\widehat{\it HFR}$ in the bottom degree is either $\mathbb Z/d$ or $\mathbb Z/(d+2)$, depending on which choice of coherent orientations we make, where $d$ is an even number. No matter what $d$ is, the two results are different, which shows that the two-element set of choices of $\mathbb Z$-lift is a well-defined invariant. A posteriori, it follows that the set of $\mathbb Z/2$-gradings is a well-defined invariant, since at least one of the choice will lead to a situation where $\widehat{\it HFR}_0 \not\cong \widehat{\it HFR}_1$. We have not been able to pinpoint exactly what $d$ is, which is a limitation inherent in {\sc Prop.~\ref{prop:trick}}'s formula only being true mod 2. One would have to either compute the genus-1 domain's contribution by hand, or refine the formula (\ref{eq:broken}) to incorporate signs, the latter presupposing some sort of comparison between orientations in $\widehat{\it CF}$ and $\widehat{\it CFR}$.

	\subsection{Non-vanishing obstruction for \texorpdfstring{$\bb Z$}{Z}-coefficients}\label{ssec:q8}

	Finally, we provide an example in which the obstruction $Q$ for $\mathbb Z$-coefficients does not vanish, which we describe first in terms of a real Heegaard diagram, and a posteriori also in more concrete terms.

	\begin{EXP}\label{exp:heegquat}
		Consider the real Heegaard diagram of genus 2 shown in the top-left corner of {\sc Fig.~\ref{fig:nonex}}, where the $\alpha$-curves are the ones being pictured, and the involution is given by flipping along the vertical axis of symmetry.
	\end{EXP}
	\begin{figure}
		\centering\includegraphics{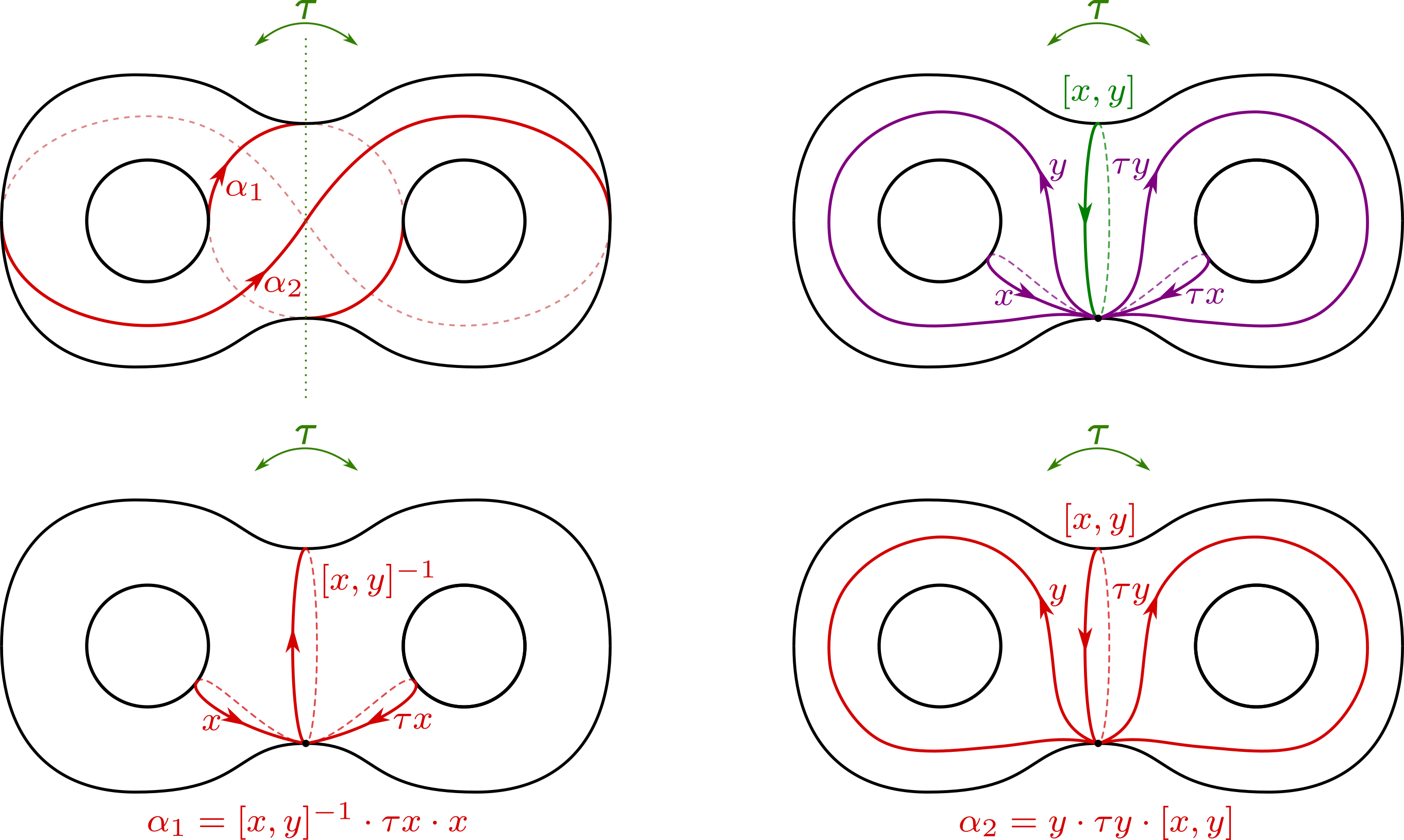}
		\caption{In the top-left corner, there is a depiction of the $\alpha$-curves for a real Heegaard diagram for which the obstruction $Q$ to $\bb Z$-coefficients does not vanish; in the top-right corner, we show the orientation conventions for the generators of $\pi_1 \Sigma$, including a word decomposition $[x, y] := xyx^{-1}y^{-1}$ for the green middle circle; on the bottom we show decompositions of the two $\alpha$-curves into words in terms of the generators.}\label{fig:nonex}
	\end{figure}

	If we let $x$ and $y$ denote standard generators for $H_1(\Sigma')$ given by the meridian and longitude, cf. top-right corner of {\sc Fig.~\ref{fig:nonex}}, then the $\alpha$-curves have been chosen specifically so that they represent classes $(1 + \tau)x$, and $(1 + \tau)y$ in homology. Let us recall from the proof of {\sc Prop.~\ref{prop:spcalc}} that $A$ can be represented as the kernel of the map (\ref{eq:composite2}), i.e.~ as the intersection
	\[ H_1(\bb T_\alpha) \cap H_1(\Sigma') \text{ inside } H_1(\Sigma). \]
	In our case, the $\alpha$-curves so happen to generate precisely $H_1(\bb T_\alpha) = H_1(\Sigma')$ inside $H_1(\Sigma)$, and hence
	\begin{equation}\label{eq:AforQ8}
		A \cong \bb Z\<(1+\tau)x, (1+\tau)y\>.
	\end{equation}
	The form $F$ may be further computed as the mod-2 intersection number of the given elements of $H_1(\Sigma')$. In this case, it is clear that $(1+\tau)x$ and $(1+\tau)y$ have intersection number 1, and self-intersection number 0, hence proving that $F = Q$ is standard symplectic on two generators, and in particular non-trivial. Thus, the obstruction to $\mathbb Z/2$-gradings vanishes, but the one for $\mathbb Z$-coefficients does not. We can identify at least $Y'$, as follows:

	\begin{PROP}
		The 3-manifold $Y'$ associated to {\sc Ex.~\ref{exp:heegquat}} is diffeomorphic to $S^3/Q_8$.
	\end{PROP}

	\begin{proof}
		Standard 3-manifold theory implies that any closed connected 3-manifold with $\pi_1 \cong Q_8$ must be the standard spherical manifold $S^3/Q_8$ up to diffeomorphism, so it suffices to compute the fundamental group. If $x$ and $y$ represent the standard meridional and longitudinal curves shown in the top-right corner of {\sc Fig.~\ref{fig:nonex}}, then the red curves end up being represented by the words $[x,y]^{-1} \cdot \tau x \cdot x$ and $y \cdot \tau y \cdot [x,y]$, cf. bottom of {\sc Fig.~\ref{fig:nonex}}, where the convention is that $[x, y] := xyx^{-1}y^{-1}$. Consequently, $\pi_1(Y')$ ends up being isomorphic to the presentation
		\[ \<x, y | x^2 = y^{-2} = [x,y]\>, \]
		because the map $\Sigma \to \Sigma' \to Y'$ identifies $x$ with $\tau x$ and $y$ with $\tau y$.
		It is not hard to see that this is a presentation of the quaternion group: e.g.~ we can send $x, y$ to $i, j$, respectively, and get a surjective map $\pi_1 Y' \twoheadrightarrow Q_8$. To see that it is an isomorphism, it remains to show that the domain has at most 8 elements, exhaustively enumerated by the list
		\[ \{x^i y^j : 0 \le i \le 3, 0 \le j \le 1\}. \]
		To this end, first note that the relation $x^2 = y^{-2}$ implies that in any word we may ensure that at most one $y$ appears at a time (and no negative powers of $y$), i.e.~ every word is of the form
		\[ x^{n_1}y x^{n_2} y \cdots y x^{n_k}, \]
		for $n_1, \ldots, n_k \in \bb Z$. Next, the relation $x^2 = xyx^{-1}y^{-1}$ becomes equivalent to $xy = yx^{-1}$. From here, this allows us to inductively get rid of the powers of $x$ in between the $y$'s, which after further using $y^2 = x^{-2}$, reduces the word to the form $x^i y^j$, with $j \in \{0,1\}$. Lastly, it remains to show that $x^4 = 0$, which immediately follows form
		\[ x^2 y = y^{-2} y = y y^{-2} = y x^2 = x^{-2} y. \]
		This concludes the proof.
	\end{proof}

	\begin{PROP}
		The cohomology of $S^3 / Q_8$ with $\bb F_2$-coefficients is determined by the following generators and relations:
		\begin{list}{$\cdot$}{}
			\item $H^1 =\{0, a, b, c\}$ is the Klein group,
			\item $H^2 =\{0, a^2, b^2, c^2\}$ is likewise a Klein group, with $ab = c^2, bc = a^2, ca = b^2$,
			\item and $abc = a^3 = b^3 = c^3 = 0$, $a^2b = ab^2 = b^2c = bc^2 = c^2a = ca^2 = \iota$ in $H^3 = \bb F_2\<\iota\>$.
		\end{list}
	\end{PROP}

	\begin{proof}
		The abelianization of $Q_8$ is the Klein group, hence the claim about $H^1$. Also, since $H^1(S^3/Q_8; \bb Z) = 0$, and $\dim H^2(S^3/Q_8; \bb F_2) = 2$ by Poincar\'e duality, it follows that
		\[ {\rm Sq}^1 : H^1(S^3/Q_8; \bb F_2) \to H^2(S^3/Q_8; \bb F_2) \]
		is an isomorphism, which proves the claim that $H^2 = \{0, a^2, b^2, c^2\}$ is a Klein group.

		For the claim that $ab = c^2$, etc., we must work a bit harder. First, let us observe that when a group $G$ acts freely on a sphere $S^n$, the map $S^n / G \to BG$ induces an isomorphism on cohomology groups up to degree $n-1$; indeed, this follows from the Leray-Serre spectral sequence for the fiber sequence $S^n \to S^n / G \to BG$ and the vanishing of the rows $1 \le q \le n-1$. Thus, it suffices to prove the relations concerning double cup products in $H^{*\le 2}(BQ_8; \bb F_2)$. To this end, one can analyze the Leray-Serre spectral sequence associated to the fiber sequence $BC_2 \to BQ_8 \to BK_4$ (here, $K_4$ denotes the Klein group, and acts trivially on the fiber due to the extension being central.) Using $x \in H^1(BC_2; \bb F_2)$ to denote the generator, and the same letters $a, b, c$ to denote the nonzero elements of $H^1(BK_4; \bb F_2)$, we claim that $d_2(x \otimes 1) = ab + bc + ca$. Indeed, there is an action of $S_3$ permuting the nonzero elements of the Klein group, and that extends to an action on the spectral sequence, so $x$ must map to an element invariant under this $S_3$ action. It cannot map to zero for dimension reasons, since we expect $\dim H^2(BQ_8; \bb F_2) = 2$, and the only other $S_3$-invariant element is $ab + bc + ca$, as desired. Due to the fact that $c = a + b$, this proves that
		\[ 0 = ab + bc + ca = ab + (a+b)c = ab + c^2, \]
		i.e.~ the desired $ab = c^2$. The others follow by the $S_3$-symmetry.

		Finally, we must prove the relations concerning the triple cup products. First, in any \emph{orientable} 3-manifold $Y$, we have  $x^2 y = x y^2$ for all $x, y \in H^1(Y; \bb F_2).$ Indeed,
		\[ x^2 y + x y^2 = {\rm Sq}^1(x \smile y) = w_1(TY) \smile x \smile y, \]
		by the Wu relation for ${\rm Sq}^1$, and the orientability ensures the vanishing of this expression. In our case, this fact together with the $S_3$ symmetry guarantee
		\begin{equation}\label{eq:sixelts}
			a^2b = ab^2 = b^2c = bc^2 = c^2a = ca^2.
		\end{equation}
		In particular, $0 = a^2 b + a^2 c = a^2(b + c) = a^3$, and likewise for $b^3$ and $c^3$. The relation $ab = c^2$ also forces $abc$ to vanish, so we obtain the desired
		\[ a^3 = b^3 = c^3 = abc = 0. \]
		Finally, Poincar\'e duality demands that there be at least some non-trivial cup product, so in particular (\ref{eq:sixelts}) must be the fundamental class $\iota$, concluding the proof.
	\end{proof}

	Using this fact, we may finally compute the forms $F$ and $Q$ algebraically. First of all, we know that $A$ has rank 2, cf. (\ref{eq:AforQ8}); on the other hand, we know that $A/2A$ embeds in $H^1(S^3 / Q_8; \bb F_2) \cong \mathbb F_2^{\oplus 2}$, which is also rank two, hence implying that $A \twoheadrightarrow H^1(S^3 / Q_8; \bb F_2)$. The computation in the Proposition directly above shows that $F(x, x) = \int x^3$ is identically zero (so the obstruction to $\mathbb Z/2$-gradings vanishes), but the form $F(x, y) = Q(x, y)$ is sometimes nonzero, e.g.~ for $x, y$ lifting $a, b$ respectively. Thus, the obstruction to $\mathbb Z$-coefficients does not vanish, recovering the diagram-level conclusion.

\section{Canonicity of choices and Atiyah-Bott-Shapiro pushforwards in \texorpdfstring{$\it KR$}{KR}-theory}\label{sec:abs}

	In this somewhat speculative last section, we investigate the question of canonicity of $\mathbb Z/2$-gradings and $\mathbb Z$-lifts, and its conjectural relation to Atiyah-Bott-Shapiro pushforwards in $K$-theory, motivated by the monopole-theoretic analogue, cf. Baraglia \cite[\S6]{Ba25}. We already know that when $A = 0$, the choice of $\bb Z$-lift is unique, and that in certain cases even when $A \neq 0$, canonicity of $\bb Z/2$-gradings and $\bb Z$-lifts can be deduced a posteriori, due to the lack of symmetry in $\widehat{\it HFR}$ with respect to the torsor action (as seen in {\sc\S\ref{ssec:s1s2},~\S\ref{ssec:z4}}).
	
	\subsection{Various notions of canonicity, and the Fukaya category approach} For each one of the two upgrades ($\bb Z/2$-gradings and $\bb Z$-lifts), there are several closely related questions concerning canonicity of the choices:
	\begin{enumerate}
		\item Whether the set of all choices of said piece of data is a well-defined invariant of the 3-manifold $Y$ with involution $\tau$ (i.e.~ independent of the real Heegaard diagram);
		\item Whether this set can additionally be represented as a purely algebro-topological invariant of $(Y, \tau)$;
		\item Whether there is a distinguished element in this set, in the case that it is nonempty.
	\end{enumerate}
	Of course, {\sc ii} implies {\sc i}. In addition, note that {\sc iii} would imply {\sc ii}, since the torsor structure on the set would give a preferred identification of it with either $\mathbb Z/2$ or $A^\vee$. In their second paper \cite{OS-II}, Ozsv\'ath-Szab\'o show that classical Heegaard Floer theory always has a preferred $\mathbb Z$-lift, which is uniquely determined by a certain property of $\underline{\it HF}^\infty$ with $\mathbb Z$-coefficients; this fact is quite non-trivial, and requires developing a significant amount of the functoriality and exact sequence properties for the theory. Whether this approach could be made to work in the real setting remains unclear to us at the present moment.

	As for $\widehat{\it HFR}$, Srivastava \cite{Sr26} does establish {\sc iii} for $\mathbb Z/2$-gradings provided that a lift of $v \in H^1(Y' \setminus N'; \mathbb Z/2)$ to $\mathbb Z$-coefficients is given. However, as we have shown in {\sc\S\ref{ssec:z4}}, Srivastava's condition for $\mathbb Z/2$-gradings is not if-and-only-if, so some work remains to be done in establishing {\sc iii} in full generality. When it comes to the same question concerning $\mathbb Z$-coefficients in $\widehat{\it HFR}$, we have good reason to believe that {\sc i} and {\sc ii} hold, which we will present in this section, though the plausibility of {\sc iii} remains unclear to us.

	Let us begin by discussing canonicity question {\sc i} in the  context of $\widehat{\it HFR}$, and why it is reasonable to believe that it always holds true. In {\sc\S\ref{sec:gror}}, we have developed the necessary obstruction theory for $\mathbb Z/2$-gradings and $\mathbb Z$-lifts in the context of Floer chain groups for a particular pair of Lagrangians, as well as for continuation maps induced by a particular Hamiltonian isotopy of either one of them. It is a priori perfectly possible that a self-isotopy of the Lagrangians might induce a non-trivial permutation of the set of choices of $\mathbb Z/2$-gradings and $\mathbb Z$-lifts; in the context of $\widehat{\it HFR}$, this is complicated even further by the existence of stabilization moves. One would have to prove that these induced automorphisms on the set of choices are the identity; a conceivable way to accomplish this would be to establish an obstruction theory for the \emph{higher coherence homotopies}, associated to \emph{higher Heegaard moves}, cf. \cite{JTZ21, GM26}.

	These higher coherence homotopies would come from the $\mu_3$-maps of the $\mathbb A_\infty$-structure on the Fukaya category, and are associated to the singularities in the Morse function one encounters as one tries to fill in the $S^1$-parameter family of functions with a $D^2$-parameter family. Therefore, claim {\sc i} is true provided that
	\begin{list}{$\cdot$}{}
		\item These singularities and associated coherence homotopies could be classified, in the context of real Heegaard diagrams, which has recently been done in \cite{GM26};
		\item And the obstruction theory could be worked out for them in such a way that a choice of $\mathbb Z/2$-gradings or $\mathbb Z$-lifts for one Floer chain group would induce a unique extension of the choice for the whole ensemble of moduli spaces involved in the process.
	\end{list}
	We mention in passing that a similar story should hold for establishing functoriality and knot-surgery exact triangles, for which one needs to go all the way up to holomorphic pentagons, i.e.~ $\mu_4$ in the $\mathbb A_\infty$-algebra, cf. also Manolescu-Ozsv\'ath \cite{MO25} for even higher $\mu_k$ in the context of link-surgery. Thus, all that would remain to be done is setting up obstruction theory for the higher $\mu_k$ in the Fukaya category, as well as their interaction with continuation maps from Hamiltonian isotopy. This is, in principle, a matter of pure combinatorics and homological algebra, as we now quickly show, though we do not develop in full detail.

	Beginning with $\mu_2$ for Lagrangians $L_0$, $L_1$, $L_2$, one must consider either the mod-2 virtual dimension function or determinant line bundle over the space $\scr T(p, q, r)$ of triangles with boundaries on $L_0$, $L_1$, $L_2$, and with corners at intersection points $p \in L_0 \cap L_1$, $q \in L_1 \cap L_2$, $r \in L_0 \cap L_2$. In the case of determinant lines, we obtain a primary obstruction in $H^1$ of this space with $\pm 1$ coefficients, as before; we expect this to vanish in the stable range for $\widehat{\it HFR}$. The obstruction to $\mathbb Z/2$-gradings and the secondary obstruction to coherent orientations then live as relative 1- and 2-cocycles in the simplicial nerve of the category determined by:
	\begin{list}{$\cdot$}{}
		\item objects of two kinds: either pairs $(p, q)$ of intersection points as above, or single intersection points $r$;
		\item arrows of three kinds: pairs of homotopy classes of bigons $p' \rightsquigarrow p$ and $q' \rightsquigarrow q$ from $(p', q')$ to $(p, q)$; homotopy classes of bigons $r \rightsquigarrow r'$ from $r$ to $r'$; homotopy classes of triangles i.e.~ $\pi_0 \scr T(p, q, r)$ from $(p, q)$ to $r$ (the hom-set from $r$ to $(p, q)$ is empty).
	\end{list}
	Indeed, let us assume that the choices have already been made for each one of the three chain groups; then, the cocycle has already been trivialized on the subcategory consisting of no maps from $(p, q)$ to $r$, and the extendability to the whole category is given by a relative cocycle. In this case it is not hard to give an algebraic interpretation of the (co)chains of this nerve: first of all, the category is equivalent to a skeleton consisting of just one arbitrarily chosen pair $(p, q)$, and an arbitrary $r$. Just as in the case of continuation maps, the (co)chains compute a double-sided bar construction 
	\[ B_\bullet\big(\mathbb Z[G_p] \times \mathbb Z[G_q], \mathbb Z[\pi_0 \scr T (p, q, r)], \mathbb Z[G_r]\big),\]
	where $G_p, G_q, G_r$ are the isotropy groups of groupoids associated to the three chain groups, i.e.~ $\pi_1 \scr P(L_0, L_1)$, $\pi_1 \scr P(L_1, L_2)$ and $\pi_1 \scr P(L_0, L_2)$ respectively, and the action of these on $\pi_0 \scr T(p, q, r)$ is given by concatenation. 
	
	Whether one would have to do some index theory to calculate the obstruction is not so clear: however, there is hope that, just like in the case of continuation maps {\sc(\S\ref{ssec:cmaps})}, the obstruction extends uniquely to the whole category, due to a quasi-isomorphism of bar constructions, and in this case there would be no further calculations to be done. Before we move on to higher $\mu_k$, we note that $\pi_0 \scr T(p,q,r)$ should have an interpretation as the set of ${\rm Spin}^c$-structures on the 4-manifold whose trisection is given by the three sets of curves, in a similar case to classical $\it HF$, cf. \cite[\S6]{OS-II} and \cite[\S2.2]{OS-III}.
	
	For higher $\mu_k$ associated to $L_0, L_1, \ldots, L_k$, let us assume that all the lower-degree $\mu$'s have been successfully endowed with the desired structure. There is once again a primary obstruction to coherent orientations on the space of polygons (which we expect to vanish in the stable range in the case of $\widehat{\it HFR}$), and there is an obstruction for $\mathbb Z/2$-gradings and a secondary obstruction for $\mathbb Z$-coefficients which live as relative cocycles in the simplicial nerve of the category with
	\begin{list}{$\cdot$}{}
		\item objects given by tuples $(p_{0 = i_0, i_1}, p_{i_1, i_2}, \ldots, p_{i_{\ell-1}, i_\ell = k})$ of intersection points where $p_{\alpha, \beta} \in L_\alpha \cap L_\beta$, for any (potentially empty) partition $0 = i_0 < \cdots < i_\ell = k$.
		\item morphisms from $(p\ldots)$ to $(p'\ldots)$ exist only if $(i\ldots)$ is a sub-partition of $(i'\ldots)$, and in this case they are given by tuples of homotopy classes of polygons with corners given by the points all belonging to $p\ldots$ and the corresponding point in $p'\ldots$ in the same sub-segment of the partition.
	\end{list}
	If we have already endowed the lower $\mu_k$ with the necessary data, then that provides a trivialization of the obstruction on a subcomplex of the nerve, and once again the extension problem is given by a relative cocycle. We hope that this extension problem too, will turn out to have a unique solution for reasons of homological algebra. We also expect that the sets of homotopy classes expressed above can also be written in terms of ${\rm Spin}^c$-structures.

	\subsection{Pushforwards in \texorpdfstring{$\it KR$}{KR}-theory, and the analogy to the monopole side} In the remainder of this section, we address question {\sc ii} (and implicitly {\sc i}), to which we believe an affirmative answer can be given by means of the ${\rm Spin}^c$-pushforward in $K$-theory. In its algebraic form, this construction originates from work of Atiyah, Bott and Shapiro \cite{ABS64}, designed to understand Bott periodicity via Clifford algebras; in this seminal paper, they construct a Thom isomorphism in $K$-theory for bundles equipped with ${\rm Spin}$ and ${\rm Spin}^c$ structures, in $\it KO$ and $\it KU$ versions respectively. This has been generalized to $\it KR$-theory by Atiyah \cite{At66}; see also Joachim \cite{Jo04} for a modern approach that constructs a map $\it MSpin^c \to KU$ of $\bb E_\infty$-ring spectra, followed by Halladay-Kamel \cite{HK24} constructing a $C_2$-equivariant analogue in $\it KR$.
	
	The analytic significance of this $K$-theoretic pushforward is that it is supposed to model the $K$-theoretic index bundle associated to a family of linear Dirac operators on an even-dimensional manifold, which are precisely the kinds of operators that show up in Seiberg-Witten/monopole theories; this may be taken as an alternative definition, and we will sketch why they are equivalent at least in the non-equivariant setting, via the $K$-theoretic formulation of the Atiyah-Singer index theorem, cf. {\sc Constr.~\ref{constr:asind}}. The work that we are about to undertake parallels Baraglia \cite[\S6]{Ba25}'s construction in monopole theory, though we provide an alternative formulation more fitting for the Heegaard Floer side; the equivalence of the two is sketched in {\sc Rmk.~\ref{rmk:bar}}.

	Let us begin by recalling the algebraic construction first:
	
	\begin{CONSTR}\label{constr:pushfwd}\cite[(2) in {\S11}]{ABS64}
		Let $(M^n, \mathfrak s)$ be a ${\rm Spin}^c$-manifold, and $E \in {\it KU}(M)$ be a virtual complex bundle over it. An ``integration over the fundamental class''
		\[ \int_{M, \fk s} E \in {\it KU}^{-n}({\rm pt}) \]
		may be constructed as follows: First, let us embed $M^n$ into some Euclidean space $\mathbb R^{n + 2k}$ with large $k$ (the connectivity of the space of such embeddings goes to $\infty$ as $k \to \infty$). The tangential ${\rm Spin}^c$-structure $\fk s$ (of either positive or negative type, cf. {\sc Def.~\ref{def:spinc}}) on $M$ may be transported to its normal bundle $\scr N_M$ (more on this in the line after (\ref{eq:twospinc})). This admits a Clifford-theoretic interpretation: there are two isomorphism types of irreducible ${\rm Cl}_\pm(\scr N_M)$ super-modules (i.e.~ $\bb Z/2$-graded, with ${\rm gr}(v) = 1$), depending on whether the involutive element $\omega_\bb C = i^{k(2k-1)} e_1 \cdots e_{2k}$ induced by the orientation choice acts by the identity or minus the identity on the even-graded piece of the module. Let $S^* = S^{[0]} \oplus S^{[1]}$ be such a graded irrep where $\omega_\bb C$ acts trivially on the even piece $S^{[0]}$ (and hence as multiplication by $-1$ on the odd piece $S^{[1]}$). This comes equipped with a Clifford multiplication
		\begin{equation}\label{eq:clm}
			\scr N_M \to \scr H{\it om}(S^{[1]}, S^{[0]})
		\end{equation}
		Since $M$ is compact, there exists a bundle $R^{[0]}$ such that $R^{[0]} \oplus S^{[0]} \cong \underline{\mathbb C}^{r}$ for some $r$ (the space of such choices can also be arranged to be arbitrarily highly connected as $r \to \infty$). In this case, the bundle $R^{[0]} \oplus S^{[1]}$, when pulled back to the total space of $\scr N_M$, admits a trivialization away from the 0-section, given by applying the Clifford multiplication to the $S^{[1]}$ factor, and then applying the trivialization $R^{[0]} \oplus S^{[0]} \cong \underline{\mathbb C}^{r}.$ This gives a well-defined \keywd{Atiyah-Singer difference element} associated to the map (\ref{eq:clm}):
		\[ \theta_\mathfrak s \in {\it KU}^0(\scr N_M, \scr N_M - 0_{\rm sec}),\]
		which due to work of \cite[\S12]{ABS64} is a \emph{Thom class} for $\scr N_M$ in $K$-theory. Finally, we obtain the ``integral of $E$'' by tensoring $\theta_\mathfrak s$ with $[E]$, and then extending by zero to the whole of $\mathbb R^{n+2k}$:
		\[ \int_{M, \mathfrak s} E := j_!(\theta_\mathfrak s \otimes E) \in {\it KU}^0_c(\mathbb R^{n+2k}) \cong {\it KU}^{-n}({\rm pt}), \]
		where $j : \scr N_M \hookrightarrow \bb R^{n+2k}$ is the tubular inclusion. So far, this construction is not interesting at all when $n$ is odd, since the target group is zero, though it does become interesting if we allow the following two generalizations:

		First, we may apply the construction parametrically. Suppose $X$ is a topological space with a fiber-wise ${\rm Spin}^c$ structure $\underline{\mathfrak s}$ on $M^n \times X \overset p\twoheadrightarrow X$. We may embed $M^n \times X$ inside $\mathbb R^{n + 2k} \times X$ and perform the construction parametrically to get an \keywd{integration along fibers} or \keywd{pushforward}
		\[ p^{\underline{\mathfrak s}}_! : {\it KU}^0(M^n \times X) \to {\it KU}^{-n}(X). \]
		The other generalization is that, if $(M^{p + q}, \tau)$ is a manifold with involution, with fixed locus of dimension $p$, and $\underline{\mathfrak s}$ is a fiber-wise real ${\rm Spin}^c$-structure on $M \times X \overset p \twoheadrightarrow X$ (of either positive or negative type), then there is a similar pushforward
		\[ p^{\underline{\mathfrak s}}_! : {\it KR}^{0, 0}(M^{p+q} \times X) \to {\it KR}^{-p, -q}(X) \cong {\it KO}^{q-p}(X). \]
		Indeed, this is because the construction can be ensured to preserve the involution, by embedding $M^{p+q}$ equivariantly in the $C_2$-representation $\mathbb R^{p+k} \oplus i\mathbb R^{q+k}$.
	\end{CONSTR}

	Now, given any (real) 3-manifold $Y$ (with fixed locus of dimension 1), there is a tautological fiberwise (real) ${\rm Spin}^c$-structure $\underline{\fk t}$ on $Y \times \spinc(Y)$ (resp.~ $Y \times \rspinc(Y)$). We have already seen that the classical and real stable pathspaces are canonically homotopy-equivalent to the spaces of (real) relative ${\rm Spin}^c$ structures. We have already seen that the obstructions to $\mathbb Z/2$-gradings and $\mathbb Z$-lifts are expressed as Stiefel-Whitney classes of a certain K-theoretic element $[TL_0 - TL_1]$ on the said pathspace, so it is reasonable to ask whether this element can be interpreted as some kind of pushforward in K-theory. An even more general invariant than $[TL_0 - TL_1] \in {\it KO}^0(\scr P)$ is the \keywd{polarization class} $\theta \in {\it KO}^1(\scr P) \cong [\scr P, U/O]$, obtained from the fact that the formal difference $TL_0 - TL_1$ has a canonical trivialization of its complexification (as $TL_0 \otimes \bb C \cong TM \cong TL_1 \otimes \bb C$), and hence gives rise to an element of ${\rm HoFib}(BO \to BU) \cong U/O$. One can recover the formal difference via the map $U/O \to */O \cong BO$, which is classified by multiplication with the unique nonzero element $\eta$ in ${\it KO}^{-1}({\rm pt}) \cong \mathbb Z/2$, to get a homomorphism ${\it KO}^1 \to {\it KO}^0$; another useful invariant that comes from $U/O$ but not from $BO$ is the \emph{Maslov class}, given by post-composing with the map $U/O \overset{\rm det}\longrightarrow U_1 / O_1 \cong S^1$. This polarization map carries the obstruction necessary to define \emph{Floer homotopy theory} over $\mathbb S$, and as such is of great interest in the case of (real) Heegaard Floer theory.

	\begin{CONJ}\label{conj:abs}
		The polarization class for classical Heegaard Floer theory is given, up to a sign, by the image of $[\underline {\bb C}] \in {\it KU}^0(Y \times \spinc(Y))$ under the composite
		\[ {\it KU}^0(Y \times \spinc(Y)) \overset{p^{\underline{\fk t}}_!}\longrightarrow {\it KU}^{-3}(\spinc(Y)) \overunderset{\rm Bott}{\sim}{=\!\!\!=} {\it KU}^{1}(\spinc(Y)) \overset{\rm forg}\longrightarrow {\it KO}^1(\spinc(Y)). \]
		Likewise, the polarization class for real Heegaard Floer theory is given, up to a sign, by the image of $[\underline{\bb C}] \in {\it KR}^{0,0}(Y \times \rspinc(Y))$ under the composite
		\[ {\it KR}^{0,0}(Y \times \rspinc(Y)) \overset{p^{\underline{\fk t}}_!}\longrightarrow {\it KR}^{-1, -2}(\rspinc(Y)) \overset{\sim}\longrightarrow {\it KO}^1(\rspinc(Y)). \]
		Moreover, this should be true at the spectrum level up to a contractible space of choices, hence giving an algebro-topological interpretation of the set of $\bb Z/2$-gradings or $\bb Z$-coefficient lifts as the set of homotopy classes of null-homotopies of the spectrum-level polarization map, post-composed with $w_1$, or $w_2$, respectively.
	\end{CONJ}

	\begin{RMK}\label{rmk:bar}
		One can also consider the bigger $KR$-pushforward
		\[ {\it KR}^{0,0}(Y \times \spinc(Y)) \overset{p^{\underline{\fk t}}_!}\longrightarrow {\it KR}^{-1,-2}(\spinc(Y)), \]
		with the diagonal action on $Y \times \spinc(Y)$. This is rich enough to recover both conjectural polarization maps, and is probably the obstruction to a $C_2$-equivariant Heegaard Floer spectrum. It is very likely the same as ${\rm ind}(D) = (\Phi \circ \alpha_*) (1)$ of Baraglia \cite[\sc Prop.~6.6]{Ba25}, in the following sense: if we pick a background (real) ${\rm Spin}^c$-structure $\fk s_0$, then every other ${\rm Spin}^c$-structure is of the form $\fk s_0 \otimes L$ for some line bundle $L$, which could be used to construct an equivalence ${\it Pic}(Y) \overset\sim\to \spinc(Y)$. If we let $\scr L$ the tautological Atiyah-real line bundle on $Y \times {\it Pic}(Y)$, we believe that one should be able to interpret $p^{\underline{\fk t}}_!([\bb C])$ as $p^{\fk s_0}_!(\scr L)$, i.e.~ ${\rm ind}(D)$ from \cite{Ba25}.
	\end{RMK}
	
	We offer some strong evidence in favor of {\sc Conj.~\ref{conj:abs}}, by sketching a proof for how $w_1$ and $w_2$ of this conjectural polarization class provide obstructions in the case of $\widetilde{\it HMR}$, the monopole analogue of $\widehat{\it HFR}$, which is conjecturally the same. To this end, let us now relate {\sc Constr.~\ref{constr:pushfwd}} to its promised analytical interpretation in terms of indices of Dirac operators:

	\begin{CONSTR}\label{constr:asind}
		A particular case of the ${\rm Spin}^c$-pushforward is the complex structure pushforward: if $V$ is a complex vector space, then $\Lambda_\bb C^* V$ is an irreducible Clifford module over $V$. If $M^n \subset \bb R^N$ is a closed submanifold, the normal bundle of $T^*M \subset T^*\bb R^N$ is naturally the complexification $\bb C \otimes \scr N_M$ of the normal bundle of the original inclusion. This gives a pushforward
		\[ {\it KU}_c^0(T^*M) \to {\it KU}_c^0(T^*\bb R^N) \cong \bb Z. \]
		If $D : C^\infty(M, \scr E) \to C^\infty(M, \scr F)$ is an elliptic operator, it has an associated symbol element $\sigma(D) \in {\it KU}_c^0(T^*M)$, given by the Atiyah-Singer difference element associated to the symbol $T^*M \to \scr H{\it om}(\scr E, \scr F)$; Atiyah-Singer prove that the pushforward of this symbol is nothing but the analytic index of the operator $D$, and this can be generalized to families of operators \cite{AS-IV}.
	\end{CONSTR}

	Now, let us consider the case that $D = \text\dh_{M, \fk s}$ is a linear Dirac operator on $M^{2m}$, associated to a ${\rm Spin}^c$-structure $\mathfrak s$, so that the symbol $T^*M \to \scr H{\it om}(\scr E, \scr F)$ is given precisely by Clifford multiplication. In particular, $\sigma(D) \in {\it KU}_c^0(T^*M)$ itself comes from ${\it KU}^0(M)$ via a ${\rm Spin}^c$-pushforward construction, so the index of $D$ is depicted by going up-and-right from $\underline{\bb C} \in {\it KU}^0(M)$ in the following square of pushforward maps:
	\begin{equation*}
		\begin{tikzcd}
			{\it KU}_c^0(T^*M^{2m}) \rar & {\it KU}_c^0(T^*\bb R^{2m+2k}) \\
			{\it KU}^0(M^{2m}) \uar\rar & {\it KU}_c^0(\bb R^{2m+2k}) \uar
		\end{tikzcd}
	\end{equation*}
	The right-and-up path is given simply by {\sc Constr.~\ref{constr:pushfwd}}, and the two agree provided that the two ${\rm Spin}^c$-structures on
	\begin{equation}\label{eq:twospinc}
		(\scr N_M \otimes \bb C) \oplus T^*M \cong \scr N_M \oplus (\bb R^{2k})^*
	\end{equation}
	agree. Indeed, one may pick the ${\rm Spin}^c$-structure on $\scr N_M$ so that this equality holds. This can be done parametrically as well; we believe it can be done even in $\it KR$ instead of $\it KU$, provided that a $\it KR$-adaptation of the Atiyah-Singer index theorem holds.

	With this background, let us begin investigating our claimed corollary of {\sc Conj.~\ref{conj:abs}} about $\mathbb Z/2$-gradings in real monopole theory: the vanishing of $w_1(\theta) \in H^1(\rspinc(Y))$ is equivalent to the claim that $\<w_1(\theta), \gamma\> = 0$ for all $\gamma \in \pi_1 \rspinc(Y)$. We may phrase this condition by pulling back to $S^1 \overset\gamma\to \rspinc(Y)$ and requiring that the composite
	\begin{equation}\label{eq:pushw1}
		{\it KR}^{0,0}(Y \times S^1) \overset{p_!}\longrightarrow {\it KO}^1(S^1) \overset{\cdot \eta}\longrightarrow {\it KO}^0(S^1) \overset{w_1}\longrightarrow \mathbb Z/2 
	\end{equation}
	sends $[\underline{\mathbb C}]$ to zero. On the other hand, there is a commutative square
	\[ \begin{tikzcd}
		{\it KO}^1(S^1) \rar["\cdot \eta"]\dar["\int_{S^1}"] & {\it KO}^0(S^1) \dar["w_1"] \\
		{\it KO}^0({\rm pt}) \rar["\rm mod~2"] & \mathbb Z/2.
	\end{tikzcd} \]
	Indeed, the right-down composite is given by
	\[ \mathbb Z \overset{\left[\begin{smallmatrix}0 \\ 1\end{smallmatrix}\right]}\to \mathbb Z \oplus \mathbb Z/2 \overset{\left[\begin{smallmatrix}0 & 1\end{smallmatrix}\right]}\longrightarrow \mathbb Z/2 \]
	while the down-right composite is given by
	\[ \mathbb Z \overset\sim\to \mathbb Z \overset{\rm mod~2}\twoheadrightarrow \mathbb Z/2. \]
	Consequently, the composite (\ref{eq:pushw1}) can be rewritten as a double pushforward
	\[ {\it KR}^{0,0}(Y \times S^1) \overset{p_!}\longrightarrow {\it KO}^{1}(S^1) \overset{\int_{S^1}}\longrightarrow {\it KO}^{0}({\rm pt}) \cong \mathbb Z \overset{\rm mod~2}\twoheadrightarrow \mathbb Z/2, \]
	followed by reduction mod 2. It is not difficult to see that a double pushforward amounts to one single pushforward over the whole composite, so this is given by the mod-2 reduction of the ``integral in $K$-theory''
	\[ \int_{Y \times S^1} \underline {\bb C}. \]
	As discussed above, a suitable $\it KR$-version of the Atiyah-Singer theorem should imply that this integral is nothing but the numerical index of the real Dirac operator $\text\dh^R_{Y^3 \times S^1, \underline{\mathfrak s}}$ associated to the real ${\rm Spin}^c$-structure on $Y \times S^1$ given by the 1-parameter family of such structures on $Y$. Now, the vanishing of the reduction mod 2 of this index for all elements $\gamma \in \pi_1 \rspinc(Y)$ would imply that there is a well-defined relative $\mathbb Z/2$-grading for $\it HMR$: indeed, one can use a linear gluing argument at both infinite ends for the monopole Floer equation on two different cylinders $Y \times \mathbb R$ to show that the difference in index is given by the index of a Dirac operator on $Y \times S^1$, hence always even if the obstruction vanishes.

	Now, let us turn to the second Stiefel-Whitney class of the conjectural polarization class. Its vanishing can be detected equivalently on all tori $S^1 \times S^1 \to \rspinc$ induced by pairs of basis elements $\gamma, \gamma' \in \pi_1 \rspinc$. In a manner similar to (\ref{eq:pushw1}), we have the composite
	\begin{equation}\label{eq:pushw2}
		{\it KR}^{0,0}(Y \times S^1 \times S^1) \overset{p_!}\longrightarrow {\it KO}^1(S^1 \times S^1) \overset{\cdot \eta}\longrightarrow {\it KO}^0(S^1 \times S^1) \overset{w_2}\longrightarrow \mathbb Z/2,
	\end{equation}
	and the composite of the last two maps can be rewritten thanks to a commutative square
	\[ \begin{tikzcd}
		{\it KO}^1(S^1 \times S^1) \rar["\cdot \eta"]\dar["\int_{S^1 \times S^1}"] & {\it KO}^0(S^1 \times S^1) \dar["w_2"] \\
		{\it KO}^{-1}({\rm pt}) \rar["\sim"] & \mathbb Z/2.
	\end{tikzcd} \]
	Indeed, we have ${\it KO}^1(S^1 \times S^1) \cong \mathbb Z\<x, y\> \oplus \mathbb Z/2\<\eta\>$, and ${\it KO}^0(S^1 \times S^1) \cong \mathbb Z \oplus \mathbb Z/2\<\eta x, \eta y\> \oplus \mathbb Z/2\<\eta^2\>$, with multiplication by $\eta$ given by the obvious map, and the $w_2$-functional is the one that picks out the coefficient of $\eta^2$. Likewise, the vertical integration map picks out the coefficient of $\eta$, hence proving the commutativity. Thus, we can once again express (\ref{eq:pushw2}) as a double pushforward, first integrating the $Y$ factor, then the $S^1 \times S^1$ factor. This can be computed instead by integrating the $Y \times S^1$ factor, and then the remaining $S^1$-factor. In other words, it is given by integrating the \emph{parametric index} ${\rm Ind}(\text\dh^R_{Y \times S^1}) \in {\it KO}^0(S^1)$ over $S^1$. The integration map $\int_{S^1} : {\it KO}^0(S^1) \to {\it KO}^{-1}({\rm pt})$ is nothing but $w_1$ of this 1-parameter family of indices. Thus, we conclude that the image of $[\underline{\bb C}]$ under (\ref{eq:pushw2}) vanishes if and only if this 1-parameter family of indices has trivial determinant. A more useful reformulation of this is that an orientation of the determinant line for any particular Dirac operator on $Y \times S^1$ uniquely determines an orientation for the whole homotopy class of such operators. Now, we may fix an arbitrary orientation on a basis of $\pi_1 \rspinc(Y)$, and we believe one should be able to extend it by linearity, via the concatenation isomorphism relating the determinant lines of two loops and their concatenation.
	
	Finally, if {\sc Conj.~\ref{conj:abs}} were true for $\widehat{\it HFR}$, it should recover the obstructions we have obtained in our main result {\sc Thm.~\ref{thm:main}}; this appears tricky to prove algebraically, and it seems to require some sort of Stiefel-Whitney analogue of the Grothendieck-Riemann-Roch theorem. The additional difficulty comes from the fact that there is no such thing as a ``Stiefel-Whitney character'', unlike the classical Chern character. We do not pursue this idea further in the present work, due to its apparent difficulty.


  \bibliographystyle{abbrvurl} 
  \bibliography{main}{} 

\end{document}